\documentclass[12pt]{article}
\usepackage[utf8]{inputenc}
\usepackage[english]{babel}
\usepackage[T1]{fontenc}
\usepackage{epigraph}
\usepackage{mathrsfs}
\usepackage{subcaption}
\usepackage{enumitem}
\usepackage{caption}
\usepackage{amsmath,latexsym,amsthm,mathtools}
\usepackage{amssymb}
\usepackage{mathabx}
\usepackage{float}

\usepackage[bookmarks,bookmarksnumbered,pdfstartview=FitBH,backref=page]{hyperref}
\hypersetup{
	colorlinks=true, 
	linkcolor=blue,
	citecolor= blue,
	linktoc= all, 
}

\numberwithin{equation}{section}

\newtheorem{proposition}[equation]{Proposition}
\newtheorem{theorem}[equation]{Theorem}
\newtheorem{corollary}[equation]{Corollary}
\newtheorem{lemma}[equation]{Lemma}

\newtheorem*{claim}{Claim}
\newenvironment{cproof}{\begin{proof}[Proof of the
        claim]}{\end{proof}}
\newtheorem{theoremintro}{Theorem}

\theoremstyle{definition}
\newtheorem{definition}[equation]{Definition}
\newtheorem{notation}[equation]{Notation}
\newtheorem{remark}[equation]{Remark}
\newtheorem*{remark*}{Remark}
\newtheorem{example}[equation]{Example}

\def\defin#1{\textbf{#1}} 
\DeclareMathOperator{\dom}{\mathrm{dom}}
\DeclareMathOperator{\rng}{\mathrm{rng}}

\newcommand{\N}{\mathbb N}
\newcommand{\Z}{\mathbb Z}
\newcommand{\inv}{^{-1}}
\newcommand{\la}{\left\langle}
\newcommand{\ra}{\right\rangle}

\newcommand{\abs}[1]{\left\lvert #1\right\rvert}
\DeclareMathOperator{\ord}{ord}

\newcommand{\Gc}{\mathcal{G}}

\newcommand{\Hc}{\mathcal{H}}

\newcommand{\Jc}{\mathcal{J}}
\newcommand{\Nc}{\mathcal{N}}

\newcommand{\Cc}{\mathcal{C}}
\newcommand{\Vc}{\mathcal{V}}

\newcommand{\PK}{\mathcal{K}}
\newcommand{\QQ}{\mathcal{Q}}

\def\id{\mathrm{id}}

\newcommand\BSo{\mathrm{BS}}
\newcommand\BSe{\mathbf{BS}}
\newcommand\Phe{\mathrm{Ph}}
\newcommand\PHE{\mathbf{Ph}}
\newcommand\primes{\mathcal P}

\newcommand\degout{\deg_{\mathrm{out}}}
\newcommand\degin{\deg_{\mathrm{in}}}

\newcommand{\Stab}{\mathrm{Stab}}
\newcommand{\Sub}{\mathrm{Sub}}
\newcommand{\Sch}{\mathbf{Sch}}
\newcommand\Schreier{\Sch}
\newcommand\Cayley{\mathbf{Cay}}

\newcommand\Tree{\mathcal{T}}

\def\FF{\mathbf F}
\newcommand{\acts}{\curvearrowright}
\newcommand{\action}{\acts}
\newcommand{\reacts}{\curvearrowleft}

\newcommand{\source}{\mathtt{s}}
\newcommand{\target}{\mathtt{t}}

\newcommand{\bs}{\backslash}

\usepackage{xcolor}

\usepackage[draft,inline,nomargin,index]{fixme}

\fxsetup{theme=color,mode=multiuser}
\FXRegisterAuthor{ac}{aac}{\color{purple}AC}
\FXRegisterAuthor{dg}{adg}{\color{orange}DG}
\FXRegisterAuthor{flm}{aflm}{\color{blue}FLM}
\FXRegisterAuthor{ys}{ays}{\color{red}YS}

\usepackage{tikz}
\usetikzlibrary{arrows}
\usetikzlibrary{cd}

\newcommand\rk{\mathrm{rk}}

\usepackage{enumitem}
\setlist{nosep}

\newcommand{\maxclo}{\mathcal{MC}}
\newcommand\RKCB{\rk_{\textrm CB}}

\author{
	Alessandro Carderi, Damien Gaboriau, \\
	François Le Maître and Yves Stalder
}

\title{On the space of subgroups of Baumslag-Solitar groups I: perfect kernel and phenotype}
\date{}

\begin{document}

	\maketitle	
	\begin{abstract}
		Given a Baumslag-Solitar group, we study its space of subgroups from a topological and dynamical perspective. We first determine its perfect kernel (the largest closed subset without isolated points).  We then bring to light a natural partition of the space of subgroups into one closed subset and countably many open subsets that are invariant under the action by conjugation. One of our main results is that the restriction of the action to each piece is topologically transitive.
		This partition is described by an arithmetically defined function, that we call the phenotype,
		with values in the positive integers or infinity.
		We eventually study the closure of each open piece and also the closure of their union.
		We moreover identify in each phenotype a (the) maximal compact invariant subspace.
	\end{abstract}

	
	\smallskip
	
	{
		\small	
		\noindent\textbf{{MSC-classification:}}	
		20E06; 20E08; 20F65; 37B05.
	}
	\tableofcontents

	\section{Introduction and presentation of the results}
	
	The Baumslag-Solitar group of non-zero integer parameters $m$ and $n$ is defined by the presentation 
	\begin{equation}\label{presentation BS(m,n)}
		\BSo(m,n)\coloneqq \la b,t\vert tb^{m}t^{-1}=b^{n}\ra.
	\end{equation}
	These one-relator two-generators groups were defined by Baumslag and Solitar \cite{baumslagTwogeneratorOnerelatorNonHopfian1962} to provide examples of groups with surprising properties, depending on the arithmetic properties of the parameters. 
	
	It results from the work of Baumslag and Solitar and of Meskin \cite{meskinNonresiduallyFiniteOnerelator1972} that the group $\BSo(m,n)$ is 
	\begin{itemize}
		\item residually finite if and only if $\abs{m}=1$ or $\abs{n}=1$ or $\abs{m}=\abs{n}$;
		\item Hopfian if and only if it is residually finite or $m$ and $n$ have the same set of prime divisors.
	\end{itemize}
	
	The group $\BSo(m,n)$ is amenable if and only if $\abs{m}=1$ or $\abs{n}=1$, and in this case, it is metabelian.	
	All Baumslag-Solitar groups however share weak forms of amenability: they are inner-amenable \cite{stalderMoyennabiliteInterieureExtensions2006} and a-T-menable \cite{galNewATmenableHNNextensions2003}.
	
	Over the years and despite the simplicity of their presentation, these groups have served as a standard source of examples and counter-examples, sometimes to published results (!). 
	They have been considered from countless different perspectives in group theory and beyond.

	Various aspects concerning the subgroups of the $\BSo(m,n)$ have been considered such as the growth functions of their number of subgroups of finite index with various properties, or such as a description of the kind of fundamental group of graphs of groups that can be embedded as subgroups in some $\BSo(m,n)$; see for instance \cite{gelmanSubgroupGrowthBaumslag2005,dudkinSubgroupsBaumslagSolitar2009,levittQuotientsSubgroupsBaumslag2015}.
	
	In this article, we consider global aspects of the space $\Sub(\BSo(m,n))$ of subgroups of the $\BSo(m,n)$ and of the topological dynamics generated by the natural action by conjugation.

	\subsection{The perfect kernel}
	Let $\Gamma$ be a countable group. We denote by $\Sub(\Gamma)$  the space of subgroups of $\Gamma$. If one identifies each subgroup with its indicator function, one can view the space $\Sub(\Gamma)$ as a closed subset of $\{0,1\}^\Gamma$. Thus $\Sub(\Gamma)$ is a compact, metrizable space by giving it the restriction of the product topology. See Section~\ref{sec: space subgroups} for the generalities about $\Sub(\Gamma)$.
	
	By the Cantor–Bendixson theorem, $\Sub(\Gamma)$ admits a unique decomposition as a disjoint union of a perfect set, called the \defin{perfect kernel} $\PK(\Gamma)$ of $\Gamma$, and of a countable open subset. 
	It is a challenging problem to determine the perfect kernel of a given countable group. 
	
	When $\Gamma$ is finitely generated, the finite index subgroups are isolated in $\Sub(\Gamma)$.
	It is thus relevant to introduce the subspace 
	$\Sub_{[\infty]}(\Gamma)$ consisting of all infinite index subgroups of $\Gamma$.
	It is a closed subspace of $\Sub(\Gamma)$ exactly when $\Gamma$ is finitely generated (see Remark~\ref{rmk: infinite index closed iff fg}).
	
	Our first main result completely describes the perfect kernel of the various Baumslag-Solitar groups.
	When $\abs{m} = \abs{n}$, the subgroup generated by $b^m$ is normal; let us denote by $\pi$ the corresponding quotient homomorphism 
	\begin{equation*}
		\BSo(m,n)\overset{\pi}{\to} \BSo(m,n)/\la b^m\ra.
	\end{equation*}
	We also denote by $\pi$ the map it induces between the spaces of subgroups of  $\BSo(m,n)$ and $\BSo(m,n)/\la b^m\ra$.
	\begin{theoremintro}[Perfect kernel of $\BSo(m,n)$, Theorem \ref{thm: explicit description K(BS(m,n))}]
		\label{thintro-PK(BS)}
		Let $m,n\in\Z\smallsetminus\{0\}$, 
		\begin{enumerate}
			\item \label{it: thintro m or n =1} if $\abs{m}=1$ or $\abs{n}=1$, then $\PK(\BSo(m,n))$ is empty;
			\item \label{it: thintro m,n large} if ${\abs{m}},\abs{n}>1$ , then 
			\begin{enumerate}
				\item  if $\abs{m} \neq \abs{n}$, then 
				$\PK(\BSo(m,n))  = \Sub_{[\infty]}(\BSo(m,n))$;		
				\item if $\abs{m} = \abs{n}$, then
				$\PK(\BSo(m,n)) = \pi\inv\big( \Sub_{[\infty]}(\BSo(m,n)/ \la b^m \ra) \big)$.
			\end{enumerate}		
		\end{enumerate}
	\end{theoremintro}
	
	The fact that $\Sub(\BSo(m,n))$ is countable when $\abs{m}=1$ or $\abs{n}=1$ (Item~\ref{it: thintro m or n =1}),  i.e.\ for the Baumslag-Solitar groups that are metabelian, was already observed by Becker, Lubotzky, and Thom \cite[Corollary 8.4]{beckerStabilityInvariantRandom2019}.
	Fortuitously or not, it turns out that $\PK(\BSo(m,n))  = \Sub_{[\infty]}(\BSo(m,n))$ exactly when $\BSo(m,n)$ is not residually finite.
	
	There is a general correspondence between the transitive pointed $\Gamma$-actions and the subgroups of $\Gamma$. It sends an action $\alpha$ to the stabilizer of its base point. This $\Gamma$-equivariant map is a bijection when one considers the actions up to pointed isomorphisms (see Section~\ref{sec: space subgroups}). 
	Item \ref{it: thintro m,n large} of Theorem~\ref{thintro-PK(BS)} has a unified reformulation in this setting: 
	\begin{itemize}
		\item[\textit{2'}.] \textit{if ${\abs{m}},\abs{n}>1$, then 
			$\PK(\BSo(m,n))$ is the space of subgroups $\Lambda$ such that the right $\BSo(m,n)$-action on $\Lambda\backslash\BSo(m,n)$ has infinitely many $\la b\ra$-orbits.}
	\end{itemize}   
	Note that this exactly means that the quotient of the $\Lambda$-action on the standard Bass-Serre tree (see Section~\ref{sect: Bass-Serre theory}) of $\BSo(m,n)$ is infinite. 
	
	Let us now give some more context for Theorem \ref{thintro-PK(BS)}.
	By Brouwer's characterization of Cantor spaces, the space $\PK(\Gamma)$ is either empty or a Cantor space.
	It is empty exactly when $\Sub(\Gamma)$ is countable. 
	This happens for example for groups all whose subgroups are finitely generated, also known as Noetherian groups.
	For instance all finitely generated nilpotent groups and more generally all polycyclic groups have a countable space of subgroups.
	
	On the opposite side, for the free group with a countably infinite number of generators, no subgroup is isolated, thus  $\PK(\FF_\infty)=\Sub(\FF_\infty)$ (see \cite[Proposition 2.1]{carderiDenseTotipotentFree2020}).
	
	There are some classical groups for which we know that $\PK(\Gamma)= \Sub_{[\infty]}(\Gamma)$.
	This is the case for the free groups $\FF_n$ (for $1<n<\infty$), see for instance \cite[Proposition 2.1]{carderiDenseTotipotentFree2020}. This is also the case for the groups with infinitely many ends, for the fundamental groups of the closed surfaces of genus~$\geq 2$, and for the finitely generated LERF groups with non-zero first $\ell^2$-Betti number (see \cite{azuelosSubgroupSpacesMaximal2023}). 
	Recall that a group $\Gamma$ is LERF when its set of finite index subgroups is dense in $\Sub(\Gamma)$ (see for instance \cite[Theorem 3.1]{glasnerIsolatedSubgroupsGeneric2016}).
	
	Bowen, Grigorchuk and Kravchenko established that the perfect kernel of the lamplighter group $(\Z/p\Z)^n \wr \Z= (\oplus_\Z (\Z/p\Z)^n) \rtimes \Z$ (for a prime number $p$) is exactly the space $\Sub(\oplus_\Z (\Z/p\Z)^n)$ of subgroups of the normal subgroup \cite[Theorem 1.1]{bowenInvariantRandomSubgroups2015a}. 
	Skipper and Wesolek uncovered the perfect kernel for a class of branch groups containing the Grigorchuk group and the Gupta–Sidki 3 group \cite{skipperCantorBendixsonRank2020}.

	The perfect kernel can be obtained by successively, and transfinitely, removing the isolated points,
	thus obtaining for every ordinal \(\alpha\) the \(\alpha\)-th Cantor-Bendixson derivative 
	$\Sub(\Gamma)^{(\alpha)} \coloneqq \Sub(\Gamma)^{(\beta)}\smallsetminus \{$isolated points of $\Sub(\Gamma)^{(\beta)}\}$ if $\alpha=\beta+1$, and
    $\Sub(\Gamma)^{(\alpha)} \coloneqq \bigcap_{\beta<\alpha} \Sub(\Gamma)^{(\beta)}$ if $\alpha$ is a limit ordinal.
	The \defin{Cantor–Bendixson rank} $\RKCB(\Gamma)$ of $\Gamma$ is the first ordinal $\zeta$ for which 
	the derived space $\Sub(\Gamma)^{(\zeta)}$ has no more isolated points, 
	and is thus equal to the perfect kernel 
	(see for instance \cite[Section 6.C]{kechrisClassicalDescriptiveSet1995} for details).
	When ${\abs m},\abs n>1$ and $\abs m\neq\abs n$, then Theorem \ref{thintro-PK(BS)} implies that $\RKCB(\BSo(m,n))=1$. 
	The determination of the Cantor-Bendixson ranks $\RKCB(\BSo(m,n))$ for the other cases is postponed to the sequel \cite{carderiSpaceSubgroupsBaumslagSolitar2023}.

	\subsection{Dynamical partition of the perfect kernel}
	
	The compact space of subgroups $\Sub(\Gamma)$ is equipped with the continuous action of $\Gamma$ 
	by conjugation: $ \Lambda\cdot\gamma \coloneqq \gamma\inv\Lambda\gamma$. 
	The perfect kernel is $\Gamma$-invariant.
	This action is of course not minimal in general, even when restricted to the perfect kernel: the latter may contain normal subgroups and these subgroups are fixed points! 
	However, the first three named authors observed a particularly nice feature in the case of the free group $\FF_n$ (for $1<n<\infty$):
	the action  $\PK(\FF_n)\curvearrowleft\FF_n$ is topologically transitive
	(which means that the space admits a dense G$_\delta$ subset of points whose individual orbits are dense). These $\FF_n$-actions are called totipotent, see \cite{carderiDenseTotipotentFree2020}.
	
	To our surprise, we uncovered a dramatically different and very rich situation for the Baumslag-Solitar groups. 
	
	\begin{theoremintro}\label{thintro: infinite partition of PK(BS(m,n))}
		Whenever ${\abs m},\abs n\neq 1$, the perfect kernel $\PK(\BSo(m,n))$ admits a countably infinite partition into $\BSo(m,n)$-invariant and topologically transitive subspaces.
		For the induced topology on $\PK(\BSo(m,n))$), one of the subspaces is closed; all the other ones are open.
	\end{theoremintro}
	
	Theorem \ref{thintro: infinite partition of PK(BS(m,n))} follows from Proposition \ref{prop: phenotype partition} and Theorem \ref{thm: gdelta dense orbits}. In a further work \cite{CGLMS-HT}, we show that topological transitivity can be upgraded to high topological transitivity.
	
	From now on in this introduction, we stick to the case $\abs{m}\neq 1$ and $\abs{n}\neq1$.
	In order to describe the partition in Theorem \ref{thintro: infinite partition of PK(BS(m,n))}, we introduce a new invariant: the \defin{phenotype}. 
	
	The relation $tb^mb\inv=b^n$ imposes some arithmetic conditions between the  cardinalities of the $b$-orbit of a point $x$ and the $b$-orbit of $xt$. 
	For instance,  the $b$-orbit of $x$ is infinite if and only if the $b$-orbit of $xt$ is infinite.
	
	In Definition~\ref{df: phenotype natural number}, we introduce a function  $\Phe_{m,n}\colon\Z_{\geq 1}\cup\{\infty\}\to \Z_{\geq 1}\cup\{\infty\}$ called the $(m,n)$-phenotype, with the following property, which directly follows from Proposition \ref{prop:phenotype is well defined for conn graph}, Theorem~\ref{thm: merging of (m,n)-graphs} and Proposition \ref{prop: realization of BS graph}:
	
	\begin{theoremintro}\label{thintro: char phenotype}
		Whenever ${\abs m},\abs n\neq 1$, there is a transitive $\BSo(m,n)$-action with two $b$-orbits of cardinal $k$ and $\ell$ respectively if and only if $\Phe_{m,n}(k)=\Phe_{m,n}(\ell)$.
	\end{theoremintro}
	
	If for instance $m$ and $n$ are coprime, 
	the phenotype $\Phe_{m,n}(k)$ of any natural number $k\in \Z_{\geq 1}$ is obtained as $k$ expunged of all its prime divisors that appear in either $m$ or $n$. The general form is more complicated, 
	see Definition~\ref{df: phenotype natural number} and Example~\ref{ex: phenotype for non coprime},
	but it follows readily from Definition \ref{df: phenotype natural number} that $\Phe_{m,n}(q) = q$ for every $q\geq 1$ that is coprime with $m$ and $n$.
	Hence, the set of possible $(m,n)$-phenotypes
	\begin{equation*}
		\QQ_{m,n}\coloneqq \{\Phe_{m,n}(k)\colon k\in  \Z_{\geq 1}\}\cup\{\infty\}.
	\end{equation*}
	is always infinite.
	
	Theorem \ref{thintro: char phenotype} allows us to define the \defin{phenotype} of a transitive $\BSo(m,n)$-action as the common $(m,n)$-phenotype of the cardinalities of its  $b$-orbits.
	Then, we define, the \defin{phenotype} $\PHE(\Lambda)$ of a subgroup  $\Lambda\in \Sub(\BSo(m,n))$ as the phenotype of the (right) $\BSo(m,n)$-action
	 on the homogeneous space $\Lambda\bs \BSo(m,n)$.
	
	Notice that the $\BSo(m,n)$-actions on $\Lambda\bs \BSo(m,n)$ and $(g\inv \Lambda g)\bs \BSo(m,n)$ are isomorphic (both are transitive with some point stabilizer equal to $\Lambda$),
	so they have the same phenotype.
	Hence, the partition
	\begin{equation}\label{eqintro: partition Sub}
		\Sub(\BSo(m,n))=\bigsqcup_{q\in \QQ_{m,n}} \PHE\inv(q)
	\end{equation}
	is invariant under the $\BSo(m,n)$-action (recall this is the action by conjugation).
	Let us mention from Proposition \ref{prop: phenotype partition} that
	\begin{itemize}
		\item for each finite $q\in \QQ_{m,n}$, the piece $\PHE\inv(q)$ is open;
		it is also closed if and only if $\abs{m} = \abs{n}$;
		\item the piece $\PHE\inv(\infty)$ is closed but not open.
	\end{itemize}
	In particular, the function $\PHE\colon\Sub(\BSo(m,n))\to \Z_{\geq 1}\cup\{+\infty\}$ is Borel. It is continuous if and only if $\abs{m}=\abs{n}$.
	
	It now follows from Theorem \ref{thm: gdelta dense orbits} that
	the restriction of the partition \eqref{eqintro: partition Sub} to the perfect kernel
	\begin{equation}\label{eqintro: partition PK}
		\PK(\BSo(m,n))=\bigsqcup_{q\in \QQ_{m,n}} \PK_q(\BSo(m,n)),
	\end{equation}
	where $\PK_q(\BSo(m,n))\coloneqq \PK(\BSo(m,n))\cap \PHE^{-1}(q)$,
	satisfies all the conclusions of Theorem \ref{thintro: infinite partition of PK(BS(m,n))}.
	The pieces $\PK_q(\BSo(m,n))$ are indeed non-empty, see Remark \ref{rem: non-empty pieces}.

	\subsection{Approximations by subgroups of other phenotypes}

	We still stick to the case $\abs{m}\neq 1$ and $\abs{n}\neq1$.
	Since the only non-open piece in partition \eqref{eqintro: partition Sub} is $\PHE\inv(\infty)$,
	the subgroups of infinite phenotype are the only ones which can be approximated in $\Sub(\BSo(m,n))$ by subgroups of other (that is, finite) phenotypes. 
	
	The set of limits of subgroups of finite phenotype depends on whether we fix the phenotype or we let it vary. 
	About approximations by subgroups with a constant phenotype, we have the following result (see Proposition \ref{prop: phenotype partition} and Theorem \ref{thm: closure of phen q}).
	
	\begin{theoremintro}\label{thmintro: limit fixed phenotype}
		Assume ${\abs{m}}, \abs{n} \neq 1$ and let us fix a finite $(m,n)$-phenotype $q$.
		\begin{enumerate}
			\item If $\abs{m} = \abs{n}$, then $\PHE\inv(q)$ is closed, hence no infinite phenotype subgroup can be approximated by subgroups of phenotype $q$.
			\item If $\abs{m} \neq \abs{n}$, then an infinite phenotype subgroup $\Lambda$ can be approximated by subgroups of phenotype $q$ if and only if $\Lambda \leq \la\!\la b \ra\!\ra$, where $\la\!\la b \ra\!\ra$ is the normal subgroup generated by $b$.
		\end{enumerate}
	\end{theoremintro}
	It is remarkable that the set $\overline{\PHE\inv(q)}\cap \PHE\inv(\infty)$ is independent of $q$ in the previous result.
	
	Allowing the finite phenotype to vary yields new limit points.
	Our result is the following (see Proposition~\ref{prop: limit finite phen m=n} and Corollary~\ref{cor:accum finite phen in infinite phen has empty interior}).
	
	\begin{theoremintro}\label{thmintro: limit varying phenotype}
		Assume ${\abs{m}}, \abs{n} \neq 1$.
		\begin{enumerate}
			\item If  $\abs m=\abs n$ then every infinite phenotype subgroup is a limit of finite (and varying) phenotype subgroups.
			\item On the contrary,
			if $\abs m\neq \abs n$, then the set of subgroups in $\PHE^{-1}(\infty)$ which are limits of finite (and varying) phenotypes subgroups has empty interior in $\PHE^{-1}(\infty)$. 
		\end{enumerate}
	\end{theoremintro}
	
	Therefore, in the case $\abs m=\abs n$, all subgroups of infinite phenotype are limits of subgroups of finite phenotype, but none of them is a limit of subgroups of fixed finite phenotype.
	
	The case $\abs m\neq \abs n$ is more complex. We do not have a nice description of the limit set from the above theorem. We can show however that this limit set is strictly larger than 
	its fixed phenotype counterpart, see Proposition~\ref{prop: varying phenotypes vs fixed easy} and Theorem~\ref{thm: varying phenotypes vs fixed}.

	\subsection{Closures of orbits in finite phenotype}
	
	We still stick to the case $\abs{m}\neq 1$, $\abs{n}\neq1$, and assume moreover $\abs{m} \neq \abs{n}$.
	The previous subsection shows that for any finite phenotype $q$, we have
	\[ 
	\PHE\inv(q) \subsetneq \overline{\PHE\inv(q)} \subsetneq \PHE\inv(q) \cup \PHE\inv(\infty).
	\]
	Theorem \ref{thintro: infinite partition of PK(BS(m,n))} further yields that $\PHE\inv(q)$ 
    contains dense orbits.
    For such an orbit $\mathcal O$, one has $\overline{\mathcal O} = \overline{\PHE\inv(q)}$,
    thus $\overline{\mathcal O}$ intersects  $\PHE\inv(\infty)$.
	In fact, Theorem~\ref{thmintro: limit fixed phenotype} completely described $\overline{\mathcal O}$. 
	We now turn our attention to the orbits whose closure is contained in $\PHE\inv(q)$. 
	Quite remarkably, they form a compact set.

	\begin{theoremintro}[see Theorem~\ref{thm:maximalinvariantclosed}]\label{thmintro:maximalinvariantclosed}
		Suppose ${\abs m},\abs n\neq 1$ and $\abs m\neq\abs n$.
		For every finite phenotype $q$, there is a positive integer $s=s(q,m,n)$ such that the  subset 
		\[
		\maxclo_q\coloneqq \PHE^{-1}(q)\cap \left\lbrace\Lambda\in\Sub(\BSo(m,n))\colon \Lambda\geq\langle\!\langle b^s\rangle\!\rangle\right\rbrace
		\]
		is compact and contains all the invariant compact subsets of $\PHE^{-1}(q)$.
	\end{theoremintro}

	In particular every normal subgroup of phenotype $q$, and hence every finite index subgroup, contains $\la\!\la b^s\ra\!\ra$. Moreover, $\maxclo_q\cap \PK_q(\BSo(m,n))$ has empty interior in $\PK_q(\BSo(m,n))$ (Theorem~\ref{thm:maximalinvariantclosed}-\ref{item: maxclo is meager}).
	
	When $\gcd(m,n)=1$, the above theorem takes an easier form:
	$s=q$ and $\maxclo_q\cap \PK(\BSo(m,n))=\{\langle\!\langle b^q\rangle\!\rangle\}$. In particular, $\langle\!\langle b^q\rangle\!\rangle$ is the unique normal subgroup of phenotype $q$ and  infinite index, see Theorem~\ref{thm:maximalinvariantclosed}-\ref{item: gcp m n 1 maxclo small}. On the other hand, if $\gcd(m, n)\neq 1$, then the perfect kernel contains continuum many normal subgroups of phenotype $q$, see Theorem \ref{th: normal subgroups in finite phenotype}.
	
	\subsection{An example: the case of \texorpdfstring{$\BSo(2,3)$}{BS(2,3)}}
	
	Let us specialize our theorems to the case of $\BSo(2,3)$.
	An illustrative picture is given in Figure \ref{fig:flower}.

	\begin{figure}[ht]
		\includegraphics{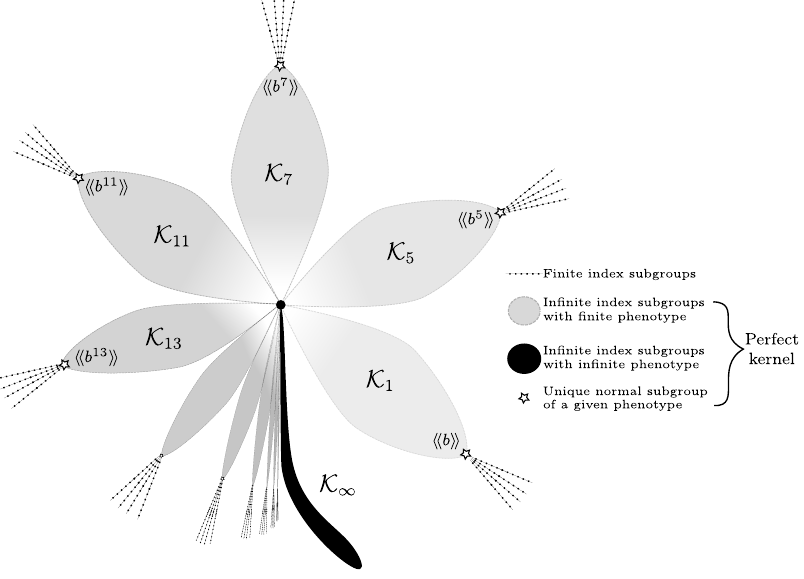}
		\caption{The space of subgroups of $\BSo(2,3)$}\label{fig:flower}
	\end{figure}

	Since $2\neq 3$, Theorem \ref{thintro-PK(BS)} tells us that $\PK(\BSo(2,3))=\Sub_{[\infty]}(\BSo(2,3))$.
	In this case the phenotype is given by the following simple formula
	\[
	\PHE(\Lambda)= \frac {I}{2^{\abs I_2}3^{\abs I_3}}, 
	\]
	where  $I$ is the index $I\coloneqq [ \la b\ra : \Lambda\cap\la b\ra ]$, and where $\abs{I}_p$ denotes the $p$-adic valuation of $I$ subject to the convention that 
	$\abs{\infty}_p=0$.
	
	Therefore, the possible phenotypes for the subgroups of $\BSo(2,3)$ are given by
	all the positive integers not divisible by 2 and 3, 
	and infinity. 
	Denoting
	$\PK_q=\{\Lambda\leq\BSo(2,3)\colon \PHE(\Lambda)=q\}$, the partition \eqref{eqintro: partition PK} becomes
	\[
	\PK(\BSo(2,3))=\PK_\infty\sqcup\bigsqcup_{q\colon \gcd(q,2)=\gcd(q,3)=1}\PK_q.
	\]
	By Theorem \ref{thintro: infinite partition of PK(BS(m,n))},
	the action on each $\PK_q$ is topologically transitive. 
	Note that all finite index subgroups have finite phenotype. The set $\PK_\infty$
	is closed and colored in black in Figure \ref{fig:flower}; the subsets $\PK_q$ are open
	and colored in gray in the figure. 
	Finally the finite index subgroups are denoted by the dotted lines. Note that there are infinitely many finite index subgroups and they accumulate on the sets $\PK_q$.

	Note that for every finite $q$, the set $\overline{\PK_q}\cap\PK_\infty$ is non-empty and independent
	of $q$; indeed by Theorem \ref{thmintro: limit fixed phenotype} this is the set of subgroups of infinite phenotype
	contained in $\la\!\la b\ra\!\ra$. This set is illustrated as the black central disk in the
	figure. As one can guess in the figure, $\overline{\cup_{q\ \text{finite}}\PK_q}\cap\PK_\infty$ is strictly bigger than this
	set, and yet not the entirety of $\PK_\infty$, as prescribed by Theorem \ref{thmintro: limit varying phenotype}.
	
	We finally apply Theorem \ref{thmintro:maximalinvariantclosed}. Since $\gcd(2,3)=1$, for every finite phenotype $q$ the largest compact invariant subset of $\PK_q$ consists only of one point: the unique normal subgroup contained in $\PK_q$, namely $\la\!\la b^q\ra\!\ra$, pictured with a small star in the figure. 
	Moreover, $\maxclo_q$ consists of the finite index subgroups of phenotype $q$ represented by the dotted lines emanating from the star together with 
	 the single accumulation point $\la\!\la b^q\ra\!\ra$ of $\maxclo_q$.

	\begin{remark*}
		Figure \ref{fig:flower} is actually quite general: as soon as $\abs m\neq \abs n$, we have the exact same picture except that the  possible phenotypes are different, and the stars turn into bigger compact maximal invariant subsets. Moreover, the phenotype is given by a more complicated formula.
	\end{remark*}
	
	\subsection{Some ideas on the techniques of proofs}

	The definition of the topology on the space of subgroups leads us to look at the restriction of transitive actions to some part of their Schreier graph and then on assembling such parts from different actions (to form new actions): this leads us to the notion of pre-action, as considered in \cite{fimaCharacterizationHighTransitivity2022},  where to facilitate the verification of the group relation, we impose that $b$ is defined everywhere, i.e. on the whole domain of the pre-action (see Section~\ref{subsec: preactions}). These pre-actions are more malleable but the algebraic conditions underlying them still make them difficult to manipulate. 

	This is why we further downgrade the data and 
	 move on to purely combinatorial objects associated with actions and pre-actions: the $(m,n)$-graphs (Section~\ref{sect: (m,n)-graphs}). These are oriented graphs which carry labels on the vertices and on the edges and which satisfy simple arithmetic conditions linking degrees and labels (Definition~ \ref{def:mngraph}, 
	equalities \eqref{eq:transfert} and inequalities \eqref{eq: deg bound by gcd L(v), n m}). They generalize the Bass-Serre graphs of pre-actions used in \cite{fimaCharacterizationHighTransitivity2022} by adding their labels which record the size of the orbits of $b$, $b^m$ or $b^n$ according to the graph element considered. Notice that in \cite{fimaCharacterizationHighTransitivity2022} the $b$-orbits were assumed to be infinite. 
	
	All the vertex labels of a connected $(m,n)$-graph have the same $(m,n)$-phenotype (Proposition~\ref{prop:phenotype is well defined for conn graph}) which is thus defined to be the phenotype of the graph (Definition~\ref{def: phenotype for m,n-graphs}).
	
	At this level, we can consider assembling together different parts (originating from different actions). Consider two connected $(m,n)$-graphs that are non-saturated (at least one of the inequalities \eqref{eq: deg bound by gcd L(v), n m} is strict), then they
	can appear as subgraphs of the same $(m,n)$-graph as soon as they have the same phenotype (Theorem~\ref{thm: merging of (m,n)-graphs}). This relies on  two basic constructions, the Welding Lemma~\ref{lem: self welding lemma} and the Connecting Theorem~\ref{thm: connecting same phenotype}.

	We then proceed by upgrading from $(m,n)$-graphs to pre-actions and actions
	(Proposition~\ref{Extending pre-action onto a (m,n)-graph}).
	These upgrades are not uniquely determined, however if an $(m,n)$-graph $\Gc_2$ contains the $(m,n)$-graph $\Gc_1$ of a pre-action $\alpha_1$, then the upgraded pre-action $\alpha_2$ can be chosen to extend $\alpha_1$ (Proposition~\ref{Extending pre-action onto a (m,n)-graph}).
	
	To summarize, we will use several times the same construction scheme: 
	Considering two actions, we restrict them to a large but proper part of their domain (pre-actions). We downgrade the resulting pre-actions to $(m,n)$-graphs and glue them together. We saturate the resulting $(m,n)$-graph and upgrade it into one action that "contains" the chosen parts of both original actions as sub-pre-actions (Theorem~\ref{thm: merging of pre-actions}).
	
	\subsection{Subsequent work}

	Since the first version of the present paper appeared, two preprints have enriched the picture as follows.

	On the one hand, the three last-named authors proved in \cite{CGLMS-HT}
	that the dynamics on the pieces $\PK_q$ is in fact \emph{highly topologically transitive}.
	They also studied the property of high transitivity for transitive actions of $\BSo(m,n)$: 
	they characterized the pieces 
	containing subgroups $\Lambda$ such that the action $\Lambda \backslash \BSo(m,n) \curvearrowleft \BSo(m,n)$
	is highly transitive and they established that this property is generic in these pieces.

	On the other hand, Sasha Bontemps has extended  Theorems~\ref{thintro-PK(BS)}, \ref{thintro: infinite partition of PK(BS(m,n))} and \ref{thintro: char phenotype} to generalized Baumslag-Solitar groups, where the right notion of phenotype is more subtle \cite{Bontemps-pkGBS}. 
	She also obtained high topological transitivity results generalizing Theorem~C from the aforementioned preprint \cite{CGLMS-HT}.
	
	\paragraph{Acknowledgements.} 
	We are very grateful to both referees for their work and their detailed remarks which helped us to improve the paper.
	A.~C.\ acknowledges funding by the Deutsche Forschungsgemeinschaft (DFG, German Research Foundation) – 281869850 (RTG 2229).
	D.~G.\ is supported by the CNRS  and partially supported by the LABEX MILYON (ANR-10-LABX-0070) of Université de Lyon, within the program “Investissements d’Avenir” (ANR-11-IDEX-0007) operated by the French National Research Agency (ANR).
	F.~L.M.\ acknowledges funding by the ANR projects ANR-17-CE40-0026 AGRUME and  ANR-19-CE40-0008 AODynG.

	\section{Preliminaries and notations}
	
	In this text, we denote by $\Z_{\geq 0}\coloneqq \{0,1,2,\ldots\}$ the set of non-negative integers and by $\Z_{\geq 1}\coloneqq \{1,2,3,\ldots\}$ the set of positive integers. Given two integers $k,l\in\Z\smallsetminus\{0\}$, we denote by $\gcd(k,l)\in\Z_{\geq 1}$ the \textit{greatest common divisor} of $k$ and $l$. We use the convention that $\gcd(k,\infty)=k$ and $\frac \infty k=k\infty=\infty$. 
	
	Let $\primes$ be the set of prime numbers. Given an integer $k\in\Z\smallsetminus\{0\}$ and a prime $p\in \primes$, we denote by $\abs{k}_p$ the $p$-\textit{adic valuation} of $k$, that is $\abs{k}_p$ is the largest positive integer such that $p^{\abs{k}_p}$ divides $k$. 	
	
	\subsection{Graphs and Schreier graphs}
	\label{subsect: Graphs and Schreier graphs}
	
	All our graphs are defined as in \cite{serreTrees1980}. 
	That is, a graph $\Gc$ is a couple $(V(\Gc),E(\Gc))$ where $V(\Gc)$ is the \defin{vertex set} and $E(\Gc)$ is the \defin{edge set}, endowed with: 
	\begin{itemize}
		\item two maps $\source,\target\colon E(\Gc)\rightarrow V(\Gc)$ called \defin{source} and \defin{target} respectively;
		\item a fixed-point-free involution $E(\Gc) \to E(\Gc), e \mapsto \bar e$;
	\end{itemize}
	such that $\source(\bar e) = \target(e)$ and $\target(\bar e) = \source (e)$.
	
	An \defin{orientation} of the graph $\Gc$ is a partition $E(\Gc) = E^+(\Gc) \sqcup E^{-}(\Gc)$ whose pieces are exchanged by the involution $e\mapsto \bar e$. Edges in $E^+(\Gc)$ are called \defin{positive} edges and edges in $E^-(\Gc)$ are \defin{negative}.
	
	\begin{remark}
		In order to define an oriented graph $\Gc$, it is enough to define the set of vertices $V(\Gc)$, the set of positive edges $E^+(\Gc)$, and the restrictions of the source and target maps $\source,\target$ to $E^+(\Gc)$.
		Indeed, we can define $E^-(\Gc)$ to be a copy of $E^+(\Gc)$ and the involution $e\mapsto \bar e$ to be the natural identification of $E^+(\Gc)$ with $E^-(\Gc)$. We extend the source and target map by setting $\source(\bar e)\coloneqq \target(e)$ and $\target(\bar e)\coloneqq \source(e)$.
	\end{remark}
	
	The \defin{degree} of a vertex $v$ in a graph $\Gc$, is the cardinal  \[\deg(v)\coloneqq \lvert \{e\in E(\Gc): \source(e) = v\}\rvert=\lvert \{e\in E(\Gc): \target(e) = v\}\rvert.\]
	If $\Gc$ is oriented, 
	we say that an edge $e$ is:
	\begin{itemize}
		\item a \defin{$v$-outgoing} edge if it is positive and $\source(e)=v$;
		\item a \defin{$v$-incoming} edge if it is positive and $\target(e)=v$.
	\end{itemize}
	The \defin{outgoing degree} $\degout (v)$ of $v$ is the number of $v$-outgoing edges while its \defin{incoming degree} $\degin (v)$ is the number of $v$-incoming edges. We clearly have $\degout(v)+\degin(v) = \deg(v)$.
	
	A \defin{subgraph} $\Gc'$ of a graph $\Gc$ is a graph such that $V(\Gc')\subseteq V(\Gc)$, $E(\Gc')\subseteq E(\Gc)$ and the structural maps of $\Gc'$ are restrictions of those of $\Gc$.
	
	Still following \cite{serreTrees1980}, we call \defin{circuit} a subgraph isomorphic to a circular graph (of length $l\geq 1$) and \defin{loop} a circuit of length $1$. 
	The edge of a loop is also called a loop.
	
	A \textbf{path} in a graph $\Gc$ is a finite sequence of 
	edges $(e_1,\dots,e_n)$, such that for all $1\leq k\leq n-1$, 
	$\target(e_k)=\source(e_{k+1})$. 
	Similarly, an \textbf{infinite path} is a sequence of edges $(e_k)_{k\geq 1}$ such that $\target(e_k)=\source(e_{k+1})$ for all $k\geq 1$. Finally a (possibly infinite) path is called \textbf{simple} when the induced sequence of vertices is injective.
	
	The \defin{ball} $B(v,R)$ of radius $R$ centered at a vertex $v$ in a graph $\Gc$ is the subgraph induced by the set of vertices of $\Gc$ at distance $\leq R$ from $v$ in the path metric.  
	
	\paragraph{Schreier graphs.} Let $\Gamma$ be a group and let $S$ be a generating set of $\Gamma$. Consider a (right) action $\alpha\colon X \curvearrowleft \Gamma$. The \defin{Schreier graph} of $\alpha$ relatively to $S$ is the oriented graph 
	$\Schreier(\alpha)=\Schreier(\alpha, S)$ defined by 
	\[	
	V(\Schreier(\alpha))\coloneqq X \text{ and } E^+(\Schreier(\alpha))\coloneqq\{(x, s)\colon x\in X, s\in S\} 
	\]
	where $\source(x,s)=x$ and $\target(x,s)=xs$, together with the following labeling:
	the edge $(x,s)$ is labeled $s$ and its opposite $\overline{(x,s)}$ is labeled by $s^{-1}$.
	
	Given a subgroup $\Lambda\leq \Gamma$, we denote by $\Schreier(\Lambda,S)$ the Schreier graph of the natural action $\Lambda\backslash \Gamma\curvearrowleft \Gamma$. 
	
	The \defin{Cayley graph} of $\Gamma$ relatively to $S$ is the Schreier graph $\Schreier(\alpha, S)$ of the action $\alpha\colon \Gamma \curvearrowleft \Gamma$  by (right) translations. 
	This graph is denoted by $\Cayley(\Gamma, S)$ and clearly $\Cayley(\Gamma,S)=\Schreier(\{\id\},S)$.
	The $\Gamma$-action by left translations extends to the standard left action of $\Gamma$ on $\Cayley(\Gamma, S)$ by graph automorphisms 
	\footnote{This is why Schreier graphs were defined with respect to right actions.}. In particular $\Lambda\backslash\Cayley(\Gamma,S)=\Schreier(\Lambda,S)$.
	
	Let $\varphi\colon X\to Y$ be a $\Gamma$-equivariant map from $\alpha\colon X\curvearrowleft\Gamma$
	to $\beta\colon Y \curvearrowleft\Gamma$ and let $S$ be a generating set of $\Gamma$.
	The map $\varphi$ extends to a graph morphism from $\Schreier(\alpha, S)$ to $\Schreier(\beta, S)$ which respects the labelings. 
	In particular, given subgroups $\Lambda_1 \leq \Lambda_2\leq \Gamma$, the equivariant map $\Lambda_1\backslash \Gamma \to \Lambda_2\backslash \Gamma$ defines a surjective morphism $\Schreier(\Lambda_1, S) \to \Schreier(\Lambda_2, S)$.

	\subsection{Space of subgroups}\label{sec: space subgroups}
	
	Let $\Gamma$ be a countable group. We identify its set of subsets with $\{0,1\}^\Gamma$ and we endow it with the product topology, thus turning it into a Polish compact space.
	The \textbf{space of subgroups} of $\Gamma$ is the closed, hence compact Polish, subspace 
	\[
	\Sub(\Gamma) \coloneqq \{ \Lambda \in \{0,1\}^\Gamma\colon \Lambda \text{ is a subgroup} \},
	\]
	which is also totally disconnected. The clopen subsets 
	\[
	\Vc(I,O) \coloneqq \{ \Lambda \in \Sub(\Gamma)\colon I \subseteq \Lambda \text{ and } \Lambda\cap O = \emptyset \}
	\]
	of $\Sub(\Gamma)$ where $I, O$ run over finite subsets of $\Gamma$, form a basis of the topology. Note that a sequence $(\Lambda_n)_{n\geq 0}$ of subgroups converges to $\Lambda$ if and only if for all $\gamma\in\Gamma$,
	\[
	\left(\gamma\in \Lambda\right) \iff \left(\gamma \in \Lambda_i \text{ for } i \text{ large enough}\right).
	\]
	
	By the Cantor-Bendixson Theorem \cite{cantorUeberUnendlicheLineare1884, bendixsonQuelquesTheoremesTheorie1883} (see e.g.\ \cite[Thm.\ 6.4]{kechrisClassicalDescriptiveSet1995}), there is a unique decomposition
	\[
	\Sub(\Gamma) = \Cc(\Gamma) \sqcup \PK(\Gamma)
	\]
	where $\Cc(\Gamma)$ is a countable open subset and $\PK(\Gamma)$ is a closed perfect\footnote{A topological space is called \textbf{perfect} if it has no isolated points.} subspace called the \textbf{perfect kernel} of $\Gamma$.
	The set $\PK(\Gamma)$ is the largest subset $\PK\subseteq \Sub(\Gamma)$ without isolated points for the induced topology.
	In fact, $\PK(\Gamma)$ is exactly the set of \textbf{condensation points}, that is, the points whose neighborhoods in $\Sub(\Gamma)$ are all uncountable.
	
	\begin{remark}
		By a theorem of Brouwer, the space $\PK(\Gamma)$ is either empty or a Cantor space, see \cite[Thm.~7.4]{kechrisClassicalDescriptiveSet1995}.
	\end{remark}

	\begin{remark}\label{rmk: infinite index closed iff fg}
		The subset $\Sub_{[\infty]}(\Gamma)$ of infinite index subgroups of $\Gamma$ is closed in $\Gamma$ if and only if $\Gamma$ is finitely generated.    
		Indeed if $\Gamma$ is finitely generated, then its finite index subgroups are isolated.
		If $\Gamma$ is not finitely generated, its finite index subgroups are not finitely generated, but they are limit points of finitely generated (thus of infinite index)
		subgroups; so $\Sub_{[\infty]}(\Gamma)$ is dense in $\Sub(\Gamma)$.
	\end{remark}

	The group $\Gamma$ acts (on the right) by conjugation via $\Lambda\cdot\gamma \coloneqq \gamma\inv\Lambda\gamma$ on the space of its subgroups $\Sub(\Gamma)$.
	This action is continuous and the Cantor-Bendixson decomposition $\Sub(\Gamma)=\Cc(\Gamma)\sqcup \PK(\Gamma)$ is $\Gamma$-invariant.

	By the Baire category theorem, any countable closed subset of $\Sub(\Gamma)$ contains an isolated point, so $\Sub(\Gamma)$ has trivial perfect kernel if and only if it is countable. The following well-known proposition is useful for showing the latter property.
	\begin{proposition}
		\label{prop:Noether-by-subcountable}
		Let $\Gamma$ be a countable group, let $N$ be a normal subgroup of $\Gamma$ such that $\Gamma/N$ is Noetherian (all its subgroups are finitely generated), and assume that $\Sub(N)$ is countable. Then $\Sub(\Gamma)$ is countable. 
	\end{proposition}
	\begin{proof}
		Let $\Lambda\leq \Gamma$ and denote by $\pi\colon \Gamma\to \Gamma/N$ the quotient map.
		Since $\Gamma/N$ is Noetherian, we have $\pi(\Lambda)=\la S\ra$ for some finite set $S$. Fix a finite set $S'\subseteq \Lambda$ such that $\pi(S')=S$. Then we can recover $\Lambda$ from $S'$ and its intersection with $N$ as
		\[
		\Lambda = \la S'\cup (\Lambda\cap N)\ra.
		\]
		In other words, the map $(S',N')\mapsto\la S'\cup N'\ra$  surjects $\mathcal P_f(\Gamma)\times\Sub(N)$ onto $\Sub(\Gamma)$, where $\mathcal P_f(\Gamma)$ is the set of finite subsets of $\Gamma$, which is countable. Since $\Sub(N)$ is countable as well we conclude that $\Sub(\Gamma)$ is countable.
	\end{proof}
	\begin{corollary}
		If $\abs{m}=1$ or $\abs n=1$ then $\Sub(\BSo(m,n))$ is countable.
	\end{corollary}
	\begin{proof}[Sketch of proof]
		We sketch the proof contained in \cite[Cor.~8.4]{beckerStabilityInvariantRandom2019}.
		By symmetry we may as well assume $m=1$. 
		Then $\BSo(m,n)$ is isomorphic to the semi-direct product $\Z[1/n]\rtimes \Z$ where $\Z$ acts by multiplication by $n$.
		As explained in the proof of \cite[Cor.~8.4]{beckerStabilityInvariantRandom2019}, $\Sub(\Z[1/n])$ is countable, so the result follows from the previous proposition.
	\end{proof}
	
	\paragraph{Space of pointed actions.}
	Let us now interpret the topological space $\Sub(\Gamma)$ in terms of pointed transitive group actions and their pointed Schreier graphs.
	To any pointed transitive group action $(\alpha,v)$, where $ \alpha:V \curvearrowleft \Gamma$ and $v\in V$, we associate the stabilizer $\Stab_\alpha(v) \in \Sub(\Gamma)$, 
	and we notice that $\Stab_{\alpha_1}(v_1) = \Stab_{\alpha_2}(v_2)$ if and only if $(\alpha_1,v_1)$ and $(\alpha_2,v_2)$ are isomorphic as pointed transitive actions. 
	\begin{notation} We denote by $[\alpha,v]$ the isomorphism class of any pointed transitive action $(\alpha,v)$.
	\end{notation}
	We therefore have a canonical bijection $[\alpha,v] \mapsto \Stab_\alpha(v)$ between the collection of isomorphism classes of pointed transitive actions and $\Sub(\Gamma)$. 
	Its inverse is given by $\Lambda \mapsto [\Lambda \bs \Gamma \reacts \Gamma, \Lambda]$. 
	Through this bijection, the action by conjugation of $\Gamma$ on $\Sub(\Gamma)$ becomes $[\alpha,v]\cdot \gamma= [\alpha,v\alpha(\gamma)]$, i.e., it moves the basepoint.
	
	Via the above identification, we obtain a topology on the set of  isomorphism classes of pointed actions $[\alpha,v]$. 
	
	It is clear that two pointed actions are isomorphic if and only if their Schreier graphs are isomorphic as pointed labeled graphs.
	Given two pointed labeled oriented graphs $(\Gc,v), (\Hc,w)$ and a positive integer $R$, we write $(\Gc,v) \simeq_R (\Hc,w)$ to mean that the $R$-balls around $v$ in $\Gc$ and around $w$ in $\Hc$ are isomorphic as pointed oriented labeled graphs.
	It is an exercise to check that if $\Gamma$ is generated by a finite set $S$, then the sets of the form
	\begin{equation}\label{eq: nbhd basis in terms of Schreier graph}
		\Nc([\alpha,v], R) \coloneqq
		\big\{ 
		[\alpha',v']\colon (\Schreier(\alpha,S),v) \simeq_R (\Schreier(\alpha',S),v')
		\big\},
	\end{equation}
	constitute a basis of clopen neighborhoods of $[\alpha,v]$.
	
	\subsection{Bass-Serre theory}
	\label{sect: Bass-Serre theory}
	
	Associated with the standard HNN-presentation of \[\BSo(m,n)=\la b,t\vert tb^{m}t^{-1}=b^{n}\ra,\] we have the $\BSo(m,n)$-action on its Bass-Serre tree $\Tree$. 
	Recall that $\Tree$ is the oriented tree with $V(\Tree)=\BSo(m,n)/\la b\ra$, $E^+(\Tree)=\BSo(m,n)/\la b^n\ra$,
	\[
	\source(\gamma\la b^n\ra)=\gamma\la b\ra,\text{ and }\target(\gamma\la b^n\ra)=\gamma t\la b\ra
	\]
	and given a subgroup $\Lambda\leq \BSo(m,n)$, the quotient $\Lambda\bs\Tree$ has the structure of a graph of groups whose fundamental group is $\Lambda$, see \cite{serreTrees1980}.
	
	\begin{remark}
		\label{rem: gpe fdmt si inter <b> trivial}
		Let $\Lambda\leq \BSo(m,n)$ be a subgroup. 
		If $\Lambda\cap \langle b\rangle = \{\id\}$, then $\Lambda$ acts freely on $\Tree$; thus it is the fundamental group of the quotient graph $\Lambda\bs \Tree$, hence $\Lambda$ is a free group.
	\end{remark}
	
	Let us now concentrate on a subgroup $\Lambda\leq \BSo(m,n)$ such that $\Lambda\cap \langle b\rangle\neq\{\id\}$. 
	Then for the induced action $\Lambda\action \Tree$, each edge and vertex stabilizer is infinite cyclic: the tree $\Tree$ is a GBS-tree (for Generalized Baumslag-Solitar), in the sense of \cite{foresterSplittingsGeneralizedBaumslag2006,levittAutomorphismGroupGeneralized2007}. 
	One can use this point of view to understand the graph of groups description of $\Lambda$.
	However, taking advantage of the transitivity of the $\BSo(m,n)$-action on the edges and the vertices, we provide a slightly more precise description.
	
	\begin{proposition}\label{Prop: Bass-Serre th. recovering the lambda pm from the quotient graph}
		Let $m$ and $n$ be non-zero integers.
		Let $\Lambda\leq \BSo(m,n)$ be a subgroup such that $\Lambda\cap \langle b\rangle\not\eq\{\id\}$.
		The quotient graph of groups arising from the action $\Lambda \curvearrowright \Tree$ 
		is isomorphic to the graph of groups obtained by attaching a copy of $\Z$ to every vertex and every edge of the quotient graph $\Lambda\bs \Tree$, with structural maps of positive edges
		\begin{eqnarray*}
			\Z_e \hookrightarrow \Z_{\source(e)}, &  & k \mapsto \frac{n}{\degout (\source(e))} \cdot k, \\
			\Z_e \hookrightarrow \Z_{\target(e)}, &  & k \mapsto \frac{m}{\degin (\target(e))} \cdot k.
		\end{eqnarray*}
	\end{proposition}
	
	\begin{proof}
		In this proof we set $\Gamma\coloneqq \BSo(m,n)$.
		Let us consider the action of $\Lambda$ on the tree $\Tree$. 
		Since $\Tree$ is locally finite, any edge adjacent to a vertex with
		infinite stabilizer has itself infinite stabilizer. 
		It follows that all vertex and edge $\Lambda$-stabilizers are infinite.
		Being subgroups of the $\Gamma$-stabilizers, they are all isomorphic to $\Z$.
		
		Observe that since $\Gamma$ acts transitively and the $\Gamma$-stabilizers are abelian, the $\Gamma$-stabilizers are canonically pairwise isomorphic: given any vertex $u \in V(\Tree)$ and $a\in \Stab_\Gamma(u)$, one has
		\begin{equation}
			\label{Eq: conj toutes pareilles 0}
			gag\inv = hah\inv
			\quad \text{ for any }
			g,h\in \Gamma \text{ such that } gu = hu.
		\end{equation}
		Indeed since $h\inv g\in\Stab_\Gamma(u)$, we get that $h^{-1}gag^{-1}h=a$.
		
		We now focus on the quotient graph of groups arising from the action $\Lambda \curvearrowright \Tree$.	
		Let us recall from \cite{serreTrees1980} that its vertex groups are $G_v\coloneqq \Stab_\Lambda(\tilde v)$ and edge groups are $G_e\coloneqq \Stab_\Lambda(\tilde e)$,
		where $\tilde v, \tilde e$ are some lifts of $v, e$ in $\Tree$.
		Given any $e\in E^+(\Lambda\bs \Tree)$, the structural map $G_e\hookrightarrow G_{\target(e)}$ is
		\begin{equation}
			\label{Eq: description des structure maps}
			\begin{array}{ccccc}
				G_e = \Stab_\Lambda(\tilde e) &\hookrightarrow & \Stab_\Lambda(\target(\tilde e)) &
				\to &\Stab_\Lambda\left(\widetilde{\target(e)}\right) = G_{\target(e)}\\
				a \ &\mapsto & a & \mapsto& \ gag\inv
			\end{array}
		\end{equation}
		where $g\in \Lambda$ is any element such that $g\cdot \target(\tilde e) = \widetilde{\target(e)}$ and
		the map $G_e\hookrightarrow G_{\source(e)}$ is similar.
		This is unambiguous by \eqref{Eq: conj toutes pareilles 0}. 
		
		Let us call \emph{orientation} of an infinite cyclic group the choice of one generator (over two). This provides an identification to $\Z$.
		Once every stabilizer is oriented,
		the inclusions $G_e\hookrightarrow G_{\source(e)}$ and $G_e\hookrightarrow G_{\target(e)}$ become multiplications by  non-zero integers  $\lambda_{\Lambda}^-(e)$ and $\lambda_{\Lambda}^+(e)$, respectively. 
		It now suffices to prove that, for well-chosen orientations, one has
		\begin{equation}
			\label{Eq: valeurs lambda +- 0}
			\lambda_\Lambda^-(e)
			= \frac{n}{\degout (\source(e))}\text{ and }
			\lambda_\Lambda^+(e) = \frac{m}{\degin (\target(e))}
		\end{equation}
		for every positive edge $e\in E^+(\Lambda\bs\Tree)$.
		
		Let us first observe that the absolute value of $\lambda_\Lambda^{\pm}(e)$ does not depend on the orientations: it is equal to $[G_v:G_e]$. In other words, if $\tilde e$ is a lift of $e$,  $\tilde v\coloneqq \source(\tilde e)=$ and $\tilde w\coloneqq\target(\tilde e)$, we have
		\begin{align}
			\label{lambda - et orbite 0}
			\abs{\lambda^-_{\Lambda}(e)} &= [\Stab_\Lambda(\tilde v):\Stab_\Lambda(\tilde e)] 
			= \abs{\Stab_\Lambda(\tilde v) \cdot \tilde e} \\
			\label{lambda + et orbite 0}
			\abs{\lambda^+_{\Lambda}(e)} &= [\Stab_\Lambda(\tilde w):\Stab_\Lambda(\tilde e)]
			= \abs{\Stab_\Lambda(\tilde w) \cdot \tilde e}.
		\end{align}
		
		Let $E_{\mathrm{out}}(\tilde v)$ be the set of $\tilde v$-outgoing edges. Its cardinal is $\abs{E_{\mathrm{out}}(\tilde v)}=\abs{n}$. 
		Any generator of $\Stab_\Gamma(\tilde v)$ acts as a single $\abs n$-cycle on $E_{\mathrm{out}}(\tilde v)$.
		Hence $E_{\mathrm{out}}(\tilde v)$ splits into $\Stab_\Lambda(\tilde v)$-orbits of equal size, that is $\abs{\lambda^-_{\Lambda}(e)}$ according to \eqref{lambda - et orbite 0}. 
		The number of these $\Stab_\Lambda(\tilde v)$-orbits is $\degout(v)$, thus $\abs{n} = \abs{\lambda^-_{\Lambda}(e)} \cdot \degout(v)$.
		We obtain similarly 
		$\abs{m} = \abs{\lambda^+_{\Lambda}(e)} \cdot \degin(w)$, using incoming edges and \eqref{lambda + et orbite 0}. 
		We have established that \eqref{Eq: valeurs lambda +- 0} holds in absolute value.
		
		Let us now turn to the signs in \eqref{Eq: valeurs lambda +- 0}, for which we need explicit orientations of the $\Lambda$-stabilizers.
		We actually start by orienting the $\Gamma$-stabilizers.
		
		Pick the vertex $\tilde u_0\coloneqq \la b\ra \in V(\Tree)$, then $\Stab_\Gamma(\tilde u_0)=\la b\ra$ and the edge $\tilde d_0\coloneqq\la b^n\ra\in E^+(\Tree)$ has source $\tilde u_0$ and target $t\tilde u_0$.
		Since the $\Gamma$-stabilizers are canonically pairwise identified by conjugation \eqref{Eq: conj toutes pareilles 0}, these choices induce a canonical conjugation-invariant orientation $x_*$ of all the vertex and edge $\Gamma$-stabilizers: $x_{g\tilde u_0} \coloneqq gbg^{-1}$ for $\Stab_\Gamma(g\tilde u_0)$ and $x_{g\tilde d_0} \coloneqq gb^ng^{-1}$ for $\Stab_\Gamma(g\tilde d_0)$.
		
		The inclusions $\Stab_\Gamma(\tilde e) \hookrightarrow \Stab_\Gamma(\source(\tilde e))$
		and $\Stab_\Gamma(\tilde e) \hookrightarrow \Stab_\Gamma(\target(\tilde e))$ become multiplications by non-zero integers that we denote by $\mu^-_{\Gamma}(\tilde e)$ and $\mu^+_{\Gamma}(\tilde e)$.
		We have $\mu^-_{\Gamma}(\tilde e) = n$ since $x_{\tilde e} = x_{\source(\tilde e)}^n$ and $\mu^+_{\Gamma}(\tilde e) = m$ since 
		\[
		x_{\tilde e} = g b^n g\inv = g (t b t\inv)^m g\inv = x_{\target(\tilde e)}^m.
		\]
		
		The $\Lambda$-stabilizers have finite index in the corresponding $\Gamma$-stabilizers.
		We orient them coherently with	the ambient $\Gamma$-stabilizers by using positive powers. 
		The $\Lambda$-conjugations between $\Lambda$-stabilizers remain orientation-preserving, therefore by \eqref{Eq: description des structure maps} the inclusion $\Stab_\Lambda(\tilde e) \hookrightarrow \Stab_\Lambda(\target(\tilde e))$ becomes the multiplication by  $\lambda^+_{\Lambda}(e)$. Similarly,
		the inclusion $\Stab_\Lambda(\tilde e) \hookrightarrow \Stab_\Lambda(\source(\tilde e))$
		becomes multiplication by $\lambda^-_{\Lambda}(e)$. Since the orientations
		are coherent, we conclude that $\lambda^-_{\Lambda}(e)$ has the same sign
		as $\mu_\Gamma^-(e)=n$ and $\lambda^+_{\Lambda}(e)$ has the same sign
		as $\mu_\Gamma^+(e)=m$.
	\end{proof}
	
	\begin{corollary}
		\label{cor: oriented graph determines isomorphis type}
		Let $m$ and $n$ be non-zero integers. 
		Let $\Lambda\leq \BSo(m,n)$ be a subgroup such that $\Lambda\cap \langle b\rangle\not\eq\{\id\}$.
		The isomorphism type of $\Lambda$ is completely determined by the oriented graph $\Lambda \bs \Tree$. \qed
	\end{corollary}

	\begin{proposition}
		Let $m$ and $n$ be non-zero integers and let $\Lambda\leq\BSo(m,n)$ be a subgroup.  
		\begin{enumerate}
			\item \label{item: finite phenotype and finite implies Z} If $\Lambda\cap \la b\ra\not\eq \{\id\}$, then either $\Lambda\simeq \Z$ is virtually a subgroup of $\langle b\rangle$ or $\Lambda$ is not a free group.
			\item  \label{item: fund group in amenable case}If $\vert m\vert =1$ or $\vert n\vert =1$, then the fundamental group of the underlying graph $\Lambda\bs \Tree$ is a free group of rank $\leq 1$.
		\end{enumerate}
	\end{proposition}
	If $\Lambda\cap \la b\ra= \{\id\}$,  then $\Lambda$ is the fundamental group of the underlying graph $\Lambda\bs \Tree$ (see Remark~\ref{rem: gpe fdmt si inter <b> trivial}).
		
	The first item of the proposition follows from standard techniques in $\ell^2$-cohomology: if $\Lambda\cap \la b\ra\not\eq \{\id\}$, then $\Lambda$ is the fundamental group of a graph of groups whose vertex and edge groups are isomorphic to $\mathbb Z$; all the $\ell^2$-Betti numbers of such a group vanish. For the comfort of the reader we propose a proof by hand.
	
	\begin{proof}
		We start with the first item. Recall that in a free group $F$, whenever non-trivial elements $g,h\in F$ satisfy $g h^k g\inv = h^l$ with $k\neq 0 \neq l$, then there is $a\in F$ such that $g,h$ are both powers of $a$. In particular, such elements $g,h$ always commute.
		
		Now, assume that $\Lambda$ is free and $\Lambda\cap \la b \ra\not\eq \{\id\}$, say $\Lambda\cap \la b \ra = \la b^s \ra$ where $s>0$.
		Pick any $\lambda\in\Lambda$ and set $H_\lambda \coloneqq \la b^s \ra \cap \lambda \la b^s \ra \lambda\inv$,
		which is the intersection of $\Lambda$ with the stabilizer of the geodesic $[\la b \ra, \lambda\la b \ra]$ in $\Tree$.
		Observe that $H_\lambda$
		is a finite index subgroup of both $ \la b^s \ra$ and $\lambda \la b^s \ra \lambda\inv$. 
		Therefore, there are $k\neq 0 \neq l$ such that $\lambda b^{sk} \lambda\inv = b^{sl}$. 
		As $\Lambda$ is free, $\lambda$ and $b^s$ commute.
		
		Consequently, the center of $\Lambda$ contains $\la b^s \ra$.
		Thus, the rank of $\Lambda$ is $1$; in other words $\Lambda$ is infinite cyclic.
		It is now clear that $\la b^s \ra$ has finite index in both $\Lambda$ and $\la b \ra$, so $\Lambda$ is virtually a subgroup of $\langle b\rangle$.
		
		Let us turn to the second item.
		The fundamental group of a graph of groups surjects onto the fundamental group of the underlying graph. The condition in Item \ref{item: fund group in amenable case} implies the amenability of $\BSo(m,n)$. Its subgroups thus cannot surject onto a non-amenable free group. 
	\end{proof}

	\section{Bass-Serre graphs}    
	\subsection{Pre-actions}
	\label{subsec: preactions}
	
	Let $m,n\in\Z\smallsetminus\{0\}$ and $\BSo(m,n) = \la b,t \ \vert\ tb^m =b^n t \ra$.
	
		Recall that a partial bijection of a set $X$ is a bijection between two subsets of $X$.
		Our actions are on the right; thus in a product of (partial) bijections $\sigma\tau$, the transformation $\sigma$ is applied first.
	
	\begin{definition}\label{def pre-action}
		Given a bijection $\beta$ of a set $X$ and 
		a partial bijection $\tau$ of $X$, we say that $\tau$ is \textbf{$(\beta^n,\beta^m)$-equivariant} if 
		$\tau\beta^m = \beta^n \tau$ as partial bijections,	that is:
		\begin{itemize}
			\item $\dom(\tau)$ is $\beta^n$-invariant;
			\item $\rng(\tau)$ is $\beta^m$-invariant;
			\item $x\tau \beta^m=x\beta^n\tau$ for all $x\in\dom(\tau)$.
		\end{itemize}
		A \textbf{pre-action} of $\BSo(m,n)$ on a set $X$ is a couple $(\beta,\tau)$ where $\beta$ is a bijection of $X$ and $\tau$ is a $(\beta^n,\beta^m)$-equivariant partial bijection of $X$. The set $X$ is called the \textbf{domain} of the pre-action.
		Such a pre-action is \textbf{saturated} if $\dom(\tau) = X = \rng(\tau)$.
	\end{definition}
	
	\begin{remark}
		Saturated pre-actions $(\beta,\tau)$ correspond to actions $\alpha$ of $\BSo(m,n)$ on the same set $X$ under the association 
		$\beta\leftrightarrow \alpha(b)$ and $\tau\leftrightarrow \alpha(t)$.
	\end{remark}
	
	\begin{definition}
		Given a pre-action $(\beta,\tau)$ of $\BSo(m,n)$, its \textbf{Schreier graph} is the oriented labeled
		graph $\Sch(\beta,\tau)=\mathcal G$ defined by
		\[ V(\mathcal G)\coloneqq X,\quad \left\{\begin{array}{l}
			E^+(\mathcal G)\coloneqq X\times\{b\}\sqcup \dom(\tau)\times\{t\}, \\
			E^-(\mathcal G)\coloneqq X\times\{b\inv\}\sqcup \rng(\tau)\times\{t\inv\}, 
		\end{array}\right.
		\]
		where the label of any edge is its second component and:
		\begin{itemize}
			\item for all $x\in X$, we set \[\source(x,b)\coloneqq x,\  \target(x,b)\coloneqq x\beta,\text{ and }\overline{(x,b)}\coloneqq (x\beta,b\inv);\]
			\item for all $x\in \dom(\tau)$, we set \[\source(x,t)\coloneqq x,\  \target(x,t)\coloneqq x\tau,\text{ and }\overline{(x,t)}\coloneqq (x\tau,t\inv).\]
		\end{itemize}
	\end{definition}
	
	Notice that the orientation of any edge $(x,l)$ is determined by its label $l$ and that the source of $(x,l)$ is $x$, regardless of its orientation.
	
	Noting that a $\BSo(m,n)$-action is transitive if and only if the associated Schreier graph is connected,
	we make the following definition.
	
	\begin{definition}\label{Def: transitive pre-action}
		A pre-action of $\BSo(m,n)$ is \textbf{transitive} if its Schreier graph is connected.
	\end{definition}

	\subsection{Bass-Serre graphs}
	
	We now introduce an important tool for our study. It is the labeled graph obtained from the Schreier graph defined in Section~\ref{subsec: preactions} by ``shrinking each $\beta$-orbit to one point''. We identify together the $t$-edges whose initial vertices belong to the same $\beta^{n}$-orbit. 
	Note that their terminal vertices automatically belong to the same $\beta^{m}$-orbit. 
	
	We label the vertices by the cardinality of the corresponding $\beta$-orbit and the edges by the cardinality of the corresponding $\beta^n$-orbit. This is illustrated by Figure \ref{fig: projection of Schreier graph onto Bass-Serre graph}. 
	The formal definition is as follows.
	
	\begin{definition}\label{def: Bass-Serre graph of pre-action}
		The \textbf{Bass-Serre graph} associated to a pre-action \mbox{$\alpha=(\beta,\tau)$} of $\BSo(m,n)$ on a set $X$
		is the oriented labeled graph $\BSe(\alpha)$ defined by
		$$V(\BSe(\alpha))\coloneqq X/\la\beta\ra, \quad
		\left\{\begin{array}{l}
			E^+(\BSe(\alpha))\coloneqq\dom(\tau)/\la\beta^n\ra,  \\
			E^-(\BSe(\alpha))\coloneqq\rng(\tau)/\la\beta^m\ra.
		\end{array}\right.
		$$
		For every $x\in\dom \tau$, we set
		\[\source(x\la \beta^n\ra)\coloneqq x\la \beta\ra,\  \target(x\la \beta^n\ra)\coloneqq x\tau\la \beta\ra,\text{ and }\overline{x\la \beta^n\ra}\coloneqq x\tau \la \beta^m\ra=x\la \beta^n\ra \tau.\]
		We define the label map $L\colon V(\BSe(\alpha))\sqcup E(\BSe(\alpha))\to \Z_{\geq 1}\cup \{\infty\}$ by
		\[
		L(x\la\beta\ra)\coloneqq\abs{x\la\beta\ra},
		\quad L(x\la\beta^n\ra)\coloneqq \abs{x\la\beta^n\ra},
		\quad L(y\la\beta^m\ra)\coloneqq \abs{y\la\beta^m\ra}.
		\]
	\end{definition}
	
	\begin{remark}
		For any $x\in \dom(\tau)$,
		the $(\beta^n,\beta^m)$-equivariant partial bijection $\tau$ induces a bijection
		from $x\langle \beta^n\rangle$ to $x\tau\langle \beta^m\rangle$. Thus both the target and the opposite maps of $\BSe(\alpha)$ are well-defined and the label of each edge is equal to the label of its opposite.
	\end{remark}
	
	\begin{remark}
		We view the sets $E^+(\BSe(\alpha))$ and $E^-(\BSe(\alpha))$ as disjoint sets, even though we might have that $\dom(\tau)/\la\beta^n\ra \cap \rng(\tau)/\la\beta^m\ra \neq \emptyset$.
		Note that the source of an edge $x\la \beta^k\ra\in E^\pm(\BSe(\alpha))$ is $x\la\beta\ra$ regardless of its orientation.
	\end{remark}
	
	\begin{remark}\label{rmk: exchange m and n in BS graph}
		The groups $\BSo(m,n)$ and $\BSo(n,m)$ are isomorphic via $b\mapsto b$ and $t\mapsto t\inv$.
		For every pre-action $(\beta,\tau)$ of $\BSo(m,n)$, the couple $(\beta,\tau\inv)$ is a pre-action of $\BSo(n,m)$. At the level of Bass-Serre graphs, $\BSe(\beta,\tau)$ and $\BSe(\beta,\tau\inv)$ coincide, except that the orientation is reversed.
	\end{remark}
	
	\begin{remark}
		In the case of a transitive $\BSo(m,n)$-action, 
		the graph underlying our Bass-Serre graph
		is the quotient of the Bass-Serre tree $\Tree$ by the stabilizer of any point of $X$, as will be explained in Section~\ref{sect: Bass-Serre graphs and Bass-Serre theory}.
	\end{remark}
	
	We now clarify what we meant by ``shrinking each $\beta$-orbit to a point'', by noting that we have the following simplicial map from the Schreier graph to the Bass-Serre graph of any pre-action.
	
	\begin{definition}\label{dfn: projection BS graph}
		The \textbf{projection} associated to a pre-action $\alpha=(\beta,\tau)$ is the application $\pi_\alpha$ given by
		\begin{alignat*}{4}
			V(\Schreier{(\alpha)}) &\to V(\BSe(\alpha)) ,&
			x &\mapsto x\la \beta \ra\\
			E_t^+(\Schreier{(\alpha)}) &\to E^+(\BSe(\alpha)) ,& (x,t) &\mapsto x\la \beta^n \ra \\
			E_t^-(\Schreier{(\alpha)}) &\to E^-(\BSe(\alpha)),\quad & (x,t^{-1}) &\mapsto x\la \beta^m \ra \\
			E_b(\Schreier(\alpha)) &\to V(\BSe(\alpha)), & (x,b^{\pm 1}) &\mapsto x\la\beta\ra
		\end{alignat*}
		where $E_{t}^\pm(\Schreier{(\alpha)})$ is the subset of edges in $\Schreier{(\alpha)}$ whose label is $t$ or $t\inv$ respectively and $E_b$ is the subset of edges whose label is $b$ or $b^{-1}$.
	\end{definition}
	This projection is illustrated in Figure \ref{fig: projection of Schreier graph onto Bass-Serre graph}.
	Given any subgraph $\mathcal G\subseteq \Schreier(\alpha)$ or path $\mathfrak p$ in $\Schreier{(\alpha)}$ we obtain a subgraph $\pi_\alpha(\mathcal G)\subseteq \BSe(\alpha)$ or a path $\pi_\alpha(\mathfrak p)$ in $\BSe(\alpha)$.
	
	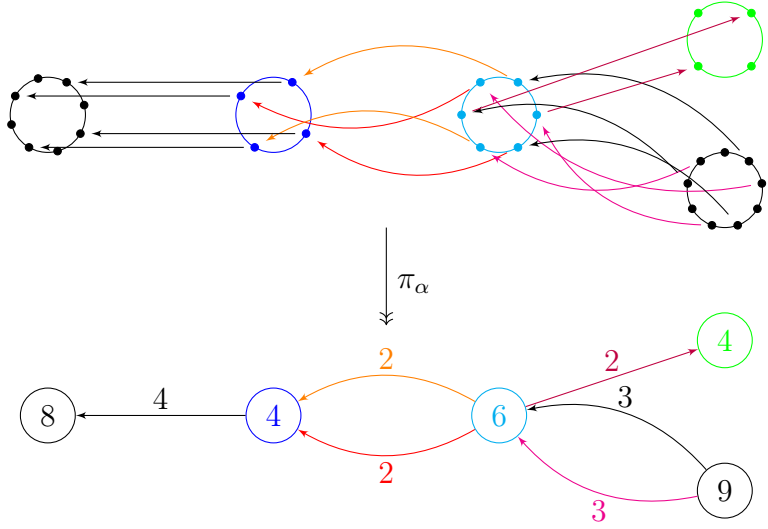
\begin{figure}[h!]
		\centering
		\begin{tikzpicture}
			\tikzset{vertex/.style = {shape=circle,draw,minimum size=1.5em}}
			\tikzset{edge/.style = {->,> = latex'}}
			\node[vertex] (e) at  (-4.5,0) {8};
			\node[vertex, color = blue] (a) at  (-1.5,0) {4};
			\node[vertex, color = cyan] (b) at  (1.5,0) {6};
			\node[vertex, color = green] (c) at  (4.5,1) {4};
			\node[vertex] (d) at  (4.5,-1) {9};
			\draw[edge, color = red] (b) to[bend left] node [label={[label distance=-0.1cm,below]:2}] {} (a);
			\draw[edge, color= orange] (b) to[bend right] node [label={[label distance=-0.2cm,above]:2}] {} (a);
			\draw[edge, color = purple] (b) to node [label={[label distance=-0.2cm,above]:2}] {} (c);
			\draw[edge,color = magenta] (d) to[bend left] node [label={[label distance=-0.1cm,below]:3}] {} (b);
			\draw[edge] (d) to[bend right] node [label={[label distance=-0.2cm,above]:3}] {} (b);
			\draw[edge] (a) to node [label={[label distance=-0.2cm,above]:4}] {} (e);	
			\draw[->>] (0,2.5) -- (0,1.2);
			\draw (0, 1.75) node[right]{$\pi_\alpha$};
			\node (e') at (-4.5,4) {};
			\draw (e') circle (0.5);  
			\draw[fill] (e') +(60:0.5) node (e1){}  circle (0.05);
			\draw[fill] (e') +(105:0.5) node (e2){}  circle (0.05);
			\draw[fill] (e')+(150:0.5) node (e3){} circle (0.05);
			\draw[fill] (e') +(195:0.5) node (e4){}  circle (0.05);
			\draw[fill] (e')+(240:0.5) node (e5){} circle (0.05);
			\draw[fill] (e') +(285:0.5) node (e6){}  circle (0.05);
			\draw[fill] (e')+(330:0.5) node (e7){} circle (0.05);
			\draw[fill] (e') +(15:0.5) node (e8){}  circle (0.05);
			\node (a') at (-1.5,4) {};
			\draw[color = blue] (a') circle (0.5);  
			\draw[fill, color = blue] (a') +(60:0.5) node (a1){}  circle (0.05);
			\draw[fill, color = blue] (a')+(150:0.5) node (a2){} circle (0.05);
			\draw[fill, color = blue] (a')+(240:0.5) node (a3){} circle (0.05);
			\draw[fill, color = blue] (a')+(330:0.5) node (a4){} circle (0.05);
			\node (b') at (1.5,4) {};
			\draw[color = cyan] (b') circle (0.5);
			\draw[fill, color = cyan] (b') +(0:0.5) node (b1){}  circle (0.05);
			\draw[fill, color = cyan] (b') +(60:0.5) node (b2){}  circle (0.05);
			\draw[fill, color = cyan] (b') +(120:0.5) node (b3){}  circle (0.05);
			\draw[fill, color = cyan] (b') +(180:0.5) node (b4){}  circle (0.05);
			\draw[fill, color = cyan] (b') +(240:0.5) node (b5){}  circle (0.05);
			\draw[fill, color = cyan] (b') +(300:0.5) node (b6){}  circle (0.05);
			\node (c') at (4.5,5) {};
			\draw[color = green] (c') circle (0.5);  
			\draw[fill, color = green] (c') +(45:0.5) node (c1){}  circle (0.05);
			\draw[fill, color = green] (c')+(135:0.5) node (c2){} circle (0.05);
			\draw[fill, color = green] (c')+(225:0.5) node (c3){} circle (0.05);
			\draw[fill, color = green] (c')+(315:0.5) node (c4){} circle (0.05);
			\node (d') at (4.5,3) {};
			\draw (d') circle (0.5);
			\draw[fill] (d') +(10:0.5) node (d1){}  circle (0.05);
			\draw[fill] (d') +(50:0.5) node (d2){}  circle (0.05);
			\draw[fill] (d') +(90:0.5) node (d3){}  circle (0.05);
			\draw[fill] (d') +(130:0.5) node (d4){}  circle (0.05);
			\draw[fill] (d') +(170:0.5) node (d5){}  circle (0.05);
			\draw[fill] (d') +(210:0.5) node (d6){}  circle (0.05);
			\draw[fill] (d') +(250:0.5) node (d7){}  circle (0.05);
			\draw[fill] (d') +(290:0.5) node (d8){}  circle (0.05);
			\draw[fill] (d') +(330:0.5) node (d9){}  circle (0.05);
			\draw[edge, color = red] (b6) to[bend left] (a4);
			\draw[edge, color = red] (b3) to[bend left] (a2);
			\draw[edge, color = orange] (b5) to[bend right] (a3);
			\draw[edge, color = orange] (b2) to[bend right] (a1);
			\draw[edge, color = magenta] (d7) to[bend left] (b1);
			\draw[edge, color = magenta] (d1) to[bend left] (b3);
			\draw[edge, color = magenta] (d4) to[bend left] (b5);
			\draw[edge] (d8) to[bend right] (b6);
			\draw[edge] (d2) to[bend right] (b2);
			\draw[edge] (d5) to[bend right] (b4);
			\draw[edge, color = purple] (b4) to (c1);
			\draw[edge, color = purple] (b1) to (c3);
			\draw[edge] (a1) to (e1);
			\draw[edge] (a2) to (e3);
			\draw[edge] (a3) to (e5);
			\draw[edge] (a4) to (e7);
		\end{tikzpicture}
		\caption{The projection from the Schreier graph onto the Bass-Serre graph of some non-saturated transitive $\BSo(2,3)$-pre-action.
			The dotted circles represent the $\beta$-orbits in the Schreier graph.}
		\label{fig: projection of Schreier graph onto Bass-Serre graph}
	\end{figure}
	
	Note that for every vertex $v=x\la \beta\ra$, 
	\[\abs{x\la \beta^k\ra}=\frac{\abs{x\la \beta\ra}}{\gcd(\abs{x\la \beta\ra},k)},\]
	thus the following facts hold:
	\begin{itemize}
		\item all the $v$-outgoing edges $e$ have the same label, which is:
		\begin{equation*}
			L(e)=\frac{L(v)}{\gcd(L(v),n)},
		\end{equation*}
		\item all the $v$-incoming edges $e'$ have the same label, which is:
		\begin{equation*}
			L(e')=\frac{L(v)}{\gcd(L(v),m)}.
		\end{equation*}
	\end{itemize}
	We also have the following relations between labels and degrees:
	\begin{itemize}
		\item The outgoing degree $\deg_{\mathrm{out}}(v)$
		is equal to the number of $\beta^n$-orbits contained
		in $x\la \beta\ra \cap \dom(\tau)$. 
		Recall that $\dom(\tau)$ is $\beta^n$-invariant.
		Since $x\la \beta\ra$ contains exactly $\gcd(L(v),n)$ orbits under $\beta^n$, we get
		\begin{equation*}
			\degout(v)\leq\gcd(L(v),n),   
		\end{equation*}
		with equality if and only if $x\la\beta\ra \subseteq \dom(\tau)$.
		\item Similarly, the incoming degree $\degin(v)$ is equal to the number of $\beta^m$-orbits contained
		in $x\la \beta\ra \cap \rng(\tau)$, so
		\begin{equation*}
			\degin(v) \leq \gcd(L(v),m),
		\end{equation*}
		with equality if and only if $x\la\beta\ra \subseteq \rng(\tau)$.
	\end{itemize}
	
	\begin{remark}
		As a consequence of the last two items, the pre-action is an action if and only if, for every vertex $v$,
		\[
		\degout(v) = \gcd(L(v),n)\ \ \text{ and } \ \ \degin(v) = \gcd(L(v),m).
		\] 
	\end{remark}

	\subsection{\texorpdfstring{$(m,n)$-graphs}{(m,n)-graphs}}
	\label{sect: (m,n)-graphs}
	
	We now introduce an axiomatization
	of the Bass-Serre graphs we obtain from pre-actions. Recall that by convention $\gcd(\infty,k)=\abs k$ for all $k\neq 0$.
	
	\begin{definition}\label{def:mngraph}
		An $(m,n)$\textbf{-graph} is an oriented labeled graph $\mathcal G=(V,E)$ with label map $L\colon V\sqcup E\to\Z_{\geq 1}\cup\{\infty\}$ such that:
		\begin{itemize}
			\item for every positive edge $e\in E^+$,
			\begin{equation}
				\label{eq:transfert}
				\frac{L(\source(e))}{\gcd(L(\source(e)),n)}=L(e)=\frac{L(\target(e))}{\gcd(L(\target(e)),m)};
			\end{equation}
			
			\item for every negative edge $e\in E^-$, $L(e)=L(\bar e)$;
			\item for every vertex $v\in V$, we have
			\begin{equation}\label{eq: deg bound by gcd L(v), n m}
				\degout(v)\leq \gcd(L(v),n)\ \ \text{ and } \ \ \degin(v)\leq \gcd(L(v),m).
			\end{equation}
		\end{itemize}
	\end{definition}
	
	\begin{example}
		The Bass-Serre graph of any pre-action of $\BSo(m,n)$ is an $(m,n)$-graph. The converse will be shown in Proposition \ref{prop: realization of BS graph}.
	\end{example}
	
	\begin{remark}
		Observe that an edge label is uniquely determined by the label of any of its vertices. 
		The edge labels are thus redundant and are just calculation tools 
		(see also Remark~\ref{rem: edge phenotype is coarser}).		
	\end{remark}

	\begin{example}
		Let us see how labels interact for $m=2$ and $n=3$. If $e$ is an edge in a $(2,3)$-graph, then once we fix the label of one of the  extremities, the other one can be chosen according to the Table \ref{table}, using Formula \eqref{eq:transfert} for $L(e)$. The reader is invited to consult the webpage
		\cite{carderiHowBuildGraphs2022} to see the kinds of local constraints which occur in general.
		
		\begin{center}
			\begin{tabular}{|c|c|}
				\hline
				\begin{tabular}{c}
					\\
					If $\gcd(L(\source(e)),2)=1$\\
					$L(\target(e))\in \displaystyle{\left\{L(e),\ 2L(e) \right\}}$
				\end{tabular}
				& \begin{tabular}{c}
					\\
					If $\gcd(L(\source(e)),2)=2$\\
					$L(\target(e))=2L(e)$
				\end{tabular}\\
				&\\
				\hline
				\begin{tabular}{c}
					\\
					If $\gcd(L(\target(e)),3)=1$\\
					$L(\source(e))\in \displaystyle{\left\{L(e),\ 3L(e) \right\}}$
				\end{tabular} & 
				\begin{tabular}{c}
					\\
					If $\gcd(L(\target(e)),3)=3$\\
					$L(\source(e))=3L(e)$
				\end{tabular}\\
				&\\
				\hline
			\end{tabular}
			\captionof{table}{How the label of the extremities impact each other}\label{table}
		\end{center}
		
		In Figure \ref{figure: 2 examples}, we give an illustrative example.

		\begin{figure}[h!]
			\centering
			\begin{subfigure}{.45\textwidth}
				\centering
				\begin{tikzpicture}
					\tikzset{vertex/.style = {shape=circle,draw,minimum size=1.5em}}
					\tikzset{edge/.style = {->,> = latex'}}
					\node[vertex] (b) at  (1.5,0) {3};
					\node[vertex] (c) at  (4.5,1.5) {1};
					\node[vertex] (d) at  (4.5,-1.5) {1};
					\node[vertex] (e) at  (4.5,0) {2};
					\draw[edge] (b) to node [label={[label distance=-0.2cm,above]:1}] {} (d);
					\draw[edge] (b) to node [label={[label distance=-0.2cm,above]:1}] {} (c);
					\draw[edge] (b) to node [label={[label distance=-0.2cm,above]:1}] {} (e);
					(e);
				\end{tikzpicture}
				\captionsetup{justification=centering}
				\caption{Various choices \\ \ \  for the label $L(t(e))$.}
				\label{fig:sub1}
			\end{subfigure}
			\begin{subfigure}{.45\textwidth}
				\centering
				\begin{tikzpicture}
					\tikzset{vertex/.style = {shape=circle,draw,minimum size=1.5em}}
					\tikzset{edge/.style = {->,> = latex'}}
					\node[vertex] (b) at  (1.5,0) {6};
					\node[vertex] (c) at  (4.5,1.5) {4};
					\node[vertex] (d) at  (4.5,-1.5) {4};
					\node[vertex] (e) at  (4.5,0) {4};
					\draw[edge] (b) to node [label={[label distance=-0.2cm,above]:2}] {} (d);
					\draw[edge] (b) to node [label={[label distance=-0.2cm,above]:2}] {} (c);
					\draw[edge] (b) to node [label={[label distance=-0.2cm,above]:2}] {} (e);
					(e);
				\end{tikzpicture}
				\captionsetup{justification=centering}
				\caption{No choice for \\ \ \   the label $L(t(e))$.}
				\label{fig:sub2}
			\end{subfigure}
			\caption{Two examples of $(2,3)$-graphs.\label{figure: 2 examples}} 
	\end{figure}
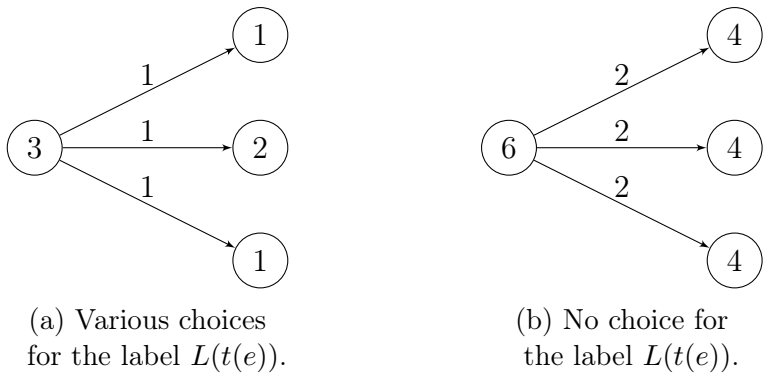
\end{example} 

\begin{remark}\label{rmk: exchange m and n on (m,n)-graph}
	As in Remark \ref{rmk: exchange m and n in BS graph},
	every $(m,n)$-graph can be turned into an $(n,m)$-graph by flipping the orientations of its edges. 
	Note that this operation does not affect the labels. 
\end{remark}

\begin{remark}\label{rmk: cardinal of Labels}
	In a connected $(m,n)$-graph, the labels are, either all finite, or all $\infty$ by Equation~\eqref{eq:transfert}. This will be made more precise in Proposition~\ref{prop:phenotype is well defined for conn graph}.
	Observe that any oriented graph $\Gc$ satisfying $\degin(v)\leq m$ and $\degout(v)\leq n$ for every $v\in V(\Gc)$ becomes an $(m,n)$-graph if we set all the labels to be infinite.
	However one cannot always put finite labels, see Lemma~\ref{lem: non-maximal degree in finite phenotype}. 
\end{remark}

\begin{definition} \label{def: saturated m,n-graph, degrees as functions of labels}
	Let $\Gc$ be an $(m,n)$-graph. A vertex $v$ in $\Gc$ is \textbf{saturated} if
	the inequalities \eqref{eq: deg bound by gcd L(v), n m} are indeed equalities, i.e.
	\begin{equation*}
		\degout(v)=\gcd(L(v),n)\ \ \text{ and } \ \ \degin(v)=\gcd(L(v),m).
	\end{equation*} 
	The $(m,n)$-graph $\Gc$ is \textbf{saturated} if all its vertices are saturated.
\end{definition}

\begin{example}\label{ex:saturated-> action}
	The Bass-Serre graph of a pre-action of $\BSo(m,n)$ is saturated if and only if the pre-action is an action. 
\end{example}

\subsection{Realizing \texorpdfstring{$(m,n)$}{(m,n)}-graphs as Bass-Serre graphs}

\begin{proposition}\label{prop: realization of BS graph}
	Every $(m,n)$-graph $\mathcal G$ is the Bass-Serre graph of at least one pre-action of $\BSo(m,n)$.
	Any such pre-action is an action if and only if $\Gc$ is saturated.
\end{proposition}

The above proposition is a consequence of the following stronger statement where by definition, a \textbf{sub-}$(m,n)$\textbf{-graph} of an $(m,n)$-graph $\mathcal G$ is a
subgraph $\mathcal G'$ labeled by the restriction of the label map
of $\mathcal G$.

\begin{proposition}[Extension of pre-actions from $(m,n)$-graphs] \label{Extending pre-action onto a (m,n)-graph}
	Let $\Gc_1$ be the Bass-Serre graph of a pre-action $\alpha_1$ and let $\Gc_2$ be an $(m,n)$-graph that contains $\Gc_1$ as a  sub-$(m,n)$-graph.
	Then $\mathcal G_2$ is the Bass-Serre graph of a pre-action $\alpha_2$ that extends $\alpha_1$.
\end{proposition}
\begin{proof}
	We start with a pre-action $(\beta_1,\tau_1)$ on $X_1$ which yields the Bass-Serre graph $\mathcal G_1$.
	Let $W\coloneqq V(\mathcal G_2)\smallsetminus V(\mathcal G_1)$ and $X_2\coloneqq X_1\sqcup\bigsqcup_{v\in W} X_v$
	where each $X_v$ is a set of cardinality $\abs{X_v}=L(v)$.
	We extend $\beta_1$ to a permutation $\beta_2$ of $X_2$ by making it act as a cycle of length $L(v)$ on $X_v$.
	
	By Zorn's lemma, it suffices to extend $\tau_1$ when $\mathcal G_1$ only lacks one positive $\mathcal G_2$-edge.
	So suppose $E^+(\mathcal G_1)\sqcup\{e\}=E^+(\mathcal G_2)$. Then by 
	Inequality \eqref{eq: deg bound by gcd L(v), n m} in Definition \ref{def:mngraph},
	\[\degout^{\mathcal G_1}(\source(e))<\degout^{\mathcal G_2}(\source(e))\leq \gcd(L(\source(e)),n)\] 
	and similarly
	\[\degin^{\mathcal G_1}(\target(e))<\degin^{\mathcal G_2}(\target(e))\leq \gcd(L(\target(e)),m).\]
	We can thus find a $\beta_2^n$-orbit $y\la\beta_2^n\ra$ contained in the $\beta_2$-orbit $\source(e)$
	but disjoint from $\dom(\tau_1)$ and a $\beta_2^m$-orbit $z\la\beta_2^m\ra$ contained in the $\beta_2$-orbit
	$\target(e)$ but disjoint from $\rng(\tau_1)$.
	
	Since these two orbits $y\la \beta_2^n\ra$ and $z\la \beta_2^m\ra$ share the same cardinal $L(e)$, 
	we can define $\tau_2$ as an extension of $\tau_1$ which is also $(\beta_2^n,\beta_2^m)$-equivariant when restricted to $y\la\beta_2^n\ra$ by letting
	\[
	y\beta_2^{kn}\tau_2=z\beta_2^{km}\text{ for all }k\in\Z.
	\]
	By construction $\tau_2$ is the desired extension.
\end{proof}

The pre-action $\alpha_2$ arising in Proposition \ref{Extending pre-action onto a (m,n)-graph} is definitively not unique in general.
In a forthcoming work, we will characterize which $(m,n)$-graphs arise as Bass-Serre graphs of continuum many non-isomorphic actions.
In particular we will show that the $(m,n)$-graphs whose underlying graph has non-finitely generated fundamental group are 
of this kind.
Such $(m,n)$-graphs always exist as soon as $\abs m\geq 2$ and $\abs n\geq 2$. 
Here we give a simple example of a graph associated to continuum many non-isomorphic actions 
for $n=m=2$.

\begin{example}\label{ex: many realizations}
	Let $\Gc$ be the $(2,2)$-graph whose underlying graph is such that $V(\Gc)=\mathbb Z$ and for every $z\in V(\Gc)$ there are exactly two $z$-outgoing edges, one to $z$ and the other to $z+1$. That is, $\Gc$ is a line where every vertex has an extra loop. We set the labels of $\Gc$ to be all infinite. 
	
	\begin{figure}[h!]
		\centering
		\begin{tikzpicture}
			\tikzset{vertex/.style = {shape=circle,draw,minimum size=1.5em}}
			\tikzset{edge/.style = {->,> = latex'}}
			\node[vertex] (e) at  (-4.5,0) {$\infty$};
			\node[vertex] (a) at  (-1.5,0) {$\infty$};
			\node[vertex] (b) at  (1.5,0) {$\infty$};
			\node[vertex] (c) at  (4.5,0) {$\infty$};
			\draw[edge] (e) to node {} (a);
			\draw[edge] (a) to node {} (b);
			\draw[edge] (b) to node {} (c);
			\draw[edge] (e) to [out=45,in=135,looseness=10] (e);
			\draw[edge] (a) to [out=45,in=135,looseness=10] (a);
			\draw[edge] (b) to [out=45,in=135,looseness=10] (b);
			\draw[edge] (c) to [out=45,in=135,looseness=10] (c);
			\node (f) at (-6,0) {};
			\node (g) at (6,0) {};
			\draw[edge, dotted] (f) to node {} (e);
			\draw[dotted] (c) to node {} (g);
		\end{tikzpicture}
		\caption{The (2,2)-graph $\Gc$}
	\end{figure}
	
	Set $X\coloneqq V(\Gc)\times \mathbb Z\cong \mathbb Z\times\mathbb Z$. 
	For every function $f\colon \Z\to\Z$ such that $\forall w < 0,\, f(w)=0$ and $f(0)\neq 0$,
	we define an action $\alpha_f$ as follows: for all $(k,l)\in X$ 
	\begin{align*}
		(k,l)\alpha_f(b)\coloneqq&(k,l+1);\\
		(k,l)\alpha_f(t)\coloneqq&\begin{cases}(k+1,l)&\text{ if }l\text{ is odd};\\ (k,l+f(k))&\text{ if }l\text{ is even}.\end{cases}
	\end{align*}
	It is easy to check that all $\alpha_f$ are actions of $\BSo(2,2)$ whose Bass-Serre graph is $\Gc$,
	that $\alpha_f$ and $\alpha_g$ are non-conjugate for $f\neq g$,
	and that there are continuum many such actions.
\end{example}

\subsection{Additional properties of \texorpdfstring{$(m,n)$}{(m,n)}-graphs}
In this section, we collect some basic consequences of the definition of $(m,n)$-graphs.
Observe that Equation \eqref{eq:transfert} is equivalent to the fact that 
\begin{equation}\label{eq: def mngraphes with max}
	\max(\abs{L(s(e))}_p - \abs{n}_p,0)
	= \abs{L(e)}_p
	= \max(\abs{L(\target(e))}_p - \abs{m}_p, 0)
\end{equation}
from which we obtain the following.  

\begin{remark}\label{labels of extremities}   
	Consider an oriented labeled graph $\mathcal G=(V,E)$ with label map $L\colon V\sqcup E\to\Z_{\geq 1}$ satisfying $L(\bar e) = L(e)$ for every edge $e$. 
	The labeled graph $\Gc$ is an $(m,n)$-graph if and only if the following two conditions hold:
	\begin{itemize}
		\item for every positive edge $e$ and every prime $p$ such that $\abs{L(e)}_p \geq 1$,
		\begin{equation}\label{eq: labels of extremities: equality case}
			\abs{L(\source(e))}_p = \abs{L(e)}_p + \abs{n}_p
			\quad \text{ and } \quad
			\abs{L(\target(e))}_p = \abs{L(e)}_p + \abs{m}_p ,
		\end{equation}
		\item for every positive edge $e$ and every prime $p$ such that $\abs{L(e)}_p = 0$,
		\begin{equation}\label{eq: labels of extremities: inequalities case}
			0 \leq \abs{L(\source(e))}_p \leq \abs{n}_p
			\quad \text{ and } \quad
			0 \leq \abs{L(\target(e))}_p \leq \abs{m}_p . 
		\end{equation}
	\end{itemize}
\end{remark}    

In particular, in an $(m,n)$-graph, $\abs{L(\source(e))}_p>\abs{n}_p$ if and only if $\abs{L(\target(e))}_p>\abs m_p$, and if one of these two equivalent conditions is met then
\begin{equation}\label{eq: transfer for p}
	\abs{L(\target(e))}_p=\abs{L(\source(e))}_p+\abs m_p-\abs n_p.
\end{equation}

\begin{lemma}\label{lem: infinite paths}
	Fix a prime $p$ such that $\abs n_p<\abs m_p$
	and let $\mathcal G$ be an $(m,n)$-graph. 
	If $(e_k)_{k\geq 1}$ is any infinite path consisting of positive edges with $L(\source(e_1))\neq \infty$ and $\abs{L(\source(e_1))}_p>\abs n_p$, then
	\begin{align*}
		\lim_{k\to +\infty} \abs{L(\source(e_k))}_p=+\infty.
	\end{align*}
	If $(e_k)_{k\geq 1}$ is any infinite path consisting of negative edges with $L(\source(e_1))\neq \infty$, then
	\begin{align*}
		\limsup_{k\to +\infty} \abs{L(\source(e_k))}_p<\abs m_p.
	\end{align*}
\end{lemma}
\begin{proof}
	If $(e_k)_{k\geq 1}$ is an infinite path consisting of positive edges such that $\abs{L(\source(e_1))}_p>\abs n_p$, then by a straightforward induction using Equation \eqref{eq: transfer for p} we have that  \begin{equation}\label{eq: padic goes up}
		\abs{L(\source(e_k))}_p=\abs{L(\source(e_1))}_p+k(\abs m_p-\abs n_p) 
	\end{equation}for all $k\geq 1$. The first result follows.
	
	For the second one, let $(e_k)_{k\geq 1}$ be an infinite path consisting of negative edges. 
By exchanging the roles in Equation \eqref{eq: transfer for p}, we have the claim: 
{\em  if $e$ is a negative edge  then  $\abs{L(\source(e))}_p>\abs m_p$ if and only if  $\abs{L(\target(e))}_p>\abs{n}_p$; and when this occurs 
$\abs{L(\target(e))}_p=\abs{L(\source(e))}_p-\abs m_p+\abs n_p.$}
\\
Thus, $\abs{L(\source(e_{k+1}))}_p=\abs{L(\target(e_k))}_p<\abs{L(\source(e_k))}_p$ as long as  $\abs{L(\source(e_k))}_p>\abs m_p$.
 So there must be $k_0\in\N$ such that $\abs{L(\source(e_{k_0}))}_p\leq \abs m_p$ (this could have already happened for $k_0=1$). 
From this point, we have $\abs{L(\source(e_{k_0+1}))}_p=\abs{L(\target(e_{k_0}))}_p\leq \abs n_p<\abs m_p$ and an inductive use of the claim gives $\abs{L(\source(e_k))}_p\leq \abs n_p<\abs m_p$ for all $k> k_0$. This finishes the proof.
\end{proof}

\begin{remark}
	It follows from Equation \eqref{eq: padic goes up} that any infinite path $(e_k)_{k\geq 1}$ consisting of positive edges with $L(\source(e_1))\neq \infty$ and $\abs{L(\source(e_1))}_p>\abs n_p$ has to be a simple path.
\end{remark}

\begin{lemma}\label{lem: non-maximal degree in finite phenotype}
	If $\abs m> \abs n$ and $\mathcal G$ is an $(m,n)$-graph with a vertex of finite label, then there is a vertex $v\in V(\mathcal G)$ such that $\degin(v)< \abs m$.
\end{lemma}
\begin{proof}
	Assume by contradiction that $\degin(v)= \abs m$ for all $v\in V(\mathcal G)$. 
	Then we can build inductively an infinite path $(e_k)_{k\in\N}$ consisting of negative edges with $L(\source(e_0))$ finite. 
	By the previous lemma this path goes through some vertex $v_0$ that $\abs{L(v_0)}_p<\abs m_p$. Then $\degin(v_0)=\gcd(L(v_0),m)<\abs m$, a contradiction.
\end{proof}

\subsection{Bass-Serre theory for \( \BSo(m,n) \)}
\label{sect: Bass-Serre graphs and Bass-Serre theory}

Take $m,n\in \mathbb Z\smallsetminus \{0\}$.
Set $\Gamma\coloneqq\BSo(m,n)=\la b,t\vert tb^{m}t^{-1}=b^{n}\ra$ and put $S\coloneqq \{b,t\}$. Denote by $\Tree$ the associated Bass-Serre tree and remark that it is the underlying oriented graph of the Bass-Serre graph of the transitive and free action: $\Tree=\BSe(\Gamma\curvearrowleft \Gamma)$.

Besides the Schreier graph, we can associate to each subgroup $\Lambda\leq \Gamma$ two decorated graphs:
\begin{itemize}
	\item the Bass-Serre graph of the action $\Lambda\backslash\Gamma \curvearrowleft\Gamma$;
	\item the quotient graph of groups $\Lambda\bs \Tree$ of the action $\Lambda\curvearrowright \Tree$. 
\end{itemize}
Let us observe that the underlying oriented graphs of the two above decorated graphs are the same. Indeed they are obtained as quotients of commuting actions as one can see in the following diagram where by $\curvearrowleft^V\la b \ra$ we mean that $\la b\ra$ acts only on the set of vertices,   where the $\swarrow$ arrows are graph morphisms obtained by quotienting by left $\Lambda$-actions, and where the dashed $\searrow$ arrows are projections as in Definition~\ref{dfn: projection BS graph}:
\begin{center}
	\begin{tikzcd}
		& \Lambda\curvearrowright \Cayley(\Gamma,S)\curvearrowleft^V \la b \ra
		\arrow[rd,dashed] \arrow[ld]\\
		\Lambda\backslash \Cayley(\Gamma,S)\curvearrowleft^V \la b \ra
		\arrow[d,equal]
		&& \Lambda\curvearrowright\BSe(\Gamma\curvearrowleft \Gamma)
		\arrow[d, equal] \\
		\Schreier(\Lambda,S)\curvearrowleft^V \la b \ra
		\arrow[dr,dashed]&&
		\Lambda\curvearrowright\Tree \arrow[dl]\\
		&\BSe(\Lambda\backslash \Gamma\curvearrowleft \Gamma) \simeq  \Lambda\backslash\Tree&\\
	\end{tikzcd}
\end{center}

Next, observe that, $\BSe(\Lambda\backslash \Gamma\curvearrowleft \Gamma)$ being saturated, one has $\degin(v) = \gcd(L(v),m)$ and $\degout(v) = \gcd(L(v),n)$ for every vertex $v$ in this graph.
Hence, for every edge $e$, one has
\[
\frac{L(\source(e))}{L(e)}= \gcd(L(\source(e)),n) = \degout(\source(e))
\quad \text{ and } \quad
\frac{L(\target(e))}{L(e)}=  \degin(\target(e)).
\]
Thus Remark \ref{rem: gpe fdmt si inter <b> trivial} and Proposition \ref{Prop: Bass-Serre th. recovering the lambda pm from the quotient graph} can be immediately reformulated in terms of the labels of the Bass-Serre graph $\BSe(\Lambda\backslash \Gamma\curvearrowleft \Gamma)$ as follows:
\begin{proposition}\label{prop: Bass-Serre graph vs graph of groups}
	Let $m$ and $n$ be non-zero integers.
	Let $\mathcal G$ be a saturated connected $(m,n)$-graph and 
	let $\Lambda$ be a subgroup of $\Gamma=\BSo(m,n)$ such that $\BSe(\Lambda\backslash \Gamma\curvearrowleft \Gamma) \simeq \Gc$. 
	\begin{enumerate}
		\item If all labels of $\Gc$ are infinite, then $\Lambda$ is a free group, namely isomorphic to the fundamental group of the graph $\Gc$.
		
		\item If all labels of $\Gc$ are finite, then the quotient graph of groups arising from the action $\Lambda \curvearrowright \Tree$ 
		is isomorphic to the graph of groups obtained by attaching a copy of $\Z$ to every vertex and every edge of $\Gc$, with structural maps of positive edges
		\begin{eqnarray*}
			\Z_e \hookrightarrow \Z_{\source(e)}, &  & k \mapsto \frac{n \cdot L(e)}{L(\source(e))} \cdot k, \\
			\Z_e \hookrightarrow \Z_{\target(e)}, &  & k \mapsto \frac{m \cdot L(e)}{L(\target(e))} \cdot k.
		\end{eqnarray*}
	\end{enumerate}
\end{proposition}

Then, combining Proposition \ref{prop: Bass-Serre graph vs graph of groups} and Lemma~\ref{lem: non-maximal degree in finite phenotype}, we get the following rephrasing of Corollary \ref{cor: oriented graph determines isomorphis type}:

\begin{corollary}\label{cor:iso type of Lambda vs BS(Lambda)}
	Let $m$ and $n$ be non-zero integers such that $\abs m\neq \abs n$. Then the isomorphism type of $\Lambda\leq \BSo(m,n)$ depends only on the graph structure of $\BSe(\Lambda)$.
\end{corollary}
\begin{proof}
	Recall that if an $(m,n)$-graph is saturated and has only infinite labels, then all vertices have incoming degree $\abs m$ and outgoing degree $\abs n$.
	Lemma~\ref{lem: non-maximal degree in finite phenotype} thus allows us to detect whether the Bass Serre graph of $\Lambda$ contains infinite labels by purely looking at its graph structure: it has infinite labels if and only if all vertices have degree $\abs n+\abs m$.
	The result now follows from Proposition \ref{prop: Bass-Serre graph vs graph of groups}.
\end{proof}

\begin{remark}  When $\abs m= \abs n$, the statement analogue to that of Corollary~\ref{cor:iso type of Lambda vs BS(Lambda)} fails since
	the central subgroup $\Lambda=\la b^{2n}\ra$ has the same Bass-Serre graph as the trivial subgroup $\{\id\}$.\end{remark}

\section{Phenotype}\label{sec: phenotype}

In this section, we introduce a central invariant to understand transitive
$\BSo(m,n)$-(pre)-actions: the \emph{phenotype} (see Definition \ref{def: phenotype for actions}). We first define the $(m,n)$-phenotype of a natural number. We then prove that given a transitive  pre-action $(\tau,\beta)$, all cardinalities of $\beta$-orbits have the same phenotype.

\subsection{Phenotypes of natural numbers}

Recall that $\primes$ denotes the set of prime numbers and that 
given $p\in \primes$ and $k\in\Z$, we denote by $\abs{k}_p$ the $p$-adic valuation of $k$.

\begin{definition}[Phenotype of a natural number]
	\label{df: phenotype natural number}
	Let $k\in\Z_{\geq 1}$. We set
	\begin{eqnarray*}
		\primes_{m,n}\phantom{(k)}&\coloneqq &\left\lbrace p\in \primes\colon|m|_p=|n|_p\right\rbrace,\\ 
		\primes_{m,n}(k)&\coloneqq& \left\lbrace p\in\primes\colon|m|_p=|n|_p\text{ and }|k|_p>|n|_p\right\rbrace.
	\end{eqnarray*}
	
	The $(m,n)$-\defin{phenotype} of $k$, denoted by $\Phe_{m,n}(k)$,  is the following positive integer:
	\[\Phe_{m,n}(k)\coloneqq \prod_{p\in \primes_{m,n}(k)} p^{|k|_p}.\]
	If $k=\infty$, we set $\Phe_{m,n}(k)\coloneqq\infty$.
\end{definition}

\begin{example}
	If $m$ and $n$ are coprime, then for every $k\in\Z$ 
	\begin{align*}
		\primes_{m,n}\phantom{(k)}&=\left\lbrace p\in\primes\colon p\text{ does not divide }mn\right\rbrace\\
		\primes_{m,n}(k)&=\left\lbrace p\in\primes\colon p \text{ divides } k\text{ and }p\text{ does not divide }mn\right\rbrace.
	\end{align*}
	In this case, $\Phe_{m,n}(k)$ is the greatest divisor of $k$ that is coprime to $mn$.
\end{example}
\begin{example}\label{ex: phenotype for non coprime}
	If $m=2^2\cdot3^2\cdot5$ and $n=2^2\cdot3$, then $\primes_{m,n}=\primes\smallsetminus\{3,5\}$ and 
	\[\primes_{m,n}(k)=\begin{cases} 
		\{p\in\primes \colon p\text{ divides } k\}\smallsetminus\{2,3,5\}&\text{if } 2^3 \text{ does not divide } k\\ 
		\{p\in\primes \colon p\text{ divides } k\}\smallsetminus \{3,5\}& \text{if } 2^3 \text{ divides } k.
	\end{cases}\]
	For example $\Phe_{m,n}(2\cdot3\cdot7)=7$ and $\Phe_{m,n}(2^5\cdot 3\cdot 7)=2^5\cdot 7$.
\end{example}

\begin{remark}\label{rmk: phen lcm gcd}
	If $k,l$ both have phenotype $q$, then so do their lcm and gcd.
\end{remark}

The following lemma will be useful in Section \ref{Sect: Perfect kernel and pluripotency}.

\begin{lemma}\label{lem: preimage phenotypes}
	Let $q=\Phe_{m,n}(k)$ be a finite $(m,n)$-phenotype. 
	Then $\Phe_{m,n}^{-1}(\{q\})$ is finite if and only if $\abs m=\abs n$.
\end{lemma}
\begin{proof}
	Assume first $\abs{m} \neq \abs{n}$. In this case, there is a prime number $p$ such that $\abs{m}_p \neq \abs{n}_p$. We get $\Phe_{m,n}(p^i k) = q$ for all $i$, hence $\Phe_{m,n}^{-1}(\{q\})$ is infinite.
	
	If $\abs{m} = \abs{n}$, then $\primes_{m,n}=\primes$. 
	If $k$ and $k'$ are two integers with the same phenotype, the only primes $p$ for which the valuations of $k$ and $k'$ may differ are those for which $\abs{k}_p\leq \abs{m}_p$ and in this case $\abs{k'}_p$ must also be bounded by $\abs{m}_p$.
	There are only finitely many such $k'$.
\end{proof}

\subsection{\texorpdfstring{Phenotypes of $(m,n)$-graphs}{Phenotypes of (m,n)-graphs}}

If $v$ is a vertex of an $(m,n)$-graph, we use the shorter expression ``phenotype of the vertex $v$'' to mean ``phenotype of the label of the vertex $v$''. The key feature of the notion of phenotype is the following statement. 

\begin{proposition}\label{prop:phenotype is well defined for conn graph}
	The vertices of a connected $(m,n)$-graph all have the same $(m,n)$-phenotype.
\end{proposition}
\begin{proof}
	It is enough to check that for any positive edge $e$ from $v_-$ to $v_+$, the phenotypes of $v_-$ and $v_+$ are the same. 
	If the phenotype of one of them is infinite, then this is a direct consequence of Equation \eqref{eq:transfert} from Definition~\ref{def:mngraph}.
	Otherwise, remark that for every positive integer $k$ and every $p\in\primes_{m,n}$, \[\left|\frac{k}{\gcd(k,n)}\right|_p>0\ \Leftrightarrow\  p\in\primes_{m,n}(k).\]
	Equation \eqref{eq:transfert} implies 
	\[
	\left|\frac{L(v_-)}{\gcd(L(v_-),n)}\right|_p=|L(e)|_p=\left|\frac{L(v_+)}{\gcd(L(v_+),m)}\right|_p
	\]
	and hence $\primes_{m,n}(L(v_-))=\primes_{m,n}(L(v_+))$. If $p\in\primes_{m,n}(L(v_-))$, then $L(v_-)$ has higher $p$-valuation than $m$ and $n$, so 
	\[
	|L(v_-)|_p-|n|_p=\left|\frac{L(v_-)}{\gcd(L(v_-),n)}\right|_p=\left|\frac{L(v_+)}{\gcd(L(v_+),m)}\right|_p=|L(v_+)|_p-|m|_p.
	\]
	Since $\abs n_p=\abs m_p$, we conclude that $\abs{L(v_-)}_p=\abs{L(v_+)}_p$ for all 
	$p\in\mathcal P_{m,n}(L(v_-))=\mathcal P_{m,n}(L(v_+))$. Therefore $L(v_-)$ and $L(v_+)$ share the same phenotype.
\end{proof}

\begin{remark}\label{rem: edge phenotype is coarser}
	One can prove that the edges of a connected $(m,n)$-graph also all have the same $(m,n)$-phenotype. 
	However, it is a coarser invariant: there are connected graphs with different vertex phenotypes, but with the same edge phenotype.
	For example, fix 
	\[
	m=2^2\cdot3^2\cdot5,\ n=2^2\cdot3
	\]
	and consider the graph consisting of a single oriented edge $e$ and its two endpoints. 
	If the label of its origin is
	$L(\source(e))=2^3\cdot 7$,	then 
	\[
	L(e)=\frac{L(\source(e))}{\gcd(L(\source(e)), n)}=2\cdot 7\text{ and }\Phe_{m,n}(L(e))=7
	\]
	while $\Phe_{m,n}(L(\source(e))=2^3\cdot 7$.
	If instead we set the label of its origin to be
	$L(\source(e))=2^4\cdot 7$, then we get
	\[
	L(e)=2^2\cdot 7\text{ and }\Phe(L(e))=7
	\]
	while $\Phe_{m,n}(L(\source(e))=2^4\cdot 7\neq 2^3\cdot 7$. 
	We will thus not use the phenotype of edges.
\end{remark}

Proposition \ref{prop:phenotype is well defined for conn graph} allows us to define the phenotypes of connected $(m,n)$-graphs and transitive $\BSo(m,n)$-pre-actions.

\begin{definition}
	\label{def: phenotype for m,n-graphs}
	The \textbf{phenotype of a connected $(m,n)$-graph} $\Gc$ is the common phenotype of the labels of its vertices. We denote it $\PHE(\Gc)$.
\end{definition}

\subsection{\texorpdfstring{Phenotypes of $\BSo(m,n)$-actions}{Phenotypes of B\source(mn)-actions}}

Recall that a pre-action is transitive if its Schreier graph is connected, which is equivalent to its Bass-Serre graph being connected.

\begin{definition}
	\label{def: phenotype for actions}
	The \textbf{phenotype of a transitive (pre)-action $\alpha$ of $\BSo(m,n)$} is the common phenotype 
	of the cardinalities $\Phe_{m,n}(\abs{x\la b\ra})$ of its $\la b\ra$-orbits.
	We denote it $\PHE(\alpha)$.
\end{definition}
By definition, the phenotype of any transitive (pre)-action coincides with the phenotype of its Bass-Serre graph. 

\begin{remark} \label{rem: finite BS finite phen}
	Any $\BSo(m,n)$-action with finite Bass-Serre graph and finite phenotype is necessarily an action on a finite set whose cardinality is the sum of the labels of the vertices. 
\end{remark}

For infinite phenotype, we have the following.

\begin{lemma} \label{lem: finite BS with inf phen} There exists an infinite phenotype transitive $\BSo(m,n)$-action with finite Bass-Serre graph if and only if $\abs m=\abs n$. 
\end{lemma}
\begin{proof}
	Consider an infinite phenotype $\BSo(m,n)$-action with finite Bass-Serre graph $\mathcal G$. Since $\mathcal G$ is saturated,
	all its vertices have outgoing degree $\abs{n}$ and incoming degree $\abs{m}$. 
	But there must be globally as many outgoing edges as incoming edges, so since $\mathcal G$ is finite
	we must have $\abs{n}=\abs{m}$.
	
	Conversely if $\abs{n}=\abs{m}$, consider the bouquet of $\abs n$ circles with edges and vertices labeled by $\infty$, 
	and observe that this is a connected saturated $(m,n)$-graph. 
	Proposition \ref{prop: realization of BS graph} 
	provides a transitive action 
	having this labeled bouquet of circles as its finite Bass-Serre graph of infinite phenotype.
\end{proof}

\subsection{Merging pre-actions}
In order to establish some of the main results of this article, we will need ``cut and paste'' operations on pre-actions, for instance:
\begin{itemize}
	\item putting two prescribed pre-actions inside a single transitive action (useful for topological transitivity properties);
	\item modifying an action so as to add or remove a circuit in its Schreier graph (useful to get a new action that is close but distinct from the original one).
\end{itemize}
We now present these ``cut and paste'' operations.
The main one is the following and the rest of this section will be devoted to its proof. Other useful results will appear in the course of the proof.

\begin{theorem}[The merging machine]
	\label{thm: merging of pre-actions}
	Assume $\abs{m}\geq 2$ and $\abs{n}\geq 2$. 
	Let $\alpha_1$ and $\alpha_2$ two transitive non-saturated pre-actions of $\BSo(m,n)$ with the same phenotype. 
	There exists a transitive action $\alpha$ which contains copies of $\alpha_1$ and $\alpha_2$ with disjoint domains.
\end{theorem}
Given a pre-action $\alpha = (\beta,\tau)$ and two sub-pre-actions $\alpha_1,\alpha_2$, 
let us recall that the domain of $\alpha$ is the set $\dom(\beta) = \rng(\beta)$.
Notice that $\alpha_1$ and $\alpha_2$ have disjoint domains if and only if their Bass-Serre graphs $\BSe(\alpha_1)$ and $\BSe(\alpha_2)$ are disjoint (that is, have no common vertex) in $\BSe(\alpha)$.

First, taking advantage of Proposition \ref{Extending pre-action onto a (m,n)-graph}, we reduce to the case of $(m,n)$-graphs, for which the analogous result is the following. 
\begin{theorem}[The merging machine for $(m,n)$-graphs]
	\label{thm: merging of (m,n)-graphs}
	Assume $\abs{m}\geq 2$ and $\abs{n}\geq 2$. 
	Let $\Gc_1$ and $\Gc_2$ be two connected and non-saturated $(m,n)$-graphs with the same phenotype.
	There exists a connected and saturated $(m,n)$-graph $\Gc$ which contains disjoint copies of  $\Gc_1$ and $\Gc_2$. 
\end{theorem}
\begin{remark}
	The hypothesis that both ${\abs m},\abs n\geq 2$ is necessary.
	If $m=1$ but $\abs n\neq 1$, we can consider the $(1,n)$-graph consisting of a single vertex with infinite label and only one loop. This graph is not saturated but it cannot be connected to another copy of itself. Indeed, the reader can check that the only saturated graph containing it admits a unique circuit, namely the loop itself.
\end{remark}
\begin{proof}[Proof of Theorem \ref{thm: merging of pre-actions} based on Theorem \ref{thm: merging of (m,n)-graphs}]
	The Bass-Serre graphs $\BSe(\alpha_1)$ and $\BSe(\alpha_2)$ are connected non-saturated $(m,n)$-graphs with the same phenotype. Therefore we can apply Theorem \ref{thm: merging of (m,n)-graphs} to obtain a connected and saturated $(m,n)$-graph $\Gc$ which contains disjoint copies of $\BSe(\alpha_1)$ and $\BSe(\alpha_2)$.
	
	Then, we apply Proposition~\ref{Extending pre-action onto a (m,n)-graph} to the pre-action $\alpha_1\sqcup \alpha_2$, 
	whose Bass-Serre graph $\BSe(\alpha_1)\sqcup \BSe(\alpha_2)$ is contained in $\Gc$,
	to ensure the existence of a $\BSo(m,n)$-pre-action $\alpha$ which extends $\alpha_1\sqcup \alpha_2$. Thus $\alpha$ extends both $\alpha_1$ and $\alpha_2$ with disjoint domains.
	Since $\Gc$ is connected and saturated, $\alpha$ is a transitive and saturated pre-action, i.e., it is a genuine transitive action of $\BSo(m,n)$ that satisfies the requirements of Theorem~\ref{thm: merging of pre-actions}.
\end{proof}

We now present some general results we will use in order to prove Theorem \ref{thm: merging of (m,n)-graphs}.
We begin with two easy properties of phenotypes which will be useful in the proof.

\begin{lemma}\label{properties of phenotypes}
	For any $k\in \Z_{\geq 1}$, if $q=\Phe_{m,n}(k)$, then $ \Phe_{m,n}(q)=q$ and $\gcd(q,n) = \gcd(q,m)$.
\end{lemma}

\begin{proof}
	We get directly from Definition \ref{df: phenotype natural number} that $\abs{q}_p = \abs{k}_p$ if $p\in \primes_{m,n}(k)$, and $\abs{q}_p = 0$ for the other primes $p$. Consequently, we get $\primes_{m,n}(q) = \primes_{m,n}(k)$ and then $\Phe_{m,n}(q) = \Phe_{m,n}(k) = q$. Finally, since every prime $p$ dividing $q$ satisfies $\abs{m}_p = \abs{n}_p$ and  $\abs{n}_p < \abs{q}_p$, we obtain
	\[
	\gcd(q,n) = \prod_{p\in \primes : \,p|q} p^{\abs{n}_p}
	= \prod_{p\in \primes : \,p|q} p^{\abs{m}_p} = \gcd(q,m).\qedhere
	\]
\end{proof}

In the following lemma, by welding
two vertices we mean taking the quotient graph obtained by identifying these vertices. 
Its proof 
is a direct consequence of the definition of an $(m,n)$-graph, so we omit it.
\begin{lemma}[Welding lemma] 
	\label{lem: self welding lemma}
	Let $m,n\in\Z\smallsetminus\{0\}$ and let $\Gc$ be an $(m,n)$-graph and $v,w$ be two distinct vertices such that:
	\begin{itemize}
		\item $L\coloneqq L(v)=L(w)$;
		\item $\degout (v)+\degout(w)\leq \gcd(n, L)$;
		\item $\degin (v)+\degin(w)\leq \gcd(m, L)$.
	\end{itemize}
	Welding together $v$ and $w$ delivers an $(m,n)$-graph.\qed
\end{lemma}
Note that in this lemma $\Gc$ can be finite or infinite, connected or not.
Together with the welding lemma, the following result will allow us to connect non saturated $(m,n)$-graphs via the well-known technique of arc welding.

\begin{theorem}[Connecting lemma]\label{thm: connecting same phenotype}
	
	Assume $\abs{m}\geq 2$ and $\abs{n} \geq 2$.
	Let $k,\ell \in \Z_{\geq 1}$ such that $\Phe_{m,n}(k) = \Phe_{m,n}(\ell)$, and let $\varepsilon_k,\varepsilon_\ell \in \{+,-\}$. There exists an $(m,n)$-graph $\Gc$ which is a simple edge path $(e_1,\ldots, e_h)$ of length $h\geq 1$ such that:
	\begin{itemize}
		\item $L(\source(e_1))= k$ and $L(\target(e_h))=\ell$;
		\item the orientations of $e_1$ and $e_h$ are given by $e_1 \in E(\Gc)^{\varepsilon_k}$ and $e_h \in E(\Gc)^{\varepsilon_\ell}$.
	\end{itemize}
\end{theorem}

\begin{proof}
	Observe that every $(m,n)$-graph can be turned into an $(n,m)$-graph by flipping the orientations of its edges. 
	Note that this operation does not affect the labels nor its phenotype.
	We thus can restrict ourselves to the case where the orientation $
	\varepsilon_k$ of the first edge in the path is asked to be positive and no assumption is made on $\varepsilon_\ell$. Let us set $q\coloneqq \Phe_{m,n}(k) = \Phe_{m,n}(\ell)$.

	We first treat the case $k=q=\ell$. Recall from Lemma \ref{properties of phenotypes} that $\Phe_{m,n}(q) = q$ and that we have $\gcd(m,q) = \gcd(n,q)$.  
	Hence, there exists an $(m,n)$-graph with two vertices and a unique positive edge $f_1$ 
	such that $L(\source(f_1)) = q = L(\target(f_1))$, and $L(f_1) = \frac{q}{\gcd(m,q)} = \frac{q}{\gcd(n,q)}$. 
	If $\varepsilon_\ell$ is positive, we are done. If not, create a vertex $v$ with label $L(v) = \frac{q}{\gcd(n,q)}m$.
	We get $\gcd(m,L(v)) = \abs{m}$, hence $\gcd(m,L(v)) \geq 2$.
	Therefore, we can equip $v$ with two distinct incoming positive edges $f_1, f_2$.
	Such edges have to be labeled by $\frac{L(v)}{\gcd(m,L(v))} = \frac{q}{\gcd(n,q)}$ so that we can label $\source(f_1)$ and $\source(f_2)$ by $q$, and $(f_1,\bar f_2)$ is the path we are looking for. The theorem is thus 
	proved for $k=\ell=q$.
	
	Let us now treat the case $k\neq q$ and $\ell=q$.
	Recall that 
	$\primes_{m,n}(k)= \left\lbrace p\in\primes\colon|m|_p=|n|_p\text{ and }|n|_p<|k|_p\right\rbrace$ 
	and
	$\Phe_{m,n}(k)= \prod_{p\in \primes_{m,n}(k)} p^{|k|_p}.$
	Thus any number $L\in \Z_{\geq 1}$ with phenotype $q$ admits a unique decomposition as follows:
	
	\begin{equation}\label{eq: general factorization}
		L = q	\cdot \prod_{\substack{p\in \primes\smallsetminus \primes_{m,n}(k)\\ \abs{m}_p \leq \abs{n}_p}} p^{\abs{L}_p} 
		\prod_{\substack{p\in \primes\\ \abs{m}_p > \abs{n}_p}} p^{\abs{L}_p} .
	\end{equation}
	
	In a first step, we construct (algorithmically) a simple path consisting of positive edges
	with vertices $v_0,v_1,\ldots,v_r$, such that $v_0$ has label $k$, and such that the decomposition of $L(v_r)$ reduces to
	\begin{equation}\label{eq: L(vr) after first step}
		L(v_r) = q  \cdot
		\prod_{p\in \primes\colon \abs{m}_p > \abs{n}_p} p^{\abs{L(v_r)}_p},
	\end{equation}
	that is, such that $\abs{L(v_r)}_p = 0$ whenever $\abs{m}_p \leq \abs{n}_p$ and $p\notin \primes_{m,n}(k)$. 
	
	To do so, starting with $i=0$ and $L(v_0)=k$, while $L(v_i)$ has prime divisors $p$ such that $\abs{m}_p \leq \abs{n}_p$ and $p\notin \primes_{m,n}(k)$, 
	we connect $v_i$ to a new vertex $v_{i+1}$ by a positive edge $f_i$. 
	According to Remark \ref{labels of extremities},
	we label $f_i$ by $\abs{L(f_i)}_p \coloneqq \max(\abs{L(v_i)}_p-\abs{n}_p,0)$ 
	and set
	\[
	\abs{L(v_{i+1})}_p \coloneqq
	\begin{cases}
		\abs{L(f_i)}_p + \abs{m}_p & \text{if } \abs{L(f_i)}_p \geq 1 \\
		0 & \text{if } \abs{L(f_i)}_p = 0
	\end{cases}
	\]
	for every prime $p$. Then, we replace $i$ by $i+1$, which terminates the ``while'' loop. Notice that we exit from the loop after finitely many steps. Indeed, given a prime $p$ such that $\abs{m}_p \leq \abs{n}_p$ and $p\notin \primes_{m,n}(k)$, we have:
	\begin{itemize}
		\item either $\abs{L(f_1)}_p=0$ in the case $\abs{m}_p = \abs{n}_p$ and $\abs{k}_p\leq \abs{n}_p$,
		which implies $\abs{L(v_i)}_p=0$ for all $i\geq 1$;
		
		\item or $\abs{L(v_{i+1})}_p = \abs{L(v_i)}_p - \abs{n}_p + \abs{m}_p < \abs{L(v_i)}_p$ whenever $\abs{L(v_i)}_p \geq 1$ in the case $\abs{m}_p < \abs{n}_p$.
	\end{itemize}
	When we exit the ``while'' loop, Remark \ref{labels of extremities} guarantees that we have constructed an $(m,n)$-graph, and the loop condition guarantees that the last vertex $v_r$ satisfies $\abs{L(v_r)}_p = 0$ whenever $\abs{m}_p \leq \abs{n}_p$ and $p\notin \primes_{m,n}(k)$.
	
	If we are lucky, we have $L(v_r) = q$. If not, in a second step, 
	we notice that the same 
	algorithm, exchanging the roles of $m$ and $n$,
	produces a simple path consisting of negative edges
	from a vertex $w_0$ such that $L(w_0) = L(v_r)$ to a vertex $w_s$ labeled by $q$. 
	The decomposition \eqref{eq: L(vr) after first step} of $L(v_r)\neq q$
	also shows that $\gcd(m,L(v_r)) \geq 2$, so vertices labeled $L(v_r)$ can have two distinct positive incoming edges. 
	Using Lemma~\ref{lem: self welding lemma}, we weld
	$v_r$ and $w_0$ together and get a simple path from $v_0$ to $w_s$.

	In any subcase, we now have a path $(e_1, \ldots, e_{h'})$ such that $e_1$ is positive, $L(\source(e_1)) = k$, and $L(\target(e_{h'})) = q$. 
	If $e_{h'}$ has the orientation prescribed by $\varepsilon_\ell$, we are done;
	if not, using the case $k=q=\ell$, with the first edge having the same orientation as $e_{h'}$, and the last one having the orientation prescribed by $\varepsilon_\ell$, 
	we extend our path to a simple path $(e_{1},\ldots,e_h)$ with $L(\source(e_1))=k$ and $L(\target(e_h)) = q$ such that $e_1,e_h$ have the correct orientations. 
	This concludes the case $\ell = q$ and $k\neq q$.
	
	The case $k=q$ and $\ell \neq q$ is obtained by	exchanging the roles of $k$ and $l$ in the above argument. Therefore, let us finally treat the case $k\neq q$ and $\ell\neq q$. The former cases furnish paths $(f_1,\ldots, f_r)$ and $(f'_1,\ldots, f'_s)$, that we may assume disjoint, such that
	\[
	L(\source(f_1)) = k, \quad L(\target(f_r)) = q = L(\source(f'_1)), \quad L(\target(f'_s))= \ell,
	\]
	the orientations of $f_1$ and $f'_s$ are given by $\varepsilon_k$ and $\varepsilon_\ell$, and the orientations $f_r, f'_1$ coincide. Then, we just weld the vertices
	$\target(f_r)$ and $\source(f'_1)$ together, and the path $(f_1,\ldots, f_r,f'_1,\ldots, f'_s)$ is as desired.
\end{proof}

\begin{remark}\label{remark: connecting lemma n=1}
	In Theorem \ref{thm: connecting same phenotype}, the assumption $\abs{m}\geq 2$ and $\abs{n} \geq 2$ is necessary.
	Indeed Theorem \ref{thm: connecting same phenotype} would be false for $n=1$. If $v$ is a vertex in a $(m,1)$-graph with $L(v)=1$ and $e$ is an edge such that $\target(e)=v$, then \[1=L(\target(e))=\frac{L(\target(e))}{\gcd(L(\target(e)),m)}=\frac{L(\source(e))}{\gcd(L(\source(e)),1)}=L(\source(e)).\]
	Clearly any vertex with label $1$ has at most one outgoing and one incoming edge. This implies that the labels of the vertices in any directed path which ends in $v$ must be all $1$. In other words,
	if we have any simple edge path as in Theorem \ref{thm: connecting same phenotype} such that $\ell=1$ and $\varepsilon_\ell=-$, then we must have that $k=1$ (and $\varepsilon_k=+$). 
\end{remark}

\begin{definition}
	\label{def: forest-saturation}
	Let $\Gc$ be a connected $(m,n)$-graph. A saturated extension $\Gc'$ of $\Gc$ is called a \defin{forest-saturation} of $\Gc$ if it 
satisfies
	\begin{itemize}
		\item the subgraph induced in $\Gc'$ by $V(\Gc)$ is exactly $\Gc$;
		\item the subgraph induced in $\Gc'$ by $V(\Gc')\smallsetminus V(\Gc)$ is a forest $\mathcal{F}$;
		\item each connected component of $\mathcal{F}$ is connected to $\Gc$ by a single edge of $\Gc'$.
	\end{itemize}
\end{definition}

\begin{lemma}[Forest-saturation lemma]\label{lem: saturation lemma}
	Let $\Gc$ be a connected $(m,n)$-graph. There is a forest-saturation $\Gc'$ of $\Gc$ such that 
	all vertices of the forest $\mathcal F$ induced in $\Gc'$ by $V(\Gc')\smallsetminus V(\Gc)$ have degree $\geq 1+\min({\abs{m}}, \abs{n})$ in $\Gc'$.
\end{lemma}

The reader can observe in the following construction proving Lemma~\ref{lem: saturation lemma} that, while the labels of the new edges are prescribed, the axioms of $(m,n)$-graphs allows some choices concerning the labels of the new vertices. 
The systematic choice of the maximal label will be made for the new vertices among all those satisfying the transfer equation \eqref{eq:transfert} $\frac{L(\source(e))}{\gcd(L(\source(e)),n)}=L(e)=\frac{L(\target(e))}{\gcd(L(\target(e)),m)}$. Hence the forest-saturation constructed in this proof is called the 
\defin{maximal forest-saturation} of $\Gc$. Notice that other choices would have led to forest-saturations with different underlying graphs, by virtue of the relationship between labels and degrees (see Definition~\ref{def: saturated m,n-graph, degrees as functions of labels}). These forest-saturations 
are further studied in the recent preprint \cite{CGLMS-HT}.

\begin{proof}[Proof of Lemma~\ref{lem: saturation lemma}]
	We can assume that the connected graph $\Gc$ is not yet saturated:  it admits non-saturated vertices i.e., vertices $v$ for which one of the inequalities~\eqref{eq: deg bound by gcd L(v), n m} $\degout(v)\leq\gcd(L(v),n)$ or $\degin(v) \leq \gcd(L(v),m)$
	is strict.
	For every non-saturated vertex $v$ of $\Gc$ we add 
	\begin{itemize}
		\item $(\gcd (L(v), n)-\degout(v))$-many new $v$-outgoing edges labeled $L_{\mathrm{out}}\coloneqq \frac{L(v)}{\gcd(n, L(v))}$ with extra target vertices labeled $m L_{\mathrm{out}}$; and 
		\item $(\gcd (L(v), m)-\degin(v))$-many new $v$-incoming edges labeled $L_{\mathrm{in}}\coloneqq \frac{L(v)}{\gcd(m, L(v))}$ with extra source vertices labeled $n L_{\mathrm{in}}$.
	\end{itemize}
	
	We then iterate this construction. All the non-saturated vertices of the $j$-th step become saturated at the $(j+1)$-th one.
	The increasing union $\Gc'$ of these $(m,n)$-graphs is a saturated $(m,n)$-graph. The complement of $\Gc$ in it is a forest since at each step, each new edge has a new vertex as one of its vertices.
	The label of each new vertex $v$ is an integer multiple of either $m$ or $n$. Thus the degree $\degout(v)+\degin(v)=\gcd(L(v),n)+\gcd(L(v),m)$ of $v$ is larger than 
	$1+\min({\abs{m}}, \abs{n})$ as expected.
\end{proof}

\begin{proof}[Proof of Theorem~\ref{thm: merging of (m,n)-graphs}]
	By hypothesis, for $i=1,2$, there is a non-saturated vertex $v_i$ in $\Gc_i$;
	i.e. a vertex for which one of the inequalities~\eqref{eq: deg bound by gcd L(v), n m} 
	is strict. If $\degin(v_i)< \gcd(L(v_i),m)$, then let $\epsilon_i\coloneqq +$; otherwise let $\epsilon_i\coloneqq-$.
	The labels of $v_1, v_2$ having identical phenotypes, the connecting lemma (Theorem~\ref{thm: connecting same phenotype}) furnishes an
	$(m,n)$-graph $\Gc_0$ which is a simple edge path $(e_1,\ldots, e_h)$ such that
	$L(\source(e_1))= L(v_1)$ and $L(\target(e_h))=L(v_2)$, and the orientations of $e_1$ and $e_h$ are given by $-\epsilon_1$ and $\epsilon_2$ respectively.
	
	We then consider the disjoint union $\Gc_1 \sqcup \Gc_0 \sqcup \Gc_2$. 
	We claim that we can merge the vertices $v_1$ and $\source(e_1)$ thanks to the welding Lemma~\ref{lem: self welding lemma}. 
	Indeed, the choice of orientation for $e_1$ and the form of $\Gc_0$ (a path of edges) are made for the assumptions of Lemma~\ref{lem: self welding lemma} to hold.
	Then, we can merge $v_2$ and $\target(e_h)$, applying Lemma~\ref{lem: self welding lemma} again
	(this time, using the fact that the orientation of $e_h$ is well chosen).
	This produces a connected $(m,n)$-graph $\Gc_3$ which contains disjoint copies of $\Gc_1$ and $\Gc_2$.
	
	It now suffices to apply the saturation Lemma~\ref{lem: saturation lemma} to $\Gc_3$ so as to obtain a connected saturated $(m,n)$-graph $\Gc$
	that satisfies the requirements of Theorem~\ref{thm: merging of (m,n)-graphs}.
\end{proof}

\section{Perfect kernel and dense orbits}
\label{Sect: Perfect kernel and pluripotency}

\subsection{Perfect kernels of Baumslag-Solitar groups}

In case $\abs{m}=1$ or $\abs{n}=1$, it follows from the proof of \cite[Cor.~8.4]{beckerStabilityInvariantRandom2019}
that $\Sub(\BSo(m,n))$ is countable, hence the perfect kernel $\PK(\BSo(m,n))$ is empty.
Our main theorem describes the perfect kernels in the remaining cases. 
\begin{theorem}\label{thm: perfect kernel of Baumslag-Solitar groups}
	Let $m,n\in\Z$ such that $\abs{m}\geq 2$ and $\abs{n}\geq 2$. We have
	\[
	\PK(\BSo(m,n)) = 
	\big\{ 
	\Lambda\in\Sub(\BSo(m,n)): \Lambda \bs \BSo(m,n) / \la b \ra \textrm{ is infinite} 
	\big\}.
	\]
\end{theorem}

Let us temporarily give a name to the set appearing in Theorem \ref{thm: perfect kernel of Baumslag-Solitar groups}:
\[
\Jc = \Jc(m,n) \coloneqq
\big\{ 
\Lambda\in\Sub(\BSo(m,n)): \Lambda \bs \BSo(m,n) / \la b \ra \text{ is infinite}
\big\},
\]
and recall that $\Sub_{[\infty]}(\Gamma)$ denotes the space of infinite index subgroups of $\Gamma$.

Given an action $\alpha$ of $\Gamma$ on a space $X$ and a point $v\in X$, we have already introduced the notation $[\alpha,v]$ for the action $\alpha$ \textit{pointed} at $v$. 

\begin{remark}\label{rem: description L}
	In terms of pointed transitive actions, $\Jc(m,n)$ is the set of pointed transitive actions with infinitely many $b$-orbits, whence
	\(
	\Jc= \big\{[\alpha,v] \colon \BSe(\alpha) \text{ is infinite} \big\}.
	\)
	Moreover:
	\begin{itemize}
		\item if $\abs{m} \neq \abs{n}$, we have
		\(
		\Jc(m,n) = \Sub_{[\infty]}(\BSo(m,n)),
		\)
		since every infinite action has an infinite Bass-Serre graph by Lemma \ref{lem: finite BS with inf phen}.
		
		\item if $\abs{m} = \abs{n}$, we have 
		\(
		\Jc(m,n) = \pi\inv\big( \Sub_{[\infty]}(\BSo(m,n)/ \la b^m \ra) \big),
		\)
		where $\pi$ is the homomorphism from $\BSo(m,n)$ to its quotient by the normal subgroup
		$\la b^m\ra=\la b^n\ra$.
		Indeed, since $\la b^m\ra$ has finite index in $\la b \ra$, we get that $\Lambda \bs \BSo(m,n) / \la b \ra$ is finite if and only if $\Lambda \bs \BSo(m,n) / \la b^m \ra$ is finite.
		
	\end{itemize}
\end{remark}
Therefore, Theorem \ref{thm: perfect kernel of Baumslag-Solitar groups} can be rephrased in two ways, as follows. 
\begin{theorem}
	\label{thm: explicit description K(BS(m,n))}
	Let $m,n\in\Z$ such that $\abs{m}\geq 2$ and $\abs{n}\geq 2$.
	\begin{enumerate}
		\item In terms of pointed transitive actions, the perfect kernel corresponds exactly to actions whose Bass-Serre graph is infinite:
		\[
		\PK(\BSo(m,n)) = \big\{[\alpha,v] \colon \BSe(\alpha) \text{ is infinite} \big\}.
		\]
		
		\item In terms of subgroups:
		\begin{itemize}
			\item if $\abs{m} \neq \abs{n}$,  the perfect kernel is equal to 
			the space of infinite index subgroups:
			\[
			\PK(\BSo(m,n)) = \Sub_{[\infty]}(\BSo(m,n));
			\]
			\item if $\abs{m} = \abs{n}$, we have: 
			\[
			\PK(\BSo(m,n)) = \pi\inv\big( \Sub_{[\infty]}(\BSo(m,n)/ \la b^m \ra) \big),
			\]
			where $\pi$ is the homomorphism from $\BSo(m,n)$ to its quotient by the normal subgroup
			$\la b^m\ra=\la b^n\ra$. \qed
		\end{itemize}
	\end{enumerate}
\end{theorem}

\begin{proof}[Proof of Theorem~\ref{thm: perfect kernel of Baumslag-Solitar groups}]
	Our aim is to prove that $\PK(\BSo(m,n)) = \Jc(m,n)$.
	It will be convenient to write one inclusion in terms of pointed transitive actions and the other in terms of subgroups. 
	
	Let us first prove the inclusion $\PK(\BSo(m,n)) \supseteq \Jc$. 
	It suffices to show that no element of $\Jc$ is isolated in $\Jc$. Recall the definition of the topology in terms of pointed actions, see Section \ref{sec: space subgroups} and in particular Equation \eqref{eq: nbhd basis in terms of Schreier graph}.
	Let us fix a pointed transitive action $(\alpha_0,v)$ representing an element of  $\Jc$
	and a radius $R\geq 0$. We will show that the basic neighborhood $\Nc([\alpha_0,v], R)$ contains at least two distinct
	elements of $\Jc$.
	
	Let $(\beta,\tau)$ be the pre-action obtained by restricting $\alpha_0$ to the
	union of the
	$b$-orbits of the vertices of the ball of radius $R+1$ centered at $v$ in the Schreier graph of $\alpha_0$.
	The Bass-Serre graph of $(\beta,\tau)$ is the projection in $\BSe(\alpha_0)$ (see Definition \ref{dfn: projection BS graph}) of this ball, hence is finite.
	Since $\BSe(\alpha_0)$ is infinite, the pre-action $(\beta,\tau)$ is not saturated. 
	
	We now build two $(m,n)$-graphs $\Gc_1,\Gc_2$
	that extend the finite non-saturated Bass-Serre graph $\Gc$ of $(\beta,\tau)$ in two different ways.
	First, let $\Gc_1$ be a forest-saturation of $\Gc$ given by Lemma~\ref{lem: saturation lemma}.
	In particular, the subgraph induced in $\Gc_1$ by $V(\Gc_1)\smallsetminus V(\Gc)$ is a forest whose vertices have degree at least $3\leq 1+\min(\abs m,\abs n)$ in $\Gc_1$.
	
	We then construct $\Gc_2$ by modifying $\Gc_1$. 
	Let us pick a vertex $v\in V(\Gc_1)\smallsetminus V(\Gc)$.
	The subgraph induced in $\Gc_1$ by $V(\Gc_1)\smallsetminus \{v\}$ has at least $3$ connected components. 
	Choose two connected components disjoint from $\Gc$ and remove them.
	In the resulting $(m,n)$-graph $\Gc_1'$, the vertex $v$ is the only one that is not saturated: two edges are missing.

	Theorem \ref{thm: connecting same phenotype} gives us an $(m,n)$-graph which is a simple edge path $\mathcal P$ whose extremities have the same label as $v$ and for which the orientations of the end edges are compatible with that of the missing edges of $v$.
	We then apply twice the welding lemma, Lemma \ref{lem: self welding lemma}, so as to weld the two extremities of $\mathcal P$ to $v$.
	We eventually define $\mathcal G_2$ to be a forest-saturation of the graph that we obtained. 
	Observe that $\mathcal G_1$ is not isomorphic to $\mathcal G_2$ since the  fundamental groups of their underlying graphs are free groups of distinct ranks. 
	
	Finally, we extend $(\beta,\tau)$ to pre-actions $\alpha_1$ and $\alpha_2$ whose Bass-Serre graphs are $\Gc_1$ and $\Gc_2$ respectively, thanks to Proposition~\ref{Extending pre-action onto a (m,n)-graph}.
	Since $\Gc_1,\Gc_2$ are saturated, $\alpha_1,\alpha_2$ are actually actions by Example~\ref{ex:saturated-> action}.
	We already remarked that $\Gc_1$ is not isomorphic to $\Gc_2$, so the pointed transitive actions $(\alpha_1,v)$ and $(\alpha_2,v)$ are not isomorphic: $[\alpha_1,v]\neq [\alpha_2,v]$. 
	Moreover, the balls of radius $R$ centered at the basepoints in the Schreier graphs of $\alpha_0, \alpha_1, \alpha_2$ all coincide by construction with that of $(\beta,\tau)$, 
	so $[\alpha_1,v]$ and $[\alpha_2,v]$ are both in $\Nc([\alpha_0,v], R)$.
	
	Let us now turn to the inclusion $\PK(\BSo(m,n)) \subseteq \Jc$.
	Let us pick a subgroup $\Lambda \in\Sub(\BSo(m,n)) \smallsetminus \Jc(m,n)$ and let us prove that it is not in the perfect kernel.
	
	If $\abs{m} \neq \abs{n}$, then $\Lambda$ has finite index in $\BSo(m,n)$ by Remark \ref{rem: description L}, hence it is isolated in $\Sub(\BSo(m,n))$.
	
	If $\abs{m} = \abs{n}$, then $\pi(\Lambda)$ has finite index in $\BSo(m,n)/\la b^m \ra$ by Remark \ref{rem: description L}, hence it is finitely generated.
	Therefore, the set 
	\[
	\Vc \coloneqq \{\Lambda'\in \Sub(\BSo(m,n)): \pi(\Lambda') \geq \pi(\Lambda) \}
	\]
	is a neighborhood of $\Lambda$, 
	since it contains the basic neighborhood $\Vc(S,\emptyset)=\{\Lambda'\in\Sub(\BSo(m,n))\colon S\subseteq \Lambda'\}$ where $S\subseteq \Lambda$ is a finite set such that $\pi(S)$ generates $\pi(\Lambda)$.
	
	Now, for any $\Lambda'\in\Vc$, the subgroup $\pi(\Lambda')$ has finite index in $\BSo(m,m)/\la b^m \ra$. Hence $\pi(\Lambda')$ is finitely generated, so $\Lambda'$ itself is finitely generated since it is written as an extension with cyclic kernel:
	\[
	1 \to \la b^m\ra \cap \Lambda'  \to \Lambda' \to \pi(\Lambda') \to 1.
	\]
	Therefore all subgroups of $\Vc$ are finitely generated, which implies that $\Vc$ is countable and hence $\Lambda$ is not in $\PK(\BSo(m,n))$.
\end{proof}

\begin{corollary}\label{cor: m neq n infinite phenotype perfect}
	If $\abs m \geq 2$, $\abs n \geq 2$ and $\abs m\neq\abs n$, then
	\[
	\PHE^{-1}(\infty)\subseteq\PK(\BSo(m,n));
	\]
	in other words, every infinite phenotype subgroup is in the perfect kernel.
\end{corollary}
\begin{proof}
	Any subgroup with infinite phenotype has infinite index and hence it is in $\PK(\BSo(m,n))$ according to Theorem \ref{thm: explicit description K(BS(m,n))}.
\end{proof}

\subsection{Phenotypical decomposition of the perfect kernel}

Let us now turn to a description of the internal structure of $\PK(\BSo(m,n))$.
\begin{notation}
	Let $m,n \in \Z\smallsetminus \{-1,0,1\}$.
	We denote by $\QQ_{m,n}$ the set of all possible $(m,n)$-phenotypes, that is,
	\(
	\QQ_{m,n} \coloneqq \Phe_{m,n}(\Z_{\geq 1}\cup \{\infty\}).
	\)
\end{notation}

\begin{definition}
	The phenotype of a subgroup $\Lambda\leq \BSo(m,n)$ is the $(m,n)$-phenotype of the index of $\Lambda\cap \la b\ra$ in $\la b\ra$
		\[
	\PHE(\Lambda)=\PHE(\Lambda\cap \langle b\rangle)
	\coloneq\Phe_{m,n}\big([\la b \ra : \Lambda\cap \la b\ra]\big).
	\]
	This yields a function
	\(
	\PHE: \Sub(\BSo(m,n)) \to \QQ_{m,n}\subseteq \mathbb Z_{\geq 1}\cup\lbrace\infty\rbrace.
	\)
\end{definition}
	In particular $\PHE(\langle b^k\rangle )=\Phe_{m,n}(k)$ for $k\in\mathbb Z_{\geq 1}$ and the phenotype of the trivial subgroup is infinite.

\begin{remark}
	\label{Rem: Phenotypes des b-sous-groupes}
	The index $[\la b \ra : \Lambda\cap \la b\ra]$ is the cardinal of the $\la b\ra$-orbit of the point $\Lambda$ in the action $\Lambda \bs \BSo(m,n) \curvearrowleft \BSo(m,n) $.
	Hence $\PHE(\Lambda)$ is the phenotype of this action (as given in Definition~\ref{def: phenotype for actions}).
	Since the latter doesn't depend on the basepoint,
	the function $\PHE$ is invariant under conjugation.
\end{remark}

It easily follows form the definitions that if $\PHE(\Lambda)=\PHE(\Lambda')$ then $\PHE(\Lambda)=\PHE(\Lambda\cap \Lambda')$, see Remark \ref{rmk: phen lcm gcd}.

\begin{proposition} \label{prop: phenotype partition}
	In the partition of the space of subgroups of $\BSo(m,n)$ according to their phenotype
	\[\Sub(\BSo(m,n))=\bigsqcup_{q\in \QQ_{m,n}} \PHE^{-1}(q),\]
	the pieces are non-empty and satisfy: 
	\begin{enumerate}
		\item For every finite $q\in \QQ_{m,n}$, the piece $\PHE^{-1}(q)$ is open; it is also closed if and only if $\abs m=\abs n$.
		\item For $q=\infty$, the piece $\PHE^{-1}(\infty)$ is closed and not open.
	\end{enumerate}
\end{proposition}
In particular, the function $\PHE\colon\Sub(\BSo(m,n))\to \Z_{\geq 1}\cup\{+\infty\}$ is Borel. It is continuous if and only if $\abs{m}=\abs{n}$.

\begin{proof}
	Given $k\in\mathbb Z_{\geq 1}$, we set 
	\[A_k\coloneqq \left\{\Lambda\in\Sub(\BSo(m,n))\colon \Lambda \cap \la b \ra= \la b^k\ra\right\}.\]
	Writing $A_k$ as 
	\[A_k=\lbrace\Lambda\in\Sub(\BSo(m,n)\colon b^k\in \Lambda,\ b^i\notin \Lambda\ \text{ for every } 1\leq i<k\rbrace\]
	makes it clear that $A_k$ is clopen for every $k\in \Z_{\geq 1}$. 
	Moreover $\langle b^k\rangle\in A_k$,
	so in particular $A_k$ is not empty. 
	Now, by definition, for every $q\in\Z_{\geq 1}$ we have
	\begin{equation}\label{eq: decomposition as clopen Ak}
		\PHE^{-1}(q)=\bigsqcup_{k\in \Phe_{m,n}\inv(q)}A_k.
	\end{equation} 
	Hence $\PHE^{-1}(q)$ is open for every finite $q$ and, by taking the complement,
	$\PHE^{-1}(\infty)$ is closed.
	
	Take a sequence of positive integers $(k_i)_{i\in\mathbb N}$ tending to $\infty$. 
	Observe that the subgroups $\lbrace\langle b^{k_i}\rangle\rbrace_{i}$ have finite phenotype  and converge to the trivial subgroup which has infinite phenotype. 
	Therefore $\PHE^{-1}(\infty)$ is not open.
	Moreover, if $\Phe_{m,n}^{-1}(q)$ is not finite, we can choose all the $k_i$'s with phenotype $q$; 
	the same argument shows that $\PHE^{-1}(q)$ is not closed.
	Finally, the clopen decomposition \eqref{eq: decomposition as clopen Ak} shows that
	$\PHE^{-1}(q)$ is closed as soon as $\Phe_{m,n}^{-1}(q)$ is finite. By Lemma~\ref{lem: preimage phenotypes}, $\Phe_{m,n}^{-1}(q)$ is finite exactly when $\abs{m}=\abs{n}$.
\end{proof}   

We now restrict the above partition to the perfect kernel
\begin{equation}
	\PK(\BSo(m,n))=\bigsqcup_{q\in \QQ_{m,n}} \PK_q(\BSo(m,n)),
\end{equation}
where
\begin{equation}
	\PK_q(\BSo(m,n))\coloneqq\PK(\BSo(m,n))\cap \PHE^{-1}_{m,n}(q).
\end{equation}

\begin{remark}\label{rem: non-empty pieces}
	Observe that each $\mathcal K_q(\BSo(m,n))$ is not empty: indeed it contains $\la b^q\ra$ which belongs to the perfect kernel by Theorem \ref{thm: perfect kernel of Baumslag-Solitar groups}. 
	Moreover, in the proof of Theorem \ref{thm: perfect kernel of Baumslag-Solitar groups} the $(m,n)$-graphs we construct have the same phenotype, so every element of $\mathcal K_q(\BSo(m,n))$ is actually a non-trivial limit of elements of $\mathcal K_q(\BSo(m,n))$. We conclude that $\mathcal K_q(\BSo(m,n))$ is equal to the perfect kernel of $\PHE_{m,n}\inv(q)$.
\end{remark}

Let us show that the action of $\BSo(m,n)$ by conjugation on each term is \textit{topologically transitive} in the following sense.

\begin{definition}
	An action by homeomorphisms of a group $\Gamma$ on a topological space $X$ is called \textbf{topologically transitive} if for every nonempty open sets $U$ and $V$,
	there is a point whose $\Gamma$-orbit intersects both $U$ and $V$.
\end{definition}

\begin{theorem}\label{thm: gdelta dense orbits}
	Let $m,n$ be integers such that ${\abs m}, \abs n\geq 2$. Then for every phenotype $q\in \QQ_{m,n}$, the action by conjugation of $\BSo(m,n)$ on the invariant subspace $\PK_q(\BSo(m,n))$ is topologically transitive.
\end{theorem}
\begin{proof}
	We again use the definition of the topology in terms of pointed actions, see Section \ref{sec: space subgroups} and in particular Equation \eqref{eq: nbhd basis in terms of Schreier graph}. So let us fix two pointed actions $(\alpha_1,v_1)$ and $(\alpha_2,v_2)$ in $\PK_q(\BSo(m,n))$, take $R>0$, and consider the basic open sets $\mathcal N([\alpha_1,v_1],R)$ and $\mathcal N([\alpha_2,v_2],R)$. We need to construct a pointed action whose orbit meets both open sets. 
	
	As in the proof of Theorem \ref{thm: perfect kernel of Baumslag-Solitar groups}, we let $(\beta_i,\tau_i)$, for $i=1,2$, be the pre-action obtained by restricting $\alpha_i$ to the
	union of the
	$b$-orbits of the vertices of the balls $B(v_i,R+1)$ of radius $R+1$ centered at $v_i$  in the Schreier graph of $\alpha_i$. 
	The Bass-Serre graph of $(\beta_i,\tau_i)$ is finite.
	Since $\BSe(\alpha_i)$ is infinite, the pre-action $(\beta_i,\tau_i)$ is not saturated. 
	
	Moreover $(\beta_1,\tau_1)$ and 
	$(\beta_2,\tau_2)$
	have the same phenotype, so we can apply the merging machine (Theorem \ref{thm: merging of pre-actions}) to obtain an action $\alpha$ whose Schreier graph contains (copies of) the balls $B(v_i,R+1)$. 
	
	Pointing $\alpha$ at the copy of $v_1$ that we denote by $v$, we have $(\Schreier(\alpha),v)\simeq_R(\Schreier(\alpha_1),v_1)$. 
	By transitivity of $\alpha$, there is  $\gamma\in\BSo(m,n)$ such that $v\alpha(\gamma)$ is the copy of $v_2$, and thus $(\Schreier(\alpha), v\alpha(\gamma))\simeq_R(\Schreier(\alpha_2),v_2)$. 
	In particular, the orbit of $[\alpha, v]$ meets both $ \mathcal N([\alpha_1,v_1],R)$ and $ \mathcal N([\alpha_2,v_2],R)$. 
\end{proof}

\begin{corollary}\label{cor: dense conj class}
	Let $m,n$ be integers such that ${\abs m}, \abs n\geq 2$. Then for every $q\in \QQ_{m,n}$, there is a dense $G_\delta$ subset of $\PK_q(\BSo(m,n))$ consisting of subgroups with dense conjugacy class in $\PK_q(\BSo(m,n))$.
\end{corollary}
\begin{proof}[Proof of Corollary~\ref{cor: dense conj class}]
	By Proposition~\ref{prop: phenotype partition}, each $\PK_q(\BSo(m,n))$ is Polish as an open or a closed subset of the Polish space $\PK(\BSo(m,n))$. 
	
	The corollary now follows from a well-known characterization of topological transitivity in Polish spaces: if $(U_i)$ is a countable base of non-empty open subsets, then the set $\cap_{i\in \N} U_i\Gamma$ of points with dense orbit is a dense $G_\delta$ by the Baire theorem.
\end{proof}

\subsection{Closed invariant subsets with a fixed finite phenotype}

Given  a finite phenotype $q$, we will show that there is a largest closed invariant subset inside the (open but non closed when $\abs m\neq \abs n$) set of subgroups of phenotype $q$. The following lemma is key.

\begin{lemma}\label{lem:geod ray to infty}
	Let $\abs m\neq \abs n$, and let $L\in\Z_{\geq 1}$  satisfying: 
	\[
	\exists p\in\mathcal P, \abs m_p\neq\abs n_p\text{ and } \abs L_p>\min(\abs m_p,\abs n_p).
	\]
	Then for any saturated $(m,n)$-graph which contains $L$ as a label, the range of the label map is unbounded.
\end{lemma}

\begin{proof}
	By symmetry, we may as well assume that $\abs n_p<\abs m_p$ for a fixed prime $p$, and so $\abs L_p>\abs n_p$. 
	Let $v_0\in V(\mathcal G)$ have label $L$.
	Since our Bass-Serre graph $\mathcal G$ is saturated, every vertex has at least one outgoing edge.
	We can thus build inductively an infinite path $(e_k)_{k\in\N}$ consisting of positive edges with $\source(e_0)=v_0$. 
	The conclusion then follows directly from Lemma \ref{lem: infinite paths}.	\end{proof}
\begin{remark}
	When $\abs n=\abs m$, the lemma fails because labels are bounded: if $L_0$ is a label then all labels in the same connected component must satisfy $\abs L_p\leq \max(\abs{L_0}_p,\abs m_p,\abs n_p)$ for all primes $p$ by Equation \eqref{eq: transfer for p} and the discussion that precedes it.
\end{remark}

Let $q$ be a finite $(m,n)$-phenotype. In order to describe which saturated $(m,n)$-graphs have unbounded labels, we now define
\begin{equation}
	\label{eq: def of integer s(q,m,n)}
	s(q,m,n)\coloneqq q\cdot \prod_{\substack{ p\in\primes\\  \abs{q}_p=0;\\ \abs{m}_p=\abs{n}_p>0}} p^{|m|_p}\cdot \prod_{\substack{p\in\primes\\\abs m_p\neq\abs n_p}} p^{\min{\{\abs n_p,\abs m_p}\}}.
\end{equation}
\begin{remark}
	The definition is motivated by the fact that $s(q,m,n)$ is the largest label of phenotype $q$ which does not satisfy the hypothesis of Lemma \ref{lem:geod ray to infty}. 
	As we will see in the proof of Theorem \ref{thm:maximalinvariantclosed}, a saturated $(m,n)$-graph with phenotype $q$ has unbounded labels if and only if one of its labels does not divide  $s(q,m,n)$.
\end{remark}

Proposition \ref{prop: phenotype partition} implies that every subgroup (or pointed action) that lies in the closure of the set of subgroups of phenotype $q$ has phenotype either $q$ or $\infty$, and phenotype $\infty$ can occur only when $\abs{m}\neq \abs{n}$.
We can now characterize the subgroups $\Lambda$ with phenotype $q$ whose orbit approaches subgroups with infinite phenotype.

\begin{theorem}\label{thm:maximalinvariantclosed}
	Let $m,n$ be integers such that ${\abs m},\abs n\geq 2$ and denote by $q\in\QQ_{m,n}\smallsetminus\{\infty\}$ a finite $(m,n)$-phenotype. 
	Let $s=s(q,m,n)$ as in Equation \eqref{eq: def of integer s(q,m,n)}. Then the space \[
	\maxclo_q\coloneqq \PHE^{-1}(q)\cap \left\lbrace\Lambda\in\Sub(\BSo(m,n)):\ \Lambda\geq\langle\!\langle b^s\rangle\!\rangle\right\rbrace
	\]
	of subgroups of phenotype $q$ containing the normal subgroup $\la\!\la b^s\ra\!\ra$ satisfies the following properties:
	
	\begin{enumerate}[label=(\arabic*)]
		\item \label{item: maxclo is max}$\maxclo_q$ is the largest closed $\BSo(m,n)$-invariant subset of $\Sub(\BSo(m,n))$ contained in $\PHE^{-1}(q)$; in particular, all normal subgroups of phenotype $q$ and all finite index subgroups of phenotype $q$
		contain $\la\!\la b^s\ra\!\ra$;

		\item \label{item: m=n closed} If $\abs m= \abs n$, then $\maxclo_q = \PHE^{-1}(q)$;
		
		\item \label{item: Lambda going to infinity} For every $\Lambda\in \PHE^{-1}(q)\smallsetminus \maxclo_q$, the orbit of $\Lambda$ accumulates to $\PHE^{-1}(\infty)$;
		
		\item \label{item: maxclo is meager} If $\abs m\neq \abs n$, then  $\maxclo_q\cap \PK_q(\BSo(m,n))$ has empty interior in $\PK_q(\BSo(m,n))$;
		
		\item \label{item: gcp m n 1 maxclo small}  If $\gcd(m,n)=1$, then $s=q$ and $\maxclo_q\cap \PK(\BSo(m,n))=\{\langle\!\langle b^q\rangle\!\rangle\}$; in particular $\langle\!\langle b^q\rangle\!\rangle$ is the unique normal subgroup of phenotype $q$ of infinite index.
	\end{enumerate}
\end{theorem}	
\begin{proof}[Proof of Theorem \ref{thm:maximalinvariantclosed}]
	The proofs of \ref{item: m=n closed} and \ref{item: Lambda going to infinity} rely on the following claim.
	
	\begin{claim}
	     For any $\Lambda\in \PHE\inv(q)\smallsetminus \maxclo_q$, 
	     there is a prime $p$ such that \(\abs{m}_p \neq \abs{n}_p\) 
	     and a vertex label \(L\) in the Bass-Serre graph of \(\Lambda\) such that 
	     \(\abs{L}_p>\abs{s}_p\).
	\end{claim}
	\begin{cproof}
	    Observe that a subgroup $\Lambda$ contains $\langle\!\langle b^s\rangle\!\rangle$ if and only if all the $b$-orbits of the corresponding action $\Lambda\backslash \BSo(m,n)\curvearrowleft\BSo(m,n)$ have cardinality which divides $s$.
	    So if $\Lambda\in \PHE\inv(q)\smallsetminus \maxclo_q$,
	    we can fix a prime \(p\) such that \(\abs{L}_p>\abs{s}_p\), and 
	    we will prove that \(\abs{m}_p \neq \abs{n}_p\).
	    
	    Assume by contradiction that \(\abs{m}_p = \abs{n}_p\).
	    Then $\abs s_p \geq \abs{m}_p=\abs{n}_p$: if \(\abs{m}_p=0\) then the inequality clearly holds, otherwise by Equation \eqref{eq: def of integer s(q,m,n)},
    	\begin{itemize}
    		\item if $p$ divides $q=\Phe_{m,n}(s)$, then $\abs{s}_p = \abs{q}_p > \abs{m}_p = \abs{n}_p$;
    		\item if $p$ does not divide $q=\Phe_{m,n}(s)$, then $\abs{s}_p = \abs{m}_p = \abs{n}_p$.
    	\end{itemize}
    	Thus, we have $\abs L_p > \abs{m}_p=\abs{n}_p$, in other words $p\in\primes_{m,n}(L)$ (see Definition \ref{df: phenotype natural number}).
    	Hence, we have  $\abs{\Phe_{m,n}(L)}_p=\abs{L}_p>\abs s_p\geq\abs{\Phe_{m,n}(s)}_p$. This is a contradiction since both phenotypes are equal to $q$.
	\end{cproof}
	
	We can now easily prove \ref{item: m=n closed} by the contrapositive: by the above claim if $\maxclo_q \neq \PHE^{-1}(q)$ then there is a prime
	\(p\) such that \(\abs{m}_p\neq\abs{n}_p\), in particular \(\abs m\neq\abs n\).

    Let us prove \ref{item: Lambda going to infinity}. 
    Let $\Lambda\in \PHE\inv(q)\smallsetminus \maxclo_q$.
    The claim above provides a prime \(p\) such that \(\abs m_p\neq\abs n_p\) and the Bass-Serre graph of \(\Lambda\) admits a vertex label $L$ such that $\abs{L}_p > \abs{s}_p$.
	It follows from Equation \eqref{eq: def of integer s(q,m,n)} that $\abs s_p=\min(\abs m_p,\abs n_p)$, so $\abs L_p>\min(\abs m_p,\abs n_p)$.
	Lemma \ref{lem:geod ray to infty} thus applies, and so there is a sequence of vertices in the Bass-Serre graph of $\Lambda$ whose labels tend to $+\infty$.
	In other words, there is a sequence $(\gamma_i)_{i\geq 0}$ such that the index of $\gamma_i\Lambda\gamma_i\inv\cap \la b\ra$ in $\la b \ra$ tends to $+\infty$. 
	By compactness, we may assume that this sequence converges, and its limit $\Delta$ cannot contain a non-zero power of $b$ since $[\la b\ra : \gamma_i\Lambda\gamma_i\inv\cap \la b\ra] \to +\infty$. Hence $\Delta$ has infinite phenotype, which proves \ref{item: Lambda going to infinity}.

	We now prove \ref{item: maxclo is max}. We first claim that $\maxclo_q$ is closed in $\Sub(\BSo(m,n))$. Indeed the set
	$$\mathcal{B}_s\coloneqq \left\lbrace\Lambda\in\Sub(\BSo(m,n)):\ \Lambda\geq\langle\!\langle b^s\rangle\!\rangle\right\rbrace$$
	is a countable intersection of basic clopen sets and hence it is closed. 
	Then, notice that $\mathcal{B}_s$ intersects only finitely many sets $\PHE\inv(q')$, since $q'$ must be finite and divide $s$. Proposition \ref{prop: phenotype partition} claims that the $\PHE\inv(q')$ are open, hence 
	\[\maxclo_q = \mathcal{B}_s \smallsetminus \bigcup_{\substack {q'\neq q\\ q' \text{ divides } s}} \PHE\inv(q')  
	\]
	is closed. Also note that $\maxclo_q$ is obviously $\BSo(m,n)$-invariant. 
		Finally Item \ref{item: Lambda going to infinity} implies that every closed $\BSo(m,n)$-invariant subset of $\PHE\inv(q)$ is contained in $\maxclo_q$.
	This proves that $\maxclo_q$ is the largest closed $\BSo(m,n)$-invariant subset of $\Sub(\BSo(m,n))$ contained in $\PHE\inv(q)$. 
	Since all normal subgroups and all finite index subgroups have finite (hence closed) orbits, the remaining statement in Item \ref{item: maxclo is max} is clear.
	
	Let us prove Item \ref{item: maxclo is meager}. Suppose $\abs n\neq\abs m$; 
	let $p$ be a prime number such that $\abs m_p\neq \abs n_p$.
	By definition $\Phe_{m,n}(sp)=\Phe_{m,n}(s)=q$, so $\la b^{sp}\ra\in \PK_q(\BSo(m,n))\smallsetminus \maxclo_q$. 
	Consider a subgroup $\Lambda\in \PK_q(\BSo(m,n))$ whose orbit is dense in $\PK_q(\BSo(m,n))$, as provided by Corollary \ref{cor: dense conj class}.
	Since the orbit of $\Lambda$ accumulates to $\la b^{sp}\ra\notin \maxclo_q$ and $\maxclo_q$ is invariant and closed, the latter does not contain any point of that orbit. Hence the complement 
	$\PK_q(\BSo(m,n))\smallsetminus \maxclo_q$ contains the dense orbit of $\Lambda$. We conclude that $\maxclo_q\cap \PK_q(\BSo(m,n))$ has empty interior in $\PK_q(\BSo(m,n))$.
	
	We finally prove Item \ref{item: gcp m n 1 maxclo small}.
	The equality $s=q$ follows immediately from Formula \eqref{eq: def of integer s(q,m,n)} for $s(q,m,n)$.
	We have the presentation
	\[
	\BSo(m,n)/\la\!\la b^q\ra\!\ra=\la \bar b, \bar t\colon \bar t \bar b^m\bar t^{-1}=\bar b^n, \bar b^q=1 \ra.
	\]
	Since $\gcd(q,m)=\gcd(q,n)=1$, the elements $\bar{b}^m$ and $\bar{b}^n$ both generate $\langle \bar b\rangle$ in $\BSo(m,n)/\la\!\la b^q\ra\!\ra$.
	We thus have a natural semi-direct product decomposition 
	\[
	\BSo(m,n)/\la\!\la b^q\ra\!\ra\cong \Z/q\Z\rtimes\Z=\la \bar{b}\ra \rtimes \la \bar{t}\ra
	\]
	
	Consider $\Lambda\in\maxclo_q$ in the perfect kernel; it contains $\la\!\la b^q\ra\!\ra$.
	It suffices to prove that the image $\Lambda_q\coloneqq \Lambda/\la\!\la b^q\ra\!\ra$ of $\Lambda$ in $\la \bar{b}\ra \rtimes \la \bar{t}\ra$ is trivial.
	Since $\PHE(\Lambda)=q$,
	the index $[\la b \ra: \Lambda \cap \la b \ra]$ is a multiple of $q$,
	so  we have $\Lambda_q\cap\la\bar b\ra=\{\id\}$. 
	Thus $\Lambda_q$ is mapped injectively in the quotient $\la \bar{b}\ra \rtimes \la \bar{t}\ra/\la \bar{b}\ra\simeq \Z$. If this image were not $\{0\}$, then $\Lambda$ would have finite index in $\BSo(m,n)$, contradicting that $\Lambda$ is in the perfect kernel.
	The group $\Lambda_q$ is thus trivial as wanted.
\end{proof}
\begin{remark}
	In terms of actions, $\maxclo_q$ is the set of classes $[\alpha,v]$ all of whose cardinalities of $b$-orbits divide $s$ and have phenotype $q$. 
\end{remark}

\begin{proposition}\label{prop: quotients by b to the k}
	Let $m,n\in \Z\smallsetminus \{0\}$ and $k\in \Z_{\geq 1}$.
	Let $$G_{m,n,k}\coloneqq \BSo(m,n)/\la\!\la b^k\ra\!\ra=\left\langle\bar t, \bar b\mid\ \bar t\bar b^m\bar t^{-1}=\bar b^n, \bar b^k=1 \right\rangle$$ 
	and let \[r(k)\coloneqq \max\{r'\in \N\colon r' \text{ divides } k \text{ and } \gcd(r',m)=\gcd(r',n)\}.\]
	Then:
	\begin{enumerate}
		\item $b$ has order $r(k)$ in the quotient $G_{m,n,k}$; in particular $\langle\!\langle b^k \rangle\!\rangle = \langle\!\langle b^{r(k)} \rangle\!\rangle$;
		
		\item the group $G_{m,n,k}=G_{m,n,r(k)}$
		is the HNN extension of $\mathbb Z/r(k)\mathbb Z=\langle \bar b\rangle$ with respect to the relation $\bar t\bar b^m\bar t^{-1}=\bar b^n$.
		
		\item $\Phe_{m,n}(k) = \Phe_{m,n}(r(k))=\PHE(\langle\!\langle b^k\rangle\!\rangle)$.
	\end{enumerate}
\end{proposition}
\begin{remark}
    It follows from Item 1 in the above proposition that $r(k)=[\la b \ra: \langle\!\langle b^k \rangle\!\rangle\cap \la b\ra ]$.
\end{remark}

It is a routine computation, working prime number by prime number, to check that
\begin{equation}\label{eq: r(k)}
	r(k) = \prod_{\substack{ p\in\primes\\  \abs{m}_p=\abs{n}_p}} p^{\vert k\vert_p} \cdot  \prod_{\substack{ p\in\primes\\  \abs{m}_p\neq \abs{n}_p}} p^{\min(\vert k\vert_p, \vert m\vert_p, \vert n\vert_p)}
\end{equation}
In particular, $r(k)$ is a multiple of all the $r'$s which divide $k$ and satisfy
$\gcd(r',m)=\gcd(r',n)$. Moreover $r(r(k))=r(k)$.

\begin{remark} It also follows from  Item 1 and 3 of the above proposition that
the set of integers $k$ of phenotype $q$ such that $r(k)=k$ parametrizes the normal subgroups of the form $\langle\!\langle b^{k'} \rangle\!\rangle$ of phenotype $q$. Comparing Equations \eqref{eq: r(k)} and \eqref{eq: def of integer s(q,m,n)}, one can check that this is exactly the set of integers $k$ that are multiple of $q$ and that divide $s(q,m,n)$, i.e.\ 
$k=q\cdot j$ where
\begin{enumerate}
\item $\abs{j}_p=0$ for $p\in\primes$ such that $\abs{m}_p=\abs{n}_p=0$;

\item 
$\abs{j}_p\leq \abs{m}_p$
for $p\in\primes$ such that $\abs{m}_p=\abs{n}_p>0$ and $\abs{q}_p=0$; 

\item
$\abs{j}_p\leq \min{\{\abs n_p,\abs m_p}\}$ for $p\in\primes$ such that $\abs m_p\neq\abs n_p$.

\end{enumerate}
\end{remark}

\begin{proof}[Proof of Proposition \ref{prop: quotients by b to the k}]
	Set $r\coloneqq r(k)$. 
	Since $\bar b^m$ and $\bar b^n$ are conjugate in $G_{m,n,k}$, they have the same order: 
	\[
	\frac{\ord(\bar b)}{\gcd(\ord(\bar b),m)}=\ord(\bar b^m)=\ord(\bar b^n)=\frac{\ord(\bar b)}{\gcd(\ord(\bar b),n)}.
	\] 
	Thus $\gcd(\ord(\bar b),m)=\gcd(\ord(\bar b),n)$. Moreover $\ord (\bar b)$ divides $k$. 
	So by the definition of $r$, the order $\ord(\bar b)$ divides $r$ and hence $b^{r}\in \langle\!\langle b^k\rangle\!\rangle$. On the other hand $b^k\in \langle b^r\rangle$, so $\langle\!\langle b^r\rangle\!\rangle=\langle\!\langle b^k\rangle\!\rangle$ and $G_{m,n,k}=G_{m,n,r}$.
	
	Since $\gcd(r,m)=\gcd(r,n)$, the subgroups generated by $\tilde b^m$ and $\tilde b^n$ in the group $\mathbb Z/r\mathbb Z=\langle \tilde b: \tilde b^r=1\rangle$  are isomorphic. 
	We can thus consider the HNN-extension of $\mathbb Z/r\mathbb Z=\langle \tilde b: \tilde b^r=1\rangle$ with the relation $\tilde t \tilde b^m\tilde t^{-1}=\tilde b^n$. It admits the presentation $\langle \tilde t,\tilde b\mid \tilde t \tilde b^m\tilde t^{-1}=\tilde b^n, \tilde b^r=1\rangle$ and it is hence isomorphic to $G_{m,n,r}$.
	
	By the Normal Form Theorem for HNN-extensions, the vertex group injects, i.e., $\bar b$ has order exactly $r$. Finally Formula~\eqref{eq: r(k)} imply that $\Phe_{m,n}(k)=\Phe_{m,n}(r)=\PHE(\langle\!\langle b^r\rangle\!\rangle)$.
\end{proof}

\begin{theorem}\label{th: normal subgroups in finite phenotype}
	Let $m,n\in \Z\smallsetminus \{0\}$ and $q$ be a finite phenotype.
	\begin{enumerate}[label=(\arabic*)]
		\item \label{item: unique normal of phen q}
		If $\gcd(m, n)=1$, then the perfect kernel contains a unique normal subgroup of phenotype $q$, namely $\la\!\la b^q\ra\!\ra$.
		\item If $\gcd(m, n)\neq 1$, then the perfect kernel contains continuum many normal subgroups of phenotype $q$.
	\end{enumerate}
\end{theorem}

\begin{proof}
	The case $\gcd (m,n)=1$ follows from Item \ref{item: gcp m n 1 maxclo small} of Theorem \ref{thm:maximalinvariantclosed}. Therefore let us assume that $\gcd(m, n)\neq 1$.
	
	Consider a prime $p$ which divides both $m$ and $n$. Then either $\abs q_p\neq 0$ and we set $k\coloneqq q$ otherwise set $k\coloneqq qp$. In both cases, remark that $\Phe_{m,n}(k)=q$, that $\gcd(k,m)=\gcd(k,n)$ and hence $r(k)=k$. 
	Then Proposition~\ref{prop: quotients by b to the k} yields that $\bar b$ has order $k$ in $G_{m,n,k}$.
	Furthermore since $k_0\coloneqq \gcd(k,m)=\gcd(k,n)>1$, the elements $\bar b^n$ and $\bar b^m$ are not generators of the subgroup $\la\bar b\ra$: the group $G_{m,n,k}$ is not a semi-direct product. 
	We claim that $G_{m,n,k}$ is not amenable. Indeed, we can write the group $G_{m,n,k}$ as the amalgamated free product
	\[G_{m,n,k}=\langle \bar{t}, \bar{c} \mid \bar t (\bar c)^{\frac{m}{k_0}} \bar t ^{-1}=(\bar c)^{\frac{n}{k_0}}, (\bar c)^{\frac{k}{k_0}}=1\rangle *_{\bar c=\bar b^{k_0}} \langle \bar b\mid  \bar b^{k}\rangle\]
	and one can easily check that $G_{m,n,k}$ admits as a quotient the non-amenable free product $\langle \tilde t \rangle *\langle \tilde b \mid \tilde b^{k_0}\rangle$.
	
	Since $G_{m,n,k}$ is the fundamental group of a finite graph of finite groups, it admits a finite index normal subgroup $F$ which is a finitely generated free group \cite[Prop.\ 11 p.\ 120]{serreTrees1980}. Since $G_{m,n,k}$ is non-amenable, this normal free subgroup is not amenable. 
	
	Every characteristic subgroup $N$ of $F$ is itself normal in $G_{m,n,k}$. 
	Thus the pull-back under the quotient map $\BSo(m,n)\twoheadrightarrow G_{m,n,k}$ is a normal subgroup $\tilde{N}\triangleleft \BSo(m,n)$.
	Since the intersection of $F$ with the finite group $\langle \bar{b}\rangle$ is trivial, the same holds for its characteristic subgroups: $N\cap \langle \bar{b}\rangle=\{\id\}$. Therefore the order of the image of $b$ in $G_{m,n,k}/N=\BSo(m,n)/\tilde{N}$ is the same as in $G_{m,n,k}$, namely $k$.
	In other words, $\tilde N \cap \langle b\rangle=\langle b^k\rangle$. By definition, \[\PHE(\tilde N)=\Phe_{m,n}([\langle b\rangle:\tilde N \cap \langle b\rangle]=\Phe_{m,n}(k)=q.\]
	
	There are continuum many characteristic subgroups $N$ in the finitely generated free subgroup $F$ \cite{bryantCharacteristicSubgroupsFree1974} (see also \cite{bowenCharacteristicRandomSubgroups2017}).
	At most countably many of them lie outside the perfect kernel, so the theorem follows. 
\end{proof}

\section{Limits of finite phenotype subgroups}

In this section, we characterize the subgroups of infinite phenotype of $\BSo(m,n)$
which arise as limits of finite phenotype subgroups.
We will use a version of the straightforward fact that finitely generated subgroups always form a dense set in the space of subgroups. 

\begin{lemma}\label{lem: finitely generated dense on each phenotype}
	Let $m,n\in\Z\smallsetminus \{0\}$. For every phenotype $q\in\QQ_{m,n}$, the finitely generated subgroups of phenotype $q$ are dense in $\PHE^{-1}(q)$.
\end{lemma}
\begin{proof}
	Let $\Lambda$ be a non finitely generated subgroup of phenotype $q$. Let $k\in\Z_{\geq 0}$ such that $\Lambda\cap \la b\ra=\la b^k\ra$. The group $\Lambda$ can be written as the increasing union of finitely generated subgroups all containing $b^k$. They have the same phenotype as $\Lambda$.
\end{proof}

\subsection{Limits of subgroups with fixed finite phenotype}

Recall from Proposition \ref{prop: phenotype partition} that, for $q$ finite, $\PHE\inv(q)$ is open while $\PHE\inv(\infty)$ is closed,
and from Theorem~\ref{thm:maximalinvariantclosed} \ref{item: Lambda going to infinity} that the orbit of any $\Lambda\in \PHE^{-1}(q)\smallsetminus \maxclo_q$ accumulates to $\PHE^{-1}(\infty)$.
We now determine the set of such accumulation points in $\PHE\inv(\infty)$:  this is exactly the set of subgroups contained in the normal closure $\la\!\la b\ra\!\ra$ of $\la b\ra$ but having trivial intersection with $\la b\ra$ itself (since they belong to $\PHE\inv(\infty)$).
\begin{theorem} \label{thm: closure of phen q}
	Suppose $\abs m\neq \abs n$ and let $q$ be a finite phenotype. Then
	\[
	\overline{\PHE\inv(q)}\cap\PHE\inv(\infty)=\{\Lambda\in\PHE^{-1}(\infty)\colon \Lambda\leq \la\!\la b\ra\!\ra\}.
	\]
\end{theorem}

We need two preparatory lemmas. We start with an easy consequence of the defining relation $tb^m=b^nt$ of $\BSo(m,n)$.

\begin{notation}\label{notation: length and height}
	Given $\gamma\in \BSo(m,n)$, let us denote: 
	\begin{itemize}
		\item by $\kappa_\gamma$ the $t$-\textit{length} 
		of $\gamma$, 
		namely the number of occurrences of $t^{\pm 1}$ in the normal form of $\gamma$;
		\item by $\Sigma_\gamma$ the number of occurrences of $t$ minus the number of occurrences of $t\inv$ in the normal form of $\gamma$,
		which is often called the $t$-\textit{height} of $\gamma$.
	\end{itemize}
	Remark that $\Sigma_\gamma$ is the image of $\gamma$ in $\BSo(m,n)/\la\!\la b\ra\!\ra\cong \Z$. 
	In particular $\Sigma_\gamma=0$ if and only if $\gamma\in \la\!\la b\ra\!\ra$.
\end{notation}

\begin{lemma}\label{lem: commutation}
	Fix $\gamma\in \BSo(m,n)$. Let $A\in \Z$ be such that for all primes $p\in\mathcal P$
	\begin{itemize}
		\item if $\abs m_p=\abs n_p$ then $\abs A_p\geq \abs m_p$;
		\item otherwise $\abs A_p\geq \kappa_\gamma \abs m_p$ and $\abs A_p\geq \kappa_\gamma \abs n_p$.
	\end{itemize}
	Then there is $B\in\mathbb Z$, such that $\gamma b^A=b^{B}\gamma$, where $\abs{B}$ is determined by: 
	\[
	\abs B_p=\abs A_p+\Sigma_\gamma (\abs n_p-\abs m_p) \text{ for all }p\in\primes.
	\]
\end{lemma}
\begin{proof}
	This follows from a straightforward induction  on $\kappa_\gamma$ using the relation $tb^m=b^nt$. We leave the details to the reader. 
\end{proof}

The proof of the inclusion in Theorem \ref{thm: closure of phen q} from left to right relies on the following lemma.

\begin{lemma}\label{lem: small power b in lambda}
	Fix $\gamma\not\in \la\!\la b\ra\!\ra$ and let $q$ be a finite phenotype.
	There is an integer $R=R(q,\gamma)$
	such that every subgroup $\Lambda$ of phenotype $q$ containing $\gamma$ must also contain $b^{R}$.
\end{lemma}
\begin{proof}
	Up to replacing $\gamma$ by its inverse, let us assume $\Sigma_\gamma>0$. 
	We define $M\coloneqq\max\{\abs m_p,\abs n_p\colon  p\in\mathcal P\}$,
	and then
	\[
	R\coloneqq q\left(\prod_{\substack{ p\in\primes\\  \abs{m}_p+\abs{n}_p>0}}p\right)^{\kappa_\gamma M}.
	\]
	Fix $\Lambda$ of phenotype $q$. 
	Since $q$ is finite, we have $\la b\ra \cap\Lambda = \la b^N\ra$ with $N>0$.
	We have to show that $N$ divides $R$.
	Notice that $\Phe_{m,n}(N) = q$, thus $N$ decomposes as
	\[
	N=q\cdot p_1^{l_1}\cdots p_k^{l_k}p_{k+1}^{l_{k+1}}\cdots p_r^{l_r},
	\]
	where $r\geq 0$ and $l_1,\ldots,l_r\geq 1$, while the $p_i$ are distinct prime numbers coprime with $q$, see Definition \ref{df: phenotype natural number}. 
	Moreover, we order them so that $p_1,\dots,p_k\in \mathcal P_{m,n}\smallsetminus\mathcal P_{m,n}(N)$ and $p_{k+1},\ldots,p_r\in \mathcal P \smallsetminus \mathcal P_{m,n}$.
	
	Observe that $\abs m_{p_i} = \abs n_{p_i} \geq \abs{N}_{p_i} = l_i \geq 1$ when $p_i\in \mathcal P_{m,n}\smallsetminus\mathcal P_{m,n}(N)$
	and $\abs m_{p_i} \neq \abs n_{p_i}$ when $p_i \in \mathcal P \smallsetminus\mathcal P_{m,n}$.
	Hence, $\abs m_{p_i}+\abs n_{p_i}>0$ for every $i\in\{1,\ldots,r\}$.     
	Consequently, to establish that $N$ divides $R$, it suffices to prove 
	\begin{equation}\label{eq: bound on li}
		\forall i\in\{1,\ldots,r\},\quad l_i \leq \kappa_\gamma M.
	\end{equation}
	
	Observe that $\kappa_\gamma \geq 1$ since $\gamma\notin \la\!\la b \ra\!\ra$.
	For $i\in \{1,\dots,k\}$, Equation \eqref{eq: bound on li} holds 
	since $p_i\in \mathcal P_{m,n}\smallsetminus\mathcal P_{m,n}(N)$, thus
	\[
	l_i\leq \abs m_{p_i}=\abs n_{p_i}\leq M\leq \kappa_\gamma M.
	\]
	
	Let us hence fix $i\in\{k+1,\dots,r\}$ and suppose by contradiction that $l_i > \kappa_\gamma M$. 
	Consider 
	\[
	N'= N \times (p_1\cdots p_k)^M \left( p_{k+1}\cdots \widehat{p_{i}}\cdots p_r \right)^{\kappa_\gamma M}
	\]
	where by $\widehat{p_i}$ we mean that the factor $p_i$ is removed from the product. 
	Clearly $b^{N'}\in \Lambda$ and $\abs{N'}_{p_i}=l_i$. 
	Put 
	\[
	\varepsilon\coloneqq \mathrm{sign}(\abs m_{p_i}-\abs n_{p_i}).
	\]
	Note that $p_i \notin \primes_{m,n}$, hence $\abs{m}_{p_i} \neq \abs{n}_{p_i}$, so $\varepsilon\neq 0$. 
	Since we assumed $\abs{N}_{p_i}=l_i\geq \kappa_\gamma M$, we also have $\abs{N'}_{p_i}\geq \kappa_\gamma M$.
	It is then clear that $N'$ satisfies the assumption of Lemma \ref{lem: commutation},
	so $\gamma^\varepsilon b^{N'}\gamma^{-\varepsilon}=b^{N''}$, where
	\begin{align*}
		\abs{N''}_{p_i}&=l_i+\Sigma_{\gamma^\varepsilon}(\abs n_{p_i}-\abs m_{p_i})
		=l_i+\varepsilon\Sigma_{\gamma}(\abs n_{p_i}-\abs m_{p_i})\\
		&=l_i-\Sigma_\gamma\abs{\abs m_{p_i}-\abs{n}_{p_i}} < l_i.
	\end{align*}
	Clearly $b^{N''}\in \Lambda$, hence $b^{N''}\in \la b^N \ra$. But $\abs{N''}_{p_i}<\abs N_{p_i}$, a contradiction.
	We thus have established Equation \eqref{eq: bound on li}, which finishes the proof.
\end{proof}

\begin{proof}[Proof of Theorem \ref{thm: closure of phen q}]
	
	Set \[\mathcal L\coloneqq\{\Lambda\in\PHE^{-1}(\infty)\colon \Lambda\leq \la\!\la b\ra\!\ra\}.\]
	We first show the inclusion
	\(
	\overline{\PHE\inv(q)}\cap\PHE\inv(\infty)\subseteq\mathcal L.
	\)
	Take $\Delta\in\PHE\inv(\infty)\smallsetminus\mathcal L$ and $\gamma\in\Delta\smallsetminus \la\!\la b\ra\!\ra$. By Lemma \ref{lem: small power b in lambda}, there is an $R$ such that every subgroup $\Lambda$ of phenotype $q$ containing $\gamma$ also contains $b^{R}$.
	Thus the clopen neighborhood of $\Delta$ given by
	\[\mathcal O\coloneqq\{\Lambda\in \Sub(\BSo(m,n))\colon \gamma\in\Lambda,\  b^{R}\not\in \Lambda\}
	\]
	does not intersect $\PHE\inv(q)$. Thus $\Delta$ is not in the closure of $\PHE\inv(q)$.

	We now show the reverse inclusion
	\(
	\mathcal L\subseteq \overline{\PHE\inv(q)}\cap\PHE\inv(\infty).
	\)
	Remark that as in Lemma \ref{lem: finitely generated dense on each phenotype}, the finitely generated elements of $\mathcal L$ are dense in $\mathcal L$: every element of $\mathcal L$ is an increasing union of finitely generated subgroups which have to be in $\mathcal L$ as well.  
	So take $\Lambda=\la S\ra\in \mathcal L$ where $S$ is finite;
	we will show that $\Lambda$ is a limit of subgroups with phenotype $q$. 
	Set $\kappa\coloneqq \max_{\gamma\in S}\kappa_\gamma$, 
	where $\kappa_\gamma$ is the $t$-length of $\gamma$ (see Notation \ref{notation: length and height}).
	Set $M\coloneqq\max\{\abs m_p,\abs n_p\colon  p\in\mathcal P\}$. 
	Note that $\primes\smallsetminus \primes_{m,n}$ is finite, since it is composed of primes $p$ such that $\abs{m}_p+\abs{n}_p>0$, and that $\abs m_p=0$ for all but finitely many primes $p$.
	Hence, for $j\geq 1$, we can define the integer
	\[
	N_j\coloneqq q\cdot\prod_{p\in\mathcal P_{m,n}\smallsetminus \mathcal P_{m,n}(q)}p^{\abs m_p}\cdot
	\prod_{p\in\primes\smallsetminus \primes_{m,n}} p^{j\kappa M}.
	\]
	Observe that $\Phe_{m,n}(N_j)=q$. 
	
	Since $\Lambda\leq \la\!\la b\ra\!\ra$, the height $\Sigma_\gamma$ is zero (see Notation \ref{notation: length and height}) for every $\gamma\in S$, 
	whence, for every $\gamma\in S$ and every $j$, Lemma \ref{lem: commutation} gives $\gamma b^{N_j}=b^{\pm N_j}\gamma$.
	Thus, $\Lambda=\la S\ra$ normalizes $\la b^{N_j}\ra $. Moreover, $\Lambda$ has trivial intersection with $\la b^{N_j}\ra$ because it has infinite phenotype. In particular for $j=1$, we have a natural isomorphism
	\[\Phi\colon \Lambda\ltimes \la b^{N_1}\ra\to  \la\!\Lambda, b^{N_1}\ra.\] 
	Since $N_1$ divides $N_j$, we get 
	\[\Phi(\Lambda\ltimes \la b^{N_j}\ra)=\la\!\Lambda,b^{N_j}\ra.\]
	Observe that $\Phi$ induces a homeomorphism 
	\[ \Sub(\Lambda\ltimes \la b^{N_1}\ra) \to \Sub(\la\!\Lambda, b^{N_1}\ra)\subseteq \Sub(\BSo(m,n)),\] 
	and that the sequence of subgroups $(\Lambda\ltimes \la b^{N_j}\ra)_{j\geq 1}$ converges to $\Lambda\ltimes \{\id\}$. Therefore we have that $\la\!\Lambda, b^{N_j}\ra$ converges to $\Lambda$. 
	Since $\PHE(\la\!\Lambda, b^{N_j}\ra)=\Phe_{m,n}(N_j)=q$, the group $\Lambda$ is the limit of a sequence of elements of phenotype $q$ as wanted.  
\end{proof}

\subsection{Limits of subgroups with varying finite phenotype}

In Theorem \ref{thm: closure of phen q}, we showed that $\overline{\PHE^{-1}(q)}\cap \PHE^{-1}(\infty)$
does not depend on the finite phenotype $q$. We will now consider the closure of all subgroups with finite phenotype and we will first analyse what happens if $\abs m=\abs n$.

\begin{proposition}\label{prop: limit finite phen m=n}
	Let $m,n$ be integers such that $\abs m=\abs n\geq 2$. Then \[\PHE\inv(\infty)\subseteq \overline{\bigcup_{q \text{ finite}}\PHE\inv(q)}.\]
	In other words, every subgroup with infinite phenotype is a limit of subgroups with finite (variable) phenotypes.
\end{proposition}
\begin{proof}
	Let us fix $\Lambda\in\PHE\inv(\infty)$.
	Note that $\la b^n\ra$ is normalized by $\Lambda$ thanks to the relation $t b^n t\inv = b^{\pm n}$.
	We now proceed as in the second part of the proof of Theorem \ref{thm: closure of phen q}: the group $\la\!\Lambda, b^{jn}\ra$ has finite phenotype, it is isomorphic to $\Lambda\ltimes\la b^{jn}\ra$ and the sequence of subgroups $(\la\!\Lambda,b^{jn}\ra)_{j\geq 1}$ converges to $\Lambda$.
\end{proof}

The situation is completely different in the case $\abs m\neq \abs n$.

\begin{proposition}\label{prop: something is not limit}
	Let $m,n$ be integers such that $\abs m\neq \abs n$ and ${\abs m},\abs n\geq 2$. 
	Then 
	\[
	\PHE^{-1}(\infty) \not\subseteq \overline{\bigcup_{q \text{ finite}}\PHE\inv(q)}.
	\]
	In other words, there are subgroups with infinite phenotype that are not limits of subgroups with finite (variable) phenotypes.	
\end{proposition}
Let us recall from Corollary \ref{cor: m neq n infinite phenotype perfect} that $\PHE^{-1}(\infty) = \PK_\infty(\BSo(m,n))$ whenever $\abs m\neq \abs n$.
Hence, the subgroups given by the proposition lie in fact in $\PK_\infty(\BSo(m,n))$.

In the proof of Proposition \ref{prop: something is not limit}, we will need a lemma and a proposition.

\begin{lemma}\label{lemma: k loops finite index}
	Let $m,n$ be integers such that $\abs m\neq \abs n$ and $\abs m$, $\abs n\geq 2$. Let $k\coloneqq \gcd(m,n)$. 
	Let $\Lambda\leq \BSo(m,n)$ be a subgroup containing the following elements
	\[ 
	t,b tb\inv,\ldots,b^{k-1}tb^{-(k-1)}.
	\]
	If $\Lambda$ has finite phenotype, then $\Lambda$ has finite index in $\BSo(m,n)$.
\end{lemma}

\begin{proof}
	Let $\alpha$ be the action $\Lambda\backslash \BSo(m,n) \curvearrowleft \BSo(m,n)$.
	Since the phenotype is finite, it is sufficient to show that the Bass-Serre graph $\BSe(\alpha)$ is finite (see Remark \ref{rem: finite BS finite phen}).
	
	Since $\Lambda$ contains $t$, there is a loop in $\BSe(\alpha)$ at the vertex $v\coloneqq \Lambda\la b\ra$.
	In particular, Equation \eqref{eq:transfert} gives
	\( \frac{L(v)}{\gcd(L(v),m)}=\frac{L(v)}{\gcd(L(v),n)} \).
	As $\Lambda$ has finite phenotype, $L(v)$ is finite, so $\gcd(L(v),m)=\gcd(L(v),n)$ .
	Moreover, as $\BSe(\alpha)$ is a saturated $(m,n)$-graph, we obtain
	\[
	\degin(v) = \gcd(L(v),m)=\gcd(L(v),n) = \degout(v) .
	\]
	This number, that we will denote $d$, is the greatest common divisor of $m$, $n$ and $L(v)$.
	Hence $d$ divides $k=\gcd(m,n)$.
	
	The $d$ outgoing edges at $v$ are exactly  $\Lambda  \la b^n\ra, \Lambda b \la b^n \ra,\dots, \Lambda b^{d-1}\la b^n\ra$.
	As $d\leq k$, the subgroup $\Lambda$ contains $t, btb\inv,\dots,b^{d-1}tb^{-(d-1)}$. Since $\Lambda b^{j}t=(\Lambda b^{j}tb^{-j}) b^{j}=\Lambda b^{j}$, the element
	$t$ fixes all the points $\Lambda, \Lambda b, \ldots, \Lambda b^{d-1} \in \Lambda\backslash \BSo(m,n)$.
	The terminal vertex of the edge $\Lambda b^j \la b^n \ra$ is precisely the vertex $\Lambda b^j t \la b \ra = \Lambda b^j \la b \ra = v$ (see Definition \ref{def: Bass-Serre graph of pre-action}), so all outgoing edges at $v$ are loops.
	
	Since the outgoing degree at $v$ is equal to the incoming degree, all incoming edges at $v$ are loops as well.
	Therefore $\BSe(\alpha)$ consists only of the vertex $v$ and $d$ loops. It is thus finite as wanted.
\end{proof}

\begin{proposition}
	\label{fg infinite phen subgps are not confined}
	Let $m,n$ be integers with ${\abs{m}},\abs n\geq 2$.
	Let $\Lambda$ be a finitely generated subgroup of infinite phenotype and infinite Bass-Serre graph.
	Then there is a sequence of conjugates of $\Lambda$ which converges to $\{\id\}$.
	In particular, $\Lambda$ does not contain any non-trivial normal subgroup of $\BSo(m,n)$.
\end{proposition}
\begin{proof}
	First recall that $\Lambda$ is free. Indeed, having infinite phenotype, it acts freely on the Bass-Serre tree $\Tree$ of $\BSo(m,n)$. 
	Taking the class $\la b\ra$ as a base point in  $\Tree$, the subgroup $\Lambda$ is the fundamental group of the quotient graph $\Lambda\bs \Tree$ based at $\Lambda\la b\ra$.
	This quotient graph is equal to the Bass-Serre graph of $\Lambda$, see Section \ref{sect: Bass-Serre graphs and Bass-Serre theory}, so it is infinite.
	Since moreover 
	$\Lambda$ is finitely generated, it consists of a finite graph to which are attached finitely many infinite trees. 
	Moving the basepoint along one of those infinite trees toward infinity amounts to conjugating $\Lambda$ by a certain sequence of elements $\gamma_i$ of $\BSo(m,n)$ for which we claim that $\gamma_i \Lambda \gamma_i^{-1}\to \{\id\}$. 
	Indeed, each non-trivial element of $\gamma_i \Lambda \gamma_i^{-1}$ is represented by a long path in the tree, followed by a closed path in the finite graph and the long path back to the new basepoint.
	All such elements have a uniformly large $t$-length which tends to $+\infty$ with $i$: their $t$-length is bounded below by twice the $t$-length of $\gamma_i$ minus the diameter of the finite graph. In particular, for any finite set
	$F\subset \Gamma\smallsetminus \{\id\}$ and large enough $i$, all the elements of $\gamma_i \Lambda\gamma_i\inv$ have $t$-length larger than all those of $F$; so $\gamma_i \Lambda \gamma_i^{-1}\cap F=\emptyset$. 
	This proves that $\gamma_i \Lambda \gamma_i^{-1}\to \{\id\}$ as wanted.
\end{proof}

\begin{proof}[Proof of Proposition \ref{prop: something is not limit}]
	Consider the group $\Lambda \coloneqq \la t,btb^{-1},\ldots,b^{k-1}tb^{-(k-1)}\ra$. 
	Observe that by Britton's Lemma (see e.g.~\cite[Chapter IV.2]{lyndonCombinatorialGroupTheory2001}), it is a free group freely generated by $t,btb^{-1},\ldots,b^{k-1}tb^{-(k-1)}$.
	Every non-trivial element of $\Lambda$ contains at least one $t^{\pm1}$ in its normal form, in particular 
	$\Lambda\cap \la b\ra=\{\id\}$: the phenotype of $\Lambda$ is infinite.
	We claim that \[\Lambda\notin \overline{\bigcup_{q \text{ finite}}\PHE\inv(q)}.\]
	
	Suppose that $(\Lambda_i)_{i\geq 0}$ is a sequence of subgroups of finite (variable) phenotypes converging to $\Lambda$. 
	For $i$ large enough, we have $t,btb^{-1},\ldots,$ $b^{k-1}tb^{-(k-1)}\in\Lambda_i$, and thus the subgroup $\Lambda_i$ has finite index by Lemma \ref{lemma: k loops finite index}.
	However, recall that since $\abs m\neq\abs n$, the group $\BSo(m,n)$ is not residually finite \cite{meskinNonresiduallyFiniteOnerelator1972}. Therefore there is a non-trivial normal subgroup $N\trianglelefteq\BSo(m,n)$ contained in every finite index subgroup, and we have $N\leq \Lambda$ since $\Lambda_i \to \Lambda$. 
	This is impossible by Proposition \ref{fg infinite phen subgps are not confined}.
\end{proof}

\begin{corollary}
	\label{cor:accum finite phen in infinite phen has empty interior}
	Let $m,n$ be integers such that $\abs m\neq \abs n$ and ${\abs m},\abs n\geq 2$. Then \[\overline{\bigcup_{q \text{ finite}}\PHE\inv(q)}\cap \PHE^{-1}(\infty)\] has empty interior in $\PHE^{-1}(\infty)$.
\end{corollary}
\begin{proof}
	Recall again that $\PHE^{-1}(\infty) = \PK_\infty(\BSo(m,n))$, see Corollary \ref{cor: m neq n infinite phenotype perfect}.
	In this space, the subset $\PK_\infty(\BSo(m,n))\smallsetminus \overline{\cup_{q \text{ finite}}\PHE\inv(q)}$ is open and Proposition \ref{prop: something is not limit} implies that it is non-empty.
	By Corollary \ref{cor: dense conj class}, this open subset contains a subgroup $\Lambda$ whose orbit is dense in $\PK_\infty(\BSo(m,n))$.
	Therefore $\overline{\cup_{q \text{ finite}}\PHE\inv(q)}$ has empty interior in $\PK_\infty(\BSo(m,n))$.
\end{proof}

\begin{proposition}\label{prop: varying phenotypes vs fixed easy}
	Let $m,n$ be integers such that ${\abs m},\abs n\geq 2$. For any finite phenotype $q_0$, the following inclusion is strict:
	\[\overline{\PHE\inv(q_0)}\cap\PHE\inv(\infty)\subsetneq \overline{\bigcup_{q \text{ finite}}\PHE\inv(q)}\cap\PHE\inv(\infty) .\] 
\end{proposition}

Observe that Proposition~\ref{prop: varying phenotypes vs fixed easy} is trivially true if $\abs m=\abs n$. Indeed, Proposition \ref{prop: limit finite phen m=n} implies that the right hand side is equal to $\PHE^{-1}(\infty)$.
Since Proposition \ref{prop: phenotype partition} yields that $\PHE^{-1}(q_0)$ is closed, the left hand side is empty.

\begin{proof}[Proof of Proposition \ref{prop: varying phenotypes vs fixed easy}]
	For a prime $p$ which divides neither $m$ nor $n$,  define $\Lambda_p\coloneqq \la b^p,t\ra$. 
	Then $\Lambda_p$ clearly has phenotype $p$ (and index $p$ in $\BSo(m,n)$). Let $\Lambda$ be an accumulation point of the sequence $(\Lambda_p)$, then by construction $\Lambda$ has infinite phenotype, so it is in the set $\overline{\bigcup_{q \text{ finite}}\PHE\inv(q)}\cap\PHE\inv(\infty)$.
	However, it contains $t\not\in\la\!\la b\ra\!\ra$ so it is not in $\overline{\PHE\inv(q_0)}$ by Theorem \ref{thm: closure of phen q}.
\end{proof}

\begin{corollary}
	Let $m,n$ be integers such that ${\abs m},\abs n\geq 2$.
	The following inclusion is strict:
	\[\bigcup_{q \text{ finite}}\overline{\PHE\inv(q)}\cap\PHE\inv(\infty)\subsetneq \overline{\bigcup_{q \text{ finite}}\PHE\inv(q)}\cap\PHE\inv(\infty) .\] 
\end{corollary}
\begin{proof}
	If $\abs m=\abs n$, then as already remarked the left hand side is empty. 
	
	If $\abs m\neq \abs n$, recall from Theorem \ref{thm: closure of phen q} that $\overline{\PHE\inv(q)}\cap\PHE\inv(\infty)$
	is independent of $q$. The corollary thus follows from Proposition~\ref{prop: varying phenotypes vs fixed easy}.
\end{proof}

We can also give a statement analogous to Proposition \ref{prop: varying phenotypes vs fixed easy} in the perfect kernel, which is less easy to obtain.

\begin{theorem}\label{thm: varying phenotypes vs fixed}
	Let $m,n$ be integers such that ${\abs m},\abs n\geq 2$. For any finite phenotype $q_0$, the following inclusion is strict:
	\[\overline{\PK_{q_0}(\BSo(m,n))}\cap\PK_\infty(\BSo(m,n))\subsetneq \overline{\bigcup_{q \text{ finite}}\PK_q(\BSo(m,n))}\cap\PK_\infty (\BSo(m,n)).\] 
\end{theorem}
\begin{proof}
	For a fixed prime $p$ which divides neither $m$ nor $n$, let us define a pre-action $(\beta_p,\tau_p)$ as follows. Consider three $\beta_p$-cycles say $o_1$, $o_2$ and $o_3$,
	of cardinals $pn$, $p$ and $pm$ respectively. Then fix basepoints $y_i\in o_i$ for $i=1,2,3$.
	Remark that $o_1$ splits into $\abs n\geq 2$ $\beta_p^n$-orbits of size $p$ and that $o_3$ splits into $\abs m\geq 2$ $\beta_p^m$-orbits of size $p$. Therefore we can define $\tau_p$ by setting
	\[
	y_1\beta_p^{jn}\tau_p \coloneqq y_2\beta_p^{jm}, \quad y_2\beta_p^{jn}\tau_p \coloneqq y_3\beta_p^{jm} \quad \text{and} \quad y_1\beta_p^{-1+jn}\tau_p \coloneqq y_3\beta_p^{1+jm}.
	\]
	Clearly the phenotype of such a pre-action is $p$ and the associated Bass-Serre graph $\Gc_{0,p}\coloneqq \BSe(\beta_p,\tau_p)$ is a triangle. Set $x_p\coloneqq y_1$ and note that for every $p$, we have \[x_p\tau_p\tau_p\beta_p\tau_p^{-1}\beta_p=x_p.\]
	
	By Lemma \ref{lem: saturation lemma}, we can then extend $\Gc_{0,p}$ to a saturated $(m,n)$-graph $\Gc_p$, see Figure \ref{fig: triangle}, and by Proposition \ref{Extending pre-action onto a (m,n)-graph} we can extend the pre-action $(\beta_p,\tau_p)$ to an action $\alpha_p$ whose Bass-Serre graph is $\Gc_p$.
	
	\begin{figure}[H]
		\centering
		\begin{tikzpicture}[scale=0.9]
			\tikzset{vertex/.style = {shape=circle,draw,minimum size=3em}}
			\tikzset{edge/.style = {->,> = latex'}}
			
			\node[vertex, label={$v_1$}] (a) at  (-2.5,0) {$3\cdot p$};
			\node[vertex, label={$v_3$}] (b) at  (2.5,0) {$2\cdot p$};
			\node[vertex, label={$v_2$}] (c) at  (0,2.5) { $p$};	
			\draw[edge] (a) to node[above] {$e_3$} (b);
			\draw[edge] (a) to node[above] {$e_1$} (c);
			\draw[edge] (c) to node[above] {$e_2$} (b);
			
			\node[vertex] (b1) at (6,0) {$4\cdot p$};
			\node[vertex] (a1) at (-6,0) {$9\cdot p$};
			\node[vertex] (a2) at (-5,2.5) {$2\cdot p$};
			\draw[edge] (b) to node {} (b1); 
			\draw[edge] (a1) to node {} (a);
			\draw[edge] (a) to node {} (a2);
			
			\node (b11) at (7.5,0) {};
			\node (b12) at (6,1.5) {};
			\node (a21) at (-7,3) {};
			\node (a22) at (-5,4) {};
			\node (a11) at (-7.5,-1) {};
			\node (a12) at (-7.5,1) {};
			\node (a13) at (-7.5,0) {};
			\draw[edge, dotted] (b1) to node {} (b11);
			\draw[edge, dotted] (b12) to node {} (b1);
			\draw[edge, dotted] (a1) to node {} (a11);
			\draw[edge, dotted] (a1) to node {} (a12);        
			\draw[edge, dotted] (a13) to node {} (a1);
			\draw[edge, dotted] (a2) to node {} (a21);
			\draw[edge, dotted] (a22) to node {} (a2);
		\end{tikzpicture}
		\caption{A (2,3)-graph $\Gc_p$, where $m=2$ and $n=3$.}
		\label{fig: triangle}
	\end{figure}
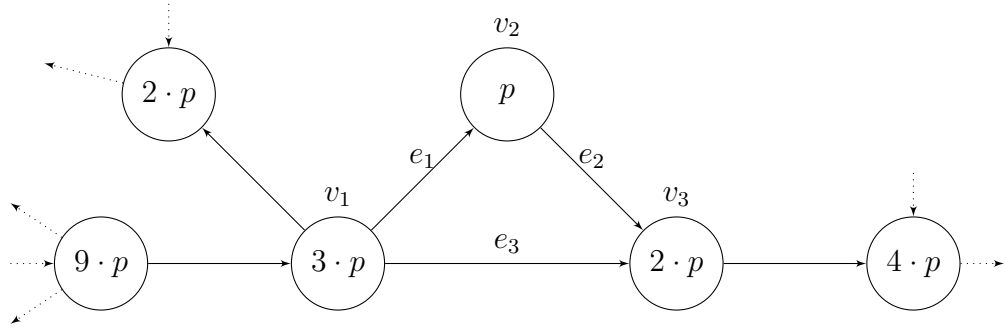    
	
	Define $\Lambda_p$ to be the stabilizer of the action $\alpha_p$ at $x_p$ and remark that $t^2bt^{-1}b\in \Lambda_p$.
	Moreover by construction $\PHE(\Lambda_p)=p$. 
	
	By compactness, we find an accumulation point $\Lambda$ of the sequence $(\Lambda_p)_{p}$. Since $\PHE(\Lambda_p)=p$, the subgroup $\Lambda$ has infinite phenotype. Since $t^2bt^{-1}b\in\Lambda_p$ for every $p$, we have that $t^2bt^{-1}b\in\Lambda$. Moreover $t^2bt^{-1}b\notin \la\!\la b\ra\!\ra$ so $\Lambda\not\in \overline{\PHE\inv(q_0)}$ by Theorem \ref{thm: closure of phen q}. Therefore  the proof is completed. 
\end{proof}

\bibliographystyle{alphaurl}
\bibliography{Biblio}

\bigskip
{\footnotesize
	\noindent
	{A.C., \textsc{Institut für Algebra und Geometrie, Karlsruhe Institute of Technology, 76128 Karlsruhe, Germany}}\par\nopagebreak \texttt{alessandro.carderi@kit.edu}
	
	\medskip 
	
	\noindent
	{D.G., \textsc{CNRS, ENS-Lyon, 
			Unité de Mathématiques Pures et Appliquées,  69007 Lyon, France}}
	\par\nopagebreak \texttt{damien.gaboriau@ens-lyon.fr}
	
	\medskip
	
	\noindent
	{F.L.M., \textsc{Université Paris Cité, Sorbonne Université, CNRS,
			Institut de Mathé\-matiques de Jussieu-Paris Rive Gauche,
			F-75013 Paris, France}}
	\par\nopagebreak \texttt{francois.le-maitre@imj-prg.fr}
	
	\medskip
	
	\noindent
	{Y.S., \textsc{Université Clermont Auvergne, CNRS, LMBP, F-63000 Clermont–Ferrand, France}}
	\par\nopagebreak \texttt{yves.stalder@uca.fr}
}

\end{document}